\documentclass[twoside,12pt]{amsart}

\usepackage[foot]{amsaddr}

\author{Matthew D. Kvalheim}
\address[Kvalheim]{School of Engineering and Applied Science, University of Pennsylvania, Philadelphia, PA 19104, USA}

\author{Anthony M. Bloch}
\address[Bloch]{Department of Mathematics, University of Michigan, Ann Arbor, MI 48109, USA}

\email{kvalheim@seas.upenn.edu, abloch@umich.edu}

\title{Families of periodic orbits: closed 1-forms and global continuability}

\usepackage{import} % Use import instead of input; for pdf_tex in figure folder

\usepackage[utf8]{inputenc}
\usepackage[T1]{fontenc}
\usepackage{lmodern}
\usepackage{amsmath}
\usepackage{amssymb}
\usepackage{amsthm}
\usepackage{stmaryrd} % for minnorm \llfloor, \rrfloor symbols
\usepackage{mathtools}
\usepackage{tikz-cd}
\usepackage{subfig}

\usepackage{thmtools}
\usepackage{thm-restate} % For restating theorem verbatim

\newcommand{\concept}[1]{\textit{#1}}

\newcommand{\N}{\mathbb{N}}
\newcommand{\Z}{\mathbb{Z}}
\newcommand{\R}{\mathbb{R}}
\newcommand{\C}{\mathbb{C}}

\newcommand{\slot}{\,\cdot\,} % Empty argument slot

\newcommand{\T}{\mathsf{T}}
\newcommand{\D}{\mathsf{D}}

\newcommand{\id}{\textnormal{id}}

\newcommand{\interior}{\textnormal{int}}
\newcommand{\cl}{\textnormal{cl}}

\newcommand{\x}{\mathbf{x}}

\newcommand{\vv}{\mathbf{v}}
\newcommand{\1}{\mathbf{1}}
\newcommand{\cf}{\eta}
\newcommand{\dom}{\textnormal{dom}}
\newcommand{\domcf}{\textnormal{dom}(\cf)}
\newcommand{\Hdr}{H_{\textnormal{dR}}}

\newcommand{\cc}{\mathcal{C}}

\newcommand{\dco}{d_{C^1}}
\newcommand{\K}{\mathcal{K}}
\newcommand{\SO}{\mathsf{SO}}
\newcommand{\GL}{\mathsf{GL}}

\DeclarePairedDelimiter\norm{\lVert}{\rVert}

\newtheorem{Lem}{Lemma}
\newtheorem{Th}{Theorem}
\newtheorem{Co}{Corollary}

\newtheorem{Prop}{Proposition}

\newcommand{\thistheoremname}{}
\newtheorem*{genericthm}{\thistheoremname}
{\renewcommand{\thistheoremname}{Theorem~\ref{#1}$'$}%
	\begin{genericthm}}
	{\end{genericthm}}

\theoremstyle{definition}
\newtheorem{Def}{Definition}
\newtheorem*{Def*}{Definition}
\newtheorem{Rem}{Remark}

\usepackage{marvosym}
\usepackage[%hyperref
hypertexnames,%
citecolor=blue,%
colorlinks=true,%
linkcolor=red%
]{hyperref}
\usepackage[all]{hypcap}

\begin{document}
	%\sffamily
	
	\maketitle
	\begin{abstract}	
	We investigate global continuation of periodic orbits of a differential equation depending on a parameter, assuming that a closed 1-form satisfying certain properties exists. We begin by extending the global continuation theory of Alexander, Alligood, Mallet-Paret, Yorke, and others to this situation, formulating a new notion of global continuability and a new global continuation theorem tailored for this situation. In particular, we show that the existence of such a 1-form ensures that local continuability of periodic orbits implies global continuability. Using our general theory, we then develop continuation-based techniques for proving the existence of periodic orbits. In contrast to previous work, a key feature of our results is that existence of periodic orbits can be proven (i) without finding trapping regions for the dynamics and (ii) without establishing a priori upper bounds on the periods of orbits. We illustrate the theory in examples inspired by the synthetic biology literature.	
	\end{abstract}	
	
	\tableofcontents
	
	\section{Introduction}\label{sec:intro}
    In this paper, we study families of periodic orbits of a $C^1$ autonomous ordinary differential equation (ODE) with one parameter 
	\begin{equation}\label{eq:ode-intro}
	\dot{x}=f(x,\mu) \eqqcolon f_\mu(x),\qquad (x,\mu)\in Q\times \R
	\end{equation} 
	on a smooth manifold $Q$.
	Our primary contributions are (i) a theorem on the \concept{global continuation} of periodic orbits as the parameter is varied and (ii) theorems on existence of periodic orbits based on our global continuation theory.   
	A key hypothesis for our theorems is the existence of a closed differential 1-form $\cf$ on $Q\times \R$ satisfying certain properties.
	(Appendix~\ref{app:closed-1-forms} recalls some standard results concerning closed $1$-forms.)
	
	Several authors have previously studied the global continuation of periodic orbits of \eqref{eq:ode-intro}.
	Some important early efforts are represented by \cite{fuller1967index,alexander1978global,chow1978fuller}. 
	These authors study connected components of periodic orbits in $(x,\mu,\tau)$-space, where $\tau$ is the period of a periodic orbit.
	Subsequently several authors showed that more refined information could be obtained by studying components of periodic orbits in $(x,\mu)$-space using other techniques \cite{alligood1981families,mallet1982snakes,chow1983periodic,alligood1983index,alligood1984families}. 
	We mention also \cite{fiedler1988global} who refined and extended many of these global continuation results to families of differential equations which are equivariant under certain groups of symmetries.
	
	The motivation for the present paper was to obtain useful techniques for proving existence of periodic orbits for concrete ODEs.
	In particular, the results in this paper grew out of our attempts to prove existence of periodic orbits for the following ODE
	\begin{equation}\label{eq:sprott-intro}
	\begin{split}
	\dot{x} &= y^2 - z - \mu x\\
	\dot{y} &= z^2 - x - \mu y\\
	\dot{z} &= x^2 - y - \mu z
	\end{split}
	\end{equation}
	on $\R^3$ which depend on the parameter $\mu\in \R$.   
	The system without damping ($\mu = 0$) was considered by Sprott \cite[Eq.~4.7]{sprott2010elegant} as an example of an ``elegant chaotic'' system, so we refer to \eqref{eq:sprott-intro} as the ``Sprott system''; Figure \ref{fig:sprott-traj} displays some of its intrinsically rich dynamical structure.		
   	Our interest in this system was originally inspired by various systems that have been analyzed in the synthetic biology literature such as the \concept{repressilator} and its generalizations, see e.g. 
   	\cite{elowitz2000synthetic, mallet1990poincare, rajapakse2017mathematics, ronquist2017algorithm}. 
   	The repressilator is a model of a synthetic genetic regulatory network consisting of a ring oscillator, and a reduced-order model for this system is given \cite{buse2009existence,buse2010dynamical} by the ODE     \begin{equation}\label{eq:repressilator-intro}
    \begin{split}
    \dot{x} &= \frac{\mu}{1+y^s} - x\\
    \dot{y} &= \frac{\mu}{1+z^s} - y\\
    \dot{z} &= \frac{\mu}{1+x^s} - z
    \end{split}
    \end{equation}	
    on $\R^3$, where $s> 2$ and $\mu>0$ are parameters.    	
    Both \eqref{eq:sprott-intro} and \eqref{eq:repressilator-intro} are symmetric with respect to the cyclic permutation $(x,y,z)\mapsto (y,z,x)$ (see \cite{maria2006homogeneous} for other work on cyclic systems).
    However, in many ways \eqref{eq:sprott-intro} is more subtle to analyze, and many of the standard techniques applied to such systems fail.
    For example, the periodic orbit existence proof for \eqref{eq:repressilator-intro} in \cite{buse2009existence} does not work for \eqref{eq:sprott-intro}; additionally, \eqref{eq:repressilator-intro} has the structure of a \concept{monotone cyclic feedback system} \cite{mallet1990poincare} while \eqref{eq:sprott-intro} does not.    
	Using a single technique based on our results we give proofs that both \eqref{eq:sprott-intro} and \eqref{eq:repressilator-intro} have nonstationary periodic orbits for all $\mu \in (-0.25,0.5)$ and all $\mu \in (\mu_c(s),\infty)$, respectively, where $s> 2$ and $\mu_c(s)$ is a certain parameter value at which a Hopf bifurcation for \eqref{eq:repressilator-intro} occurs.
	While the result for \eqref{eq:repressilator-intro} was already established in \cite{buse2009existence}, our proof is new, and the result we establish for \eqref{eq:sprott-intro} appears to be the first of its kind.	
	
	\begin{figure}
		\centering
		\includegraphics[width=0.32\linewidth]{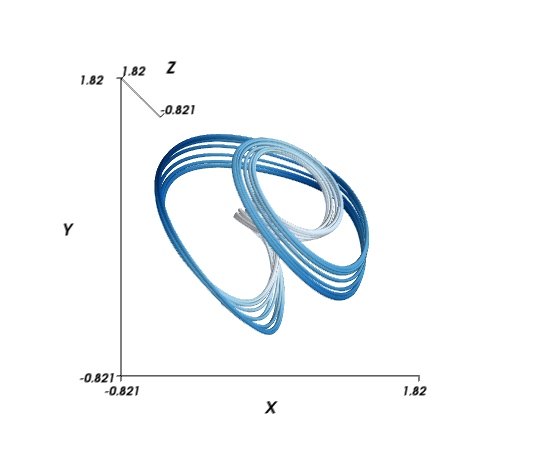}
		\includegraphics[width=0.32\linewidth]{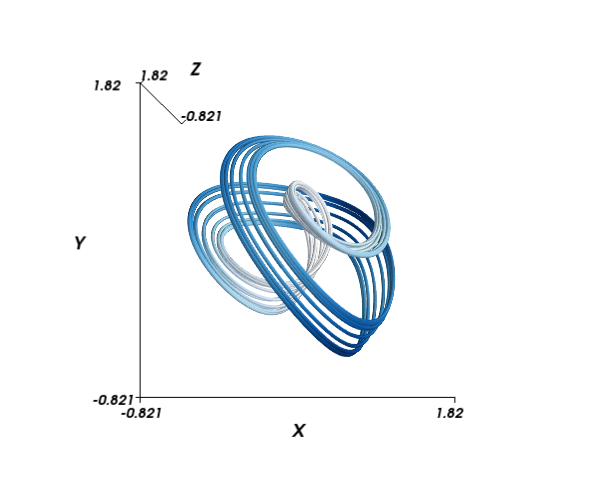}
		\includegraphics[width=0.32\linewidth]{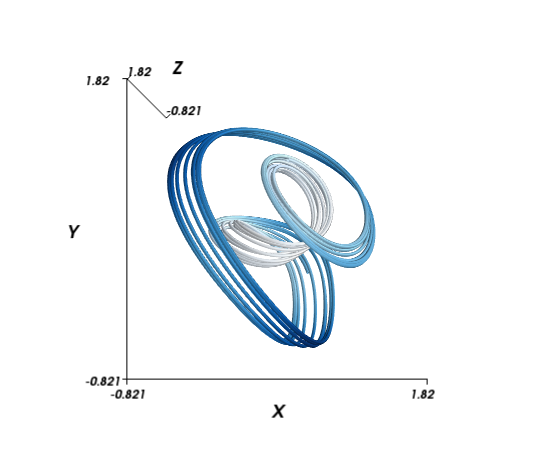}
		\includegraphics[width=0.32\linewidth]{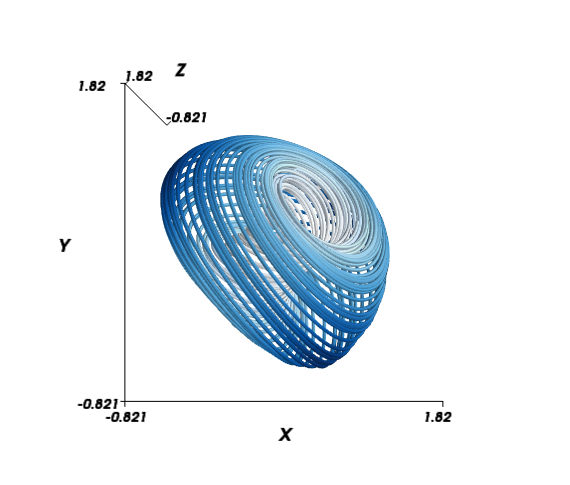}
		\includegraphics[width=0.32\linewidth]{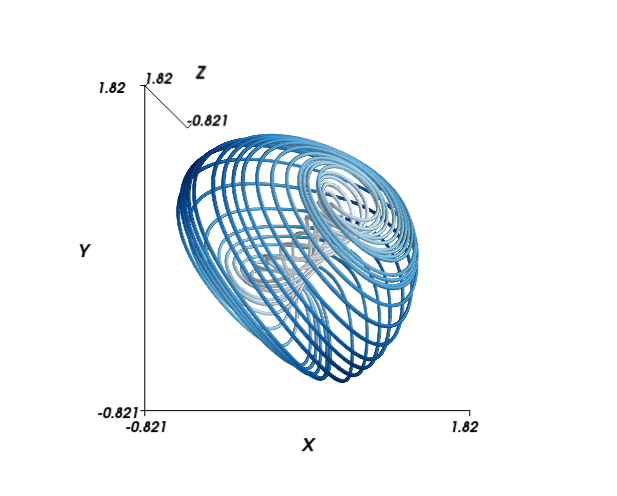}		
		\caption{Shown here are trajectory segments (each of length $150$ time units) of the Sprott system \eqref{eq:sprott-intro} for $\mu = 0$. Each of the top three figures consists of a single trajectory segment, with initial condition $(x_0,y_0,z_0)$ given from left to right by: $(1.2,0.7,0.6)$, $(0.7,0.6,1.2)$, $(0.6,1.2,0.7)$.
		These three trajectory segments are superimposed in the bottom left figure. The bottom right figure consists of a single trajectory segment with initial condition $(x_0,y_0,z_0) = (0.3,0.2,-0.3)$.
		Light portions of trajectory segments indicate where the sum $x+y+z$ is decreasing as a function of time, and dark segments indicate where $x+y+z$ is increasing.
		As an application of our theory, in \S \ref{sec:examples} we prove that this dynamical system has a nonstationary periodic orbit  (see Theorem \ref{th:sprott}).}\label{fig:sprott-traj}	
	\end{figure}

	Perhaps the most famous technique to prove that periodic orbits exist is the Poincar\'{e}-Bendixson theorem \cite{poincare1881memoire,bendixson1901courbes} for autonomous ODEs on the plane.
	More recently, some authors have proven existence theorems for $n$-dimensional ODEs by finding conditions under which an $n$-dimensional system can be projected onto a two-dimensional one so that the Poincar\'{e}-Bendixson theorem can be applied \cite{grasman1977periodic,smith1980existence}.
	Another example of this approach includes a Poincar\'{e}-Bendixson theorem for the class of monotone cyclic feedback systems \cite{mallet1990poincare} which is relevant for various applications in biology; in particular, this theorem yields an alternative proof that the repressilator \eqref{eq:repressilator-intro} has periodic orbits.
	There is also a rich literature on periodic orbit existence for Hamiltonian systems; we mention \cite{rabinowitz1978periodic,weinstein1979hypotheses} as notable examples, and also the solution \cite{conley1983birkhoff,conley1984morse,floer1989symplectic} of the celebrated Arnold conjecture \cite{zehnder1986arnold,zehnder2019beginnings}. 
	For the case of general $n$-dimensional ODEs, the ``torus principle'' \cite{li1981periodic} based on Brouwer's fixed point theorem is widely used to prove existence of periodic orbits; application of this principle is made easier by recent work of Brockett and Byrnes \cite{byrnes2007differential,byrnes2010topological} which utilizes Lyapunov 1-forms \cite{farber2003smooth,farber2004lyapunov,farber2004topology}, results on the topology of Lyapunov function level sets \cite{wilson1967structure}, and various advances in topology including the solution of the Poincar\'{e} conjecture \cite{morgan2007ricci}.	
    The torus principle is generalized by periodic orbit existence theorems based on the Conley and Lefschetz indices, which allow the toroidal trapping region to be replaced with an isolating neighborhood having the Conley index of a hyperbolic periodic orbit \cite{mccord1995zeta,conley1978isolated}; one body of work has focused on rigorous computer-assisted periodic orbit existence proofs based on these topological results \cite{pilarczyk1999computer,baker2005topological}, with applications including the aforementioned class of monotone cyclic feedback systems as well as more general cyclic systems \cite{gedeon1995structure}.

	In this paper, we are interested in proving existence of periodic orbits for families of ODEs depending on a parameter, but the existence results just mentioned are formulated for a single ODE.
	Additionally, applying these existence results is often easier said than done, and we experienced difficulties in applying these results to the Sprott system \eqref{eq:sprott-intro}: for example, we were unable to find ``by hand'' a toroidal trapping region or suitable Conley index pair to prove periodic orbit existence for \eqref{eq:sprott-intro}; equation \eqref{eq:sprott-intro} does not satisfy the ``point-dissipative'' or ``ultimately bounded'' hypothesis of \cite[Thm~4.3]{byrnes2010topological}; and as previously mentioned the Poincar\'{e}-Bendixson theorem for monotone cyclic feedback systems does not apply.
	Inspired by a suggestion of Rajapakse and Smale \cite[p.~1214]{rajapakse2017mathematics}, we set out to find continuation-based techniques to prove existence results---tailored to parametric families of ODEs---which do not require finding trapping regions or index pairs, and which therefore might prove easier to apply to systems such as  \eqref{eq:sprott-intro}.
	We found that one difficulty in using the previously mentioned continuation results \cite{fuller1967index,alexander1978global,mallet1982snakes,chow1983periodic,alligood1983index,alligood1984families,fiedler1988global} to prove existence is that a priori upper bounds on the periods (or \concept{virtual periods}, to be defined in \S \ref{sec:previous-continuation-results}) of periodic orbits of \eqref{eq:ode-intro} are required, and it seems that there are few general techniques to obtain such bounds.	
	However, we show that the existence of a closed 1-form $\cf$ on $Q\times \R$ satisfying certain properties enables a priori period upper bounds to be replaced with conditions such as $\cf((f,0))> 0 $ which are in principle computable.\footnote{Note that $\cf$ satisfying this last condition can be viewed as a Lyapunov 1-form in the sense of \cite{farber2003smooth,farber2004lyapunov,farber2004topology}.}  
	Our first such existence result is Theorem \ref{th:existence}, stated in \S \ref{sec:main-results}.
	Using Theorem \ref{th:existence} we also prove a rather specific existence result in Theorem \ref{th:hopf}, which we use in our applications.
	These theorems are essentially corollaries of our most general result, Theorem \ref{th:main-thm}.
	
	We state our main results in \S \ref{sec:main-results}.
	In order to motivate the statement of our results, in \S \ref{sec:previous-continuation-results} we first discuss in more detail related work of \cite{mallet1982snakes,alexander1983continuability,alligood1984families}.
	In the sequel, for notational simplicity we often identify the image $\Gamma$ of a periodic orbit $\gamma$ of $f_\mu$ with the set $\Gamma \times \{\mu\} \subset Q\times \R$ when there is no risk of confusion.	
	
	\subsection{Discussion of related continuation results}\label{sec:previous-continuation-results}
	Our first main result (Theorem \ref{th:main-thm}) concerns \concept{global continuability}. 
	Multiple notions of global continuability have appeared in the literature; the following definition of \concept{P-global continuability} (called global continuability in \cite{alligood1981families,alexander1983continuability}) is essentially taken from \cite{alligood1983index,alligood1984families}. 
	
	\begin{Def}[P-global continuability]\label{def:p-global-cont}
		Let $A\subset Q \times \R$ be a connected component of nonstationary periodic orbits of \eqref{eq:ode-intro}, and let $\gamma$ be a periodic orbit with image $\Gamma\subset A$.
		We say that $\gamma$ is \concept{P-globally continuable} if at least one of the following holds.
		\begin{itemize}
			\item $A\setminus \Gamma$ is connected,
		\end{itemize}
		or each connected component $A^i$ of $A\setminus \Gamma$ satisfies one of the following:
		\begin{enumerate}
			\item $A^i$ is not contained in any compact subset of $Q\times \R$,
			\item the closure $\cl(A^i)$ of $A^i$ in $Q\times \R$ contains a generalized center (i.e., a stationary point $(x,\mu)$ such that $\D_{x} f_{\mu}$ has some purely imaginary eigenvalues), or 
			\item the periods of orbits in $A^i$ are unbounded.
		\end{enumerate}		
	\end{Def}	
	Mallet-Paret and Yorke considered a certain ``generic'' subset (i.e., containing a residual subset) $\K$ of families \eqref{eq:ode-intro}---discussed in more detail in \S \ref{sec:background-generic-gamilies}---and proved several results involving the continuation of periodic orbits \cite{mallet1982snakes}.
	Necessary for the statement of these results is the concept of a \concept{M\"{o}bius orbit}, which is a periodic orbit having an odd number of Floquet multipliers in $(-\infty,1)$ and no multipliers equal to $-1$.
	The following result is a special case of \cite[Thm~4.2]{mallet1982snakes}; a direct proof appears in \cite[Thm~2.2]{alligood1983index}.
	\begin{Prop}[Mallet-Paret and Yorke]\label{prop:p-global-cont}
    	Let $f\in  \K$ be a generic family of vector fields.
    	Let $\gamma$ be a periodic orbit of some $f_{\mu_0}$.
    	Assume that $\gamma$ is not a M\"{o}bius orbit, and assume that  $\pm 1$ are not Floquet multipliers of $\gamma$. 
    	Then $\gamma$ is P-globally continuable.			
	\end{Prop}
	
	Although the subset $\K$ is generic, given a specific family \eqref{eq:ode-intro} it is usually difficult to determine whether this specific family belongs to $\K$ (c.f. \cite[p.~5]{sander2012connecting}).
	Therefore, it would be desirable to extend Proposition \ref{prop:p-global-cont} to a result valid for arbitrary (i.e., ``non-generic'') $C^1$ families.
	By extending to periodic orbits the notion of virtual periods, previously defined for stationary points of ODEs \cite{mallet1982snakes} and fixed points of maps \cite{chow1983periodic}, Alligood and Yorke introduced a modification of Definition \ref{def:p-global-cont} to prove such a generalization in \cite{alligood1984families}; see \cite{alligood1983index,fiedler1988global} for more general results. 
	Briefly, if $\tau$ is the minimal period of $\gamma$, then $\bar{\tau}=k\tau$ is a \concept{virtual period} of order $k\in \N_{\geq 1}$ for $\gamma$ if the linearization of a Poincar\'{e} map for $\gamma$ has a periodic point of minimal period $k$ \cite{alligood1983index,alligood1984families,fiedler1988global}.
	The following definition is essentially \cite[Def.~1.3]{alligood1984families} and is obtained from Definition \ref{def:p-global-cont} by simply replacing ``periods'' with ``virtual periods''.
	\begin{Def}[Global continuability]\label{def:global-cont}
		Let $A\subset Q \times \R$ be a connected component of nonstationary periodic orbits of \eqref{eq:ode-intro}, and let $\gamma$ be a periodic orbit with image $\Gamma\subset A$.
		We say that $\gamma$ is \concept{globally continuable} if at least one of the following holds.
		\begin{itemize}
			\item $A\setminus \Gamma$ is connected,
		\end{itemize}
		or each connected component $A^i$ of $A\setminus \Gamma$ satisfies one of the following:
		\begin{enumerate}
			\item $A^i$ is not contained in any compact subset of $Q\times \R$,
			\item the closure $\cl(A^i)$ of $A^i$ in $Q\times \R$ contains a generalized center (i.e., a stationary point $(x,\mu)$ such that $\D_{x} f_{\mu}$ has some purely imaginary eigenvalues), or 
			\item the virtual periods of orbits in $A^i$ are unbounded.
		\end{enumerate}		
	\end{Def}	
	The following result is \cite[Thm~3.1]{alligood1984families}; it generalizes Proposition \ref{prop:p-global-cont} to the case of arbitrary $C^1$ families of vector fields.
	\begin{Prop}[Alligood and Yorke]\label{prop:global-cont}
		Let $f\in  C^1(Q\times\R,\T Q)$ be a family of vector fields.
		Let $\gamma$ be a periodic orbit of some $f_{\mu_0}$.
		Assume that $\gamma$ is not a M\"{o}bius orbit, and assume that  $\gamma$ has no Floquet multipliers which are roots of unity. 
		Then $\gamma$ is globally continuable.			
	\end{Prop}	
	
	The assumption that $\gamma$ is not M\"{o}bius in Proposition \ref{prop:p-global-cont} is important: as shown in \cite{alligood1981families}, there are examples of hyperbolic M\"{o}bius orbits $\gamma$ whose components $A\subset Q\times \R$ satisfy none of the conditions of either Definition \ref{def:p-global-cont} or \ref{def:global-cont}.
	In other words, such an orbit $\gamma$ is not globally continuable even if it is locally continuable (via, say, the implicit function theorem applied to a Poincar\'{e} map).
	The reason is related to the possibility that $A$ can contain branches of periodic orbits emanating from a period-doubling bifurcation at one parameter value which annihilate each other at another parameter value.	
	If $\gamma$ is M\"{o}bius, this possibility implies that the ``orbit diagram'' of $A$ (orbit diagrams are discussed in detail in \S \ref{sec:background-generic-gamilies}) can look like that of Figure \ref{fig:lem-mfld-b}, so that $A$ satisfies none of the conditions of Definitions \ref{def:p-global-cont} or \ref{def:global-cont}.
	
	For families of periodic orbits in $\R^3$, however, Alexander and Yorke \cite{alexander1983continuability} showed that, in the presence of a certain additional assumption, M\"{o}bius orbits \emph{are} globally continuable.\footnote{Without this certain additional assumption, a slightly more complicated variant of the orbit diagram in Figure \ref{fig:lem-mfld-b} can still occur; see \cite[Fig.~2.1]{alexander1983continuability}.}
	The basic idea is that, in three dimensions, linking numbers (and also, e.g., knot types) of periodic orbits provide topological obstructions to various bifurcations \cite{ghrist1993knots,ghrist1997knotsBook}, including the phenomenon of orbit annihilation following period-doubling mentioned above.
	This motivates the following basic  observation which generalizes to higher dimensions: the linking number of a periodic orbit with another submanifold of state space also provides an obstruction to the same phenomenon, as long as periodic orbits do not intersect this submanifold (so that the linking number is defined).
	Now one way to compute such a linking number is to integrate a certain closed 1-form over the periodic orbit \cite[pp.~227--234]{bott1982differential}, and in fact the preceding  observation generalizes to yield an obstruction in the situation that one has \emph{any} closed 1-form having nonzero integral over the M\"{o}bius orbit.
	This observation led to the formulations of Definition \ref{def:cf-ell-global-cont} and Theorem \ref{th:main-thm} below and is crucial to the periodic orbit existence Theorems \ref{th:existence} and \ref{th:hopf}.

    \subsection{Main results}\label{sec:main-results}    
    In this section we give statements of our main results.
    In order to state Theorem \ref{th:main-thm}, we first define our own variant of global continuability---\concept{$(\cf,\ell)$-continuability}---which is motivated by the discussion at the end of \S \ref{sec:previous-continuation-results}.
	Definition \ref{def:cf-ell-global-cont} below should be compared with the very similar Definitions \ref{def:p-global-cont} and \ref{def:global-cont} of P-global continuability and global continuability, respectively.
	
	The reader unfamiliar with closed $1$-forms might wish to consult Appendix~\ref{app:closed-1-forms} before proceeding further.
	
	In Definition \ref{def:cf-ell-global-cont} and in the rest of the paper, for each $\mu \in \R$ we let $\iota_\mu\colon Q\hookrightarrow  Q \times \R$ be the inclusion $\iota_\mu(x)=(x,\mu)$ and $\iota_\mu^*\cf$ the pullback of the 1-form $\cf$ on $Q\times \R$ by $\iota_\mu$.
	
	\begin{Def}[$(\cf,\ell)$-global continuability]\label{def:cf-ell-global-cont}
		Let $\ell > 0$, $\cf$ be a $C^1$ closed 1-form on an open subset $\domcf\subset Q\times \R$, and $A\subset \domcf$ be a connected component of nonstationary periodic orbits of $f|_{\domcf}$.
        Define $A_\ell\subset A$ to be the subset of points on periodic orbits $\alpha_\mu$ with $\left|\int_{\alpha_\mu} \iota_\mu^*\cf\right| = \ell$ and $A_{\leq \ell}$ the subset with $\left|\int_{\alpha_\mu} \iota_\mu^*\cf\right| \leq \ell$.

        Let $\gamma$ be a periodic orbit with image $\Gamma\subset A_\ell$.
		Let $\widetilde{A}_{\leq \ell}\subset A_{\leq \ell}$, $\widetilde{A}_{\ell}\subset A_\ell$ be the connected components of $A_{\leq\ell},A_\ell$ containing $\gamma$.        
        We say that $\gamma$ is \concept{$(\cf,\ell)$-globally continuable} if at least one of the following holds.		
		\begin{itemize}
			\item $\widetilde{A}_{\leq \ell} \setminus \Gamma$ is connected,
		\end{itemize}
		or each connected component $\widetilde{A}_{\leq \ell}^i$ of $\widetilde{A}_{\leq \ell}\setminus \Gamma$ containing a connected component $\widetilde{A}_{ \ell}^i$ of $\widetilde{A}_{ \ell}\setminus \Gamma$ satisfies one of the following:
		\begin{enumerate}				
			\item $\widetilde{A}^i_{\ell}$ is not contained in any compact subset of $\domcf$,
				
			\item the closure $\cl(\widetilde{A}^i_{\ell})\subset \domcf$ of $\widetilde{A}^i_{\ell}$ in $\domcf$ contains a generalized center (i.e., a stationary point $(x,\mu)$ such that $\D_{x} f_{\mu}$ has some purely imaginary eigenvalues),

			\item the periods of $\widetilde{A}^i_{\ell}$ are unbounded, or
			
			\item 	$\widetilde{A}^i_{\ell} \neq \widetilde{A}^i_{\leq \ell}$.
		\end{enumerate}		
		
	\end{Def}
	The following theorem is our most general result and should be compared with Propositions \ref{prop:p-global-cont} and \ref{prop:global-cont}.
	
    \begin{restatable}[$(\cf,\ell)$-global continuability for non-generic families]{Th}{ThmContinuability}
    	\label{th:main-thm}   	
    	Let $f\in  C^1(Q\times \R,\T Q)$ be a family of vector fields, and let $\cf$ be a $C^1$ closed 1-form on an open subset $\domcf\subset Q\times \R$.
    	Let $\gamma$ be a periodic orbit of some $f_{\mu_0}$ with image $\Gamma$ satisfying $\Gamma\times \{\mu_0\}\subset \domcf$, and assume that $\gamma$ does not have $+1$ as a Floquet multiplier.
    	Define $\ell\coloneqq \left|\int_{\gamma}\iota_{\mu_0}^*\cf\right|$, and assume $\ell > 0$.		
    	Then $\gamma$ is $(\cf,\ell)$-globally continuable.	
    \end{restatable}
    
    The following theorem is our most general result for proving existence of periodic orbits and is essentially a straightforward corollary of Theorem \ref{th:main-thm}.     
    Given a subset $X\subset Q\times \R$ and any interval $J\subset \R$, in the statement of Theorem \ref{th:existence} we use the notation $X_J\coloneqq X\cap (Q\times J)$.
    
    \begin{restatable}[Global existence of periodic orbits]{Th}{ThmExistence}\label{th:existence}
    
    Assume the hypotheses of Theorem \ref{th:main-thm} and notation of Definition \ref{def:cf-ell-global-cont}.
    Assume that $\widetilde{A}_{\leq\ell}\setminus \Gamma$ is disconnected, let $\widetilde{A}_{\leq\ell}^1$ be one of its connected components, and assume that $\widetilde{A}_{\leq\ell}^1$ is equal to a connected component $\widetilde{A}_{\ell}^1$ of $\widetilde{A}_{\ell}\setminus \Gamma$.
    Further assume that there exists $\cc\subset \domcf \subset Q\times \R$ and $\mu^* < \mu_0$ (resp. $\mu^* > \mu_0$) satisfying the following properties:
    \begin{enumerate}
    	\item\label{item:th-ex-1} $\widetilde{A}_\ell^1\cap (Q \times \{\mu^*\}) = \varnothing$,
    	\item\label{item:th-ex-2} $\widetilde{A}_\ell^1\subset \cc$, 
    	\item\label{item:th-ex-3}  $\iota_\mu^*\cf(f_\mu(x)) > 0 $ for all $(x,\mu)\in \cc_{[\mu^*,\infty)}$ (resp. $(x,\mu)\in \cc_{(-\infty,\mu^*]}$), and
    	\item\label{item:th-ex-4} for every $\mu \geq \mu^*$ (resp. $\mu \leq \mu^*$), $\cc_{[\mu^*,\mu]}$ (resp. $\cc_{[\mu,\mu^*]}$) is compact.

    \end{enumerate}

    Then for all $\mu > \mu_0$ (resp. $\mu<\mu_0$), $\widetilde{A}_{\ell}^1 \cap (Q\times \{\mu\}) \neq \varnothing$.
    In particular, $f_\mu$ has a nonstationary periodic orbit for all $\mu \geq \mu_0$.    
    \end{restatable}
   
    \begin{Rem}\label{rem:classical-result-difficulties}
    	Theorem~\ref{th:existence} is a direct consequence of Theorem~\ref{th:main-thm}.
        Three key points are that the hypotheses of Theorem \ref{th:existence} do not require:
        \begin{itemize}
        	\item verification that the family of vector fields belong to $\K$ as in Proposition~\ref{prop:p-global-cont} (\cite[Thm~4.2,~Thm~2.2]{mallet1982snakes,alligood1983index}),
        	\item any a priori upper bounds on the (virtual) periods of all periodic orbits to be established, as one might hope to do in order to directly apply Propositions \ref{prop:p-global-cont} or \ref{prop:global-cont} \cite[Thm~3.1]{alligood1984families}, or
        	\item a trapping region to be found, as the hypotheses only require that periodic orbits in $\widetilde{A}^1_\ell$ do not meet the boundary of $\cc$.
        \end{itemize}
        We have found that, in particular, the second requirement represents serious difficulties in proving that periodic orbits exist (on large parameter intervals) for both the repressilator \eqref{eq:repressilator-intro} and Sprott \eqref{eq:sprott-intro} systems using the classical Propositions~\ref{prop:p-global-cont} and \ref{prop:global-cont}.
        To establish a priori upper bounds on the (virtual) periods as in the second requirement, one would need to somehow rule out, e.g., the possibility of an infinite cascade of period-doubling bifurcations.
        Theorems~\ref{th:main-thm} and \ref{th:existence} sidestep this difficulty by studying only $\widetilde{A}_{\leq\ell}^1$ rather than the larger $A^1\supset \widetilde{A}_{\leq \ell}^1$ of Definitions~\ref{def:p-global-cont} and \ref{def:global-cont} which, as will follow from Lemma~\ref{lem:mobius-c1f-integrals}, effectively ignores periodic orbits in $A^1$ resulting from period-doubling bifurcations.
        This is illustrated in Figure~\ref{fig:lem-mfld-a} for the case of a generic family of ODEs; our methods amount to studying periodic orbits represented by points in the thick curve while systematically ignoring periodic orbits represented by points in the thin curves resulting from period-doubling bifurcations (period-doubling bifurcations are represented by points where branching occurs in the figure).
    \end{Rem}

    The following result is proven using Theorem \ref{th:existence}.
    Although its statement appears rather specific and complicated, Theorem \ref{th:hopf} represents the formalization of a common argument we have used to apply Theorem \ref{th:existence} in multiple concrete examples.
    Specifically, in \S \ref{sec:examples} we use Theorem \ref{th:hopf} to prove global existence results for periodic orbits in both the Sprott system \eqref{eq:sprott-intro} and repressilator \eqref{eq:repressilator-intro}.

    Given a subset $X\subset Q\times \R$ and any interval $J\subset \R$, we again use the notation $X_J\coloneqq X\cap (Q\times J)$ in Theorem \ref{th:hopf}.   
    By a point $(x,\mu)\in Q\times \R$ of generic Hopf bifurcation, we mean a point satisfying the hypotheses of the standard Hopf bifurcation theorem (see \cite{guckenheimer1983nonlinear, ruelle1989elements, robinson1999ds, kuznetsov2013elements}).
    In Theorem \ref{th:hopf} we refer to the first de Rham cohomology $\Hdr^1(\slot)$ \cite[Ch.~17]{lee2013smooth} and to the Poincar\'{e} dual of a submanifold \cite[pp.~50--53, p.~69]{bott1982differential}; see Appendix~\ref{app:closed-1-forms} for a brief discussion.
    The de Rham version of Poincar\'{e} duality is  typically discussed in the setting of $C^\infty$ closed forms, but (as noted in Appendix~\ref{app:closed-1-forms}) every $C^1$ closed form is cohomologous to a $C^\infty$ closed form \cite[pp.~61--70]{deRham1984differentiable}, so   
    no distinctions need to be made.

    \begin{restatable}[Global existence of periodic orbits following a Hopf bifurcation]{Th}{ThHopf}\label{th:hopf}
    Assume that $Q$ is oriented, and let $N\subset Q\times \R$ be a properly embedded, smooth, oriented, codimension-1 submanifold with boundary $M=\partial N$. 
    Let $f\in  C^1(Q\times \R,\T Q)$ be a family of vector fields, and let $\cf\in [\cf]\in \Hdr^1((Q\times \R)\setminus M)$ be a $C^1$ closed 1-form representing the (closed) Poincar\'{e} dual $[\cf]$ of $N\setminus M$.
    Further assume that there exists $\cc\subset Q\times \R$, $(x_c,\mu_c)\in M\cap \interior(\cc)$ and $\mu^* < \mu_c$ (resp. $\mu^* > \mu_c$) satisfying the following properties:
    \begin{enumerate}
    	\item\label{item:th-hopf-1} $f_{\mu^*}$ has no periodic orbits contained in $\cc_{\{\mu^*\}}$, 
    	\item\label{item:th-hopf-2} no periodic orbits of $f$ intersect $(\partial \cc)_{[\mu^*,\infty)}$ (resp. $\partial \cc_{(-\infty,\mu^*]})$, 
    	\item\label{item:th-hopf-3}  For every $\mu_1 > \mu^*$ (resp. $\mu_1 < \mu^*$), there exists $\epsilon > 0$ such that $\iota_\mu^*\cf(f_\mu(x)) \geq \epsilon$ for all $(x,\mu)\in (\cc \setminus M)_{[\mu^*,\mu_1]}$ (resp. $(x,\mu)\in (\cc \setminus M)_{[\mu_1,\mu^*]}$) ,
    	\item\label{item:th-hopf-4} for every $\mu \geq \mu^*$ (resp. $\mu \leq \mu^*$), $\cc_{[\mu^*,\mu]}$ (resp. $\cc_{[\mu,\mu^*]}$) is compact, 
    	\item\label{item:th-hopf-5}  $f$ is $C^3$ on a neighborhood of $(x_c,\mu_c)$,  $(x_c,\mu_c)\in M \cap \interior(\cc)$ is a point of generic Hopf bifurcation for $f$, and $\cc_{[\mu^*,\infty)}$ (resp. $\cc_{(-\infty,\mu^*]}$) contains no other generalized centers, 
    	\item\label{item:th-hopf-6} no nonstationary periodic orbits of $f$ intersect $(\cc\cap M)_{[\mu^*,\infty)}$ (resp. $(\cc\cap M)_{(-\infty,\mu^*]}$), and
    	\item\label{item:th-hopf-7} letting $E^c\subset \T_{x_c}Q$ be the two-dimensional center subspace for $\D_{x_c}f_{\mu_c}$,
    	$$\T_{(x_c,\mu_c)}(Q\times \R) = (\D_{(x_c,\mu_c)}\iota_\mu E^c) \oplus \T_{(x_c,\mu_c)}M.$$
    \end{enumerate}
    
    Then for all $\mu > \mu_c$ (resp. $\mu < \mu_c$), $f_\mu$ has a nonstationary periodic orbit contained in $(\cc\setminus M)_{\{\mu\}}$.     
    \end{restatable}
    
    \subsection{Outline of the sequel}
    The remainder of the paper is organized as follows.
    
    In \S \ref{sec:generic-family-results} we develop the theory for the  so-called generic families of vector fields, i.e., those families belonging to a certain generic subset $\K\subset C^5(Q\times \R,\T Q)$ of the $C^5$ one-parameter families.
    We begin in \S \ref{sec:background-generic-gamilies} by discussing $\K$ and giving the relevant background on periodic orbits for $f\in \K$.
    Along the way we introduce \concept{orbit diagrams}, which are very useful in the generic setting.
    Section \ref{sec:generic-family-continuation} introduces some of the key ideas and proves Theorem \ref{th:main-thm} in the special case that the vector field family is generic (Lemma \ref{lem:cf-ell-cont-generic-fam}).
    
    In \S \ref{sec:non-generic-families} we extend the results of \S \ref{sec:generic-family-continuation} to prove Theorem \ref{th:main-thm} for the general case of an arbitrary family $f\in C^1(Q\times \R,\T Q)$.
    The proof is by generic approximation and was inspired by techniques of \cite{alligood1984families}.
    As a straightforward corollary of Theorem \ref{th:main-thm} we obtain Theorem \ref{th:existence}, which is a fairly general theorem for proving existence of periodic orbits.
    We then record as Theorem \ref{th:hopf} a systematic argument involving Theorem \ref{th:existence} for proving existence of periodic orbits on large parameter intervals following a Hopf bifurcation, in a setting which appears common in certain applications.
    
    In \S \ref{sec:examples} we illustrate the utility of our results in some specific ODEs.
    In \S \ref{sec:repressilator} we give a periodic orbit existence proof for the repressilator \eqref{eq:repressilator-intro}. 
    Our proof is distinct from the proof of \cite{buse2009existence} and does not use techniques of monotone systems \cite{mallet1990poincare}.
    \S \ref{sec:sprott} is more involved and uses our results to give a periodic orbit existence proof for the Sprott system \eqref{eq:sprott-intro}. 
    The proofs in both \S \ref{sec:repressilator} and \S \ref{sec:sprott} amount to showing that the repressilator and Sprott system satisfy the hypotheses of Theorem \ref{th:hopf}.
    
    Finally, Appendix~\ref{app:closed-1-forms} recalls some standard results concerning closed $1$-forms.	
            			
	\section{Generic families}\label{sec:generic-family-results}

	 \begin{figure}
	 	\centering
	 	%\fbox{ 
	 	\def\svgwidth{.5\columnwidth}
	 	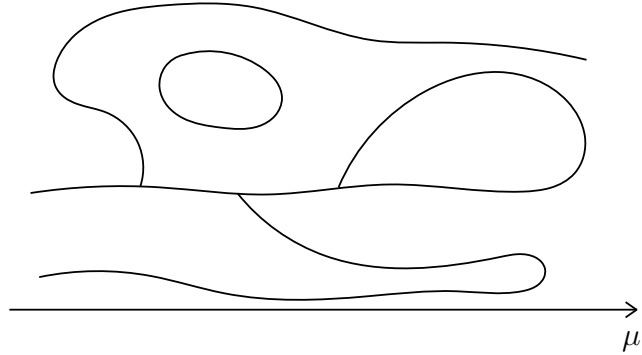
	 	%\input{global-figs/foliation-vs-fiber-bundle.pdf_tex}
	 	%}
	 	\caption{An example orbit diagram for a generic family $f\in \K$ possessing periodic orbits having a uniform upper bound on their periods. Each point corresponds to a single periodic orbit of $f$.
	 		The specific diagram above includes periodic orbits of all three types: 0, 1, and 2 (see Figure \ref{fig:orbit-types}). }\label{fig:ex-orbit-diagram}
	 \end{figure}	
	
    \begin{figure}
    	\centering
    	%\fbox{ 
    	\def\svgwidth{.5\columnwidth}
    	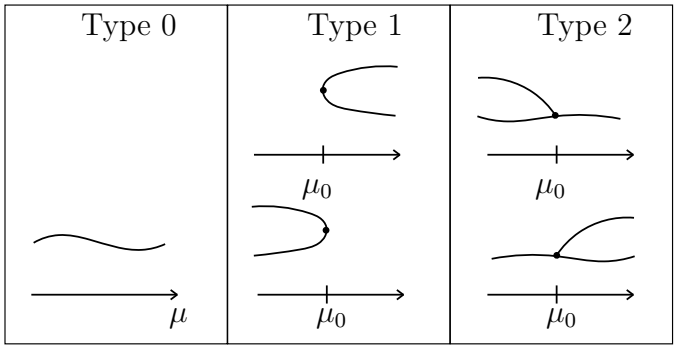
    	%\input{global-figs/foliation-vs-fiber-bundle.pdf_tex}
    	%}
    	\caption{Portions of orbit diagrams containing the three types of periodic orbits occurring in generic one-parameter families, assuming that periods are bounded on a neighborhood within the containing component of periodic orbits. In the second and third columns, the dots correspond to type 1 and type 2 orbits occurring at $\mu = \mu_0$. All other points in all three columns correspond to type 0 orbits. }\label{fig:orbit-types}
    \end{figure}

	\subsection{Background on generic families}\label{sec:background-generic-gamilies}
	
	Sotomayor showed that, in the $C^5$ Whitney (or strong $C^5$) topology, there is a residual subset $\mathcal{K}'\subset C^5(Q\times \R,\T Q)$ of vector field families such that all periodic orbits of $f\in \K'$ are either hyperbolic or are ``quasi-hyperbolic'', meaning that they possess one of three normal forms \cite[Thm~A]{sotomayor1973generic}; similar results emphasizing diffeomorphisms rather than flows were obtained by Brunovsk\`{y} \cite{brunovksy1971one,brunovsky1971one2}.\footnote{Actually, Sotomayor assumed that $Q$ is compact, replaced the parameter space $\R$ with the circle $S^1$, and considered the weak $C^5$ (or $C^5$ compact-open) topology.
	However, the same proofs work for noncompact $Q$ and parameter space $\R$ if the $C^5$ Whitney topology is used.
	Sotomayor does point out that the parameter space can be taken to be $\R$ if the Whitney topology is used in \cite[p.~572,~Rem.~4]{sotomayor1973generic}.}
    Sotomayor's results in particular imply that every periodic orbit either (i) has no Floquet multipliers which are roots of unity, (ii) is a point of generic saddle-node bifurcation, or (iii) is a point of generic period-doubling bifurcation.
	An outline of another proof is given in the appendix of \cite{alligood1984families}, where the hyperbolic and quasi-hyperbolic periodic orbits of \cite{sotomayor1973generic} are referred to simply as types $0'$, $1'$, and $2'$.
    Type $1'$ and $2'$ orbits have no multipliers other than $\pm 1$ on the unit circle, but utilizing Lyapunov-Schmidt---rather than center manifold---reduction in our reasoning will enable us to relax this restriction and prove results for a larger
    subset $\K\supset \K'$ of families having periodic orbits of three types which are more general than $0'$, $1'$, and $2'$.
    Following \cite{mallet1982snakes,alexander1983continuability,alligood1983index,alligood1984families}, we refer to these more general types of orbits as types 0, 1, and 2 (type $0$ is actually the same as type $0'$). 	
    Note that since the space of $C^5$ vector field families equipped with the Whitney topology is a Baire space \cite[Thm~4.4(b)]{hirsch1976differential}, it follows that $\K'$---hence also $\K$---is dense in the space of $C^5$ families. 
    In the sequel, as in the mentioned references we sometimes simply refer to families in $\K$ as ``generic''.     
    
    In order to provide visual aid for our descriptions of these orbit types, we  introduce ``orbit diagrams'' as in \cite{mallet1982snakes,alexander1983continuability,alligood1983index}.
    We can introduce an equivalence relation $\sim$ on the subset $\mathcal{O}\subset Q\times \R$ of periodic orbits of $f$ so that $(x,\mu)\sim (y,\nu)$ if and only if $\mu = \nu$ and $x,y$ lie on the same periodic orbit.
    Since the natural projection $\pi_2\colon \mathcal{O} \to \R$ descends to a map $\widetilde{\pi}_2\colon (\mathcal{O}/\sim) \to \R$, we can ``plot'' $(\mathcal{O}/\sim)$ as a multi-valued function of $\mu$ with each point representing a periodic orbit of $f$.
    An example orbit diagram for a generic family is shown in Figure \ref{fig:ex-orbit-diagram}.
    This specific orbit diagram happens to contain orbits of all three types.
    We now proceed to define orbits of type 0, 1, and 2, which are also illustrated via orbit diagrams in Figure \ref{fig:orbit-types}.

    A type 0 orbit is one which has no Floquet multipliers that are roots of unity.
    In particular, since $+1$ is not a Floquet multiplier, applying the implicit function theorem to a Poincar\'{e} map shows that a type 0 orbit is locally continuable as a function of $\mu$ along a unique branch of orbits on which periods vary continuously.
    
    A type 1 orbit $\gamma$ has a single (algebraically simple) Floquet multiplier equal to $+1$, no other multipliers which are roots of unity, and we require that the eigenvalue $\lambda_1(\mu)$ satisfying $\lambda_1(\mu_0)=1$ crosses the unit circle with nonzero velocity: $\lambda_1'(\mu_0)\neq 0$.
    Let $(x_0,\mu_0)\in \Gamma$ be a point on the image $\Gamma$ of $\gamma$ and let $T_0$ be the period of $\gamma$.
    Letting $S$ be a codimension-1 embedded submanifold intersecting $\Gamma$ transversely at $(x_0,\mu_0)$, $U\subset S$ and $J\subset \R$  sufficiently small neighborhoods of $x_0$ and $\mu_0$, and $P\colon U\times J\to S$ a ($\mu$-dependent) Poincar\'{e} map, for a type 1 orbit we additionally require that certain generic conditions are satisfied by the partial derivatives of $P$ at $(x_0,\mu_0)$ so that the one-dimensional Lyapunov-Schmidt reduction \cite[Ch.~1.3]{golubitsky1985singularities} of the equation $(P(x,\mu)-x = 0)$ undergoes a generic saddle node bifurcation at $(x_0,\mu_0)$  (see \cite[pp.~241--242]{robinson1999ds}). 
    It follows that there are two unique branches of fixed points of $P$---and hence two branches of periodic orbits for $f$--- which approach each other as $\mu$ increases (resp. decreases), coalesce at $\mu = \mu_0$, and disappear for $\mu > \mu_0$ (resp. $\mu<\mu_0$).
    See Figure \ref{fig:orbit-types}.
    Furthermore, it follows from the implicit function theorem that the periods of the orbits corresponding to $\mu$ in each of these two families asymptotically become equal as $\mu\to \mu_0$; additionally, there are no other periodic orbits near $\Gamma$ having periods near $T_0$ except for those orbits on one of the two bifurcating branches described above.
    
    A type 2 orbit $\gamma$ has a single (algebraically simple) Floquet multiplier equal to $-1$, no other multipliers which are roots of unity, and we require that the eigenvalue $\lambda_1(\mu)$ satisfying $\lambda_1(\mu_0)=-1$ crosses the unit circle with nonzero velocity: $\lambda_1'(\mu_0)\neq 0$.
    As above let $(x_0,\mu_0)\in \Gamma$ and let $T_0$ be the period of $\gamma$. 
    Letting $P\colon U\times J\to S$ be a Poincar\'{e} map as above, we additionally require that certain generic conditions are satisfied by the partial derivatives of $P$ at $(x_0,\mu_0)$ so that the one-dimensional Lyapunov-Schmidt reduction of the equation $(P_\mu\circ P_\mu(x)-x = 0)$ undergoes a standard pitchfork bifurcation at $(x_0,\mu_0)$;\footnote{See \cite[Thm~7.3.1]{robinson1999ds} for conditions applicable to the one-dimensional case, and \cite[p.~33,~eq.~3.23]{golubitsky1985singularities} for the Lyapunov-Schmidt translation to conditions applicable to the higher-dimensional case.} this implies that $P$ undergoes a version of the period-doubling or flip bifurcation at $(x_0,\mu_0)$.
    The preceding implies that $\gamma$ is locally continuable as a function of $\mu$ (since $+1$ is not a multiplier of $\gamma$), and also that there exists an additional branch of periodic orbits bifurcating from $\gamma$.
    Furthermore, the (minimal) periods of the orbits on the bifurcating branch at $\mu$ tend to twice the period of $\gamma$ as $\mu \to \mu_0$, and no orbits on the bifurcating branch sufficiently close to $\gamma$ have $+1$ as a multiplier.
    It follows from the implicit function theorem that the periods of orbits vary continuously when traveling between the two branches of ``short'' orbits emanating from a type 2 orbit, but that the periods jump by a factor of two when entering the branch of ``long'' orbits arising from the period-doubling bifurcation.
    Because the Lyapunov-Schmidt proof of this period-doubling bifurcation is based on the implicit function theorem, it additionally follows that there are no periodic orbits near $\Gamma$ having periods near $T_0$ or $2 T_0$ except for those orbits on one of the three branches (two ``short'' and one ``long'') described above.
    
    If $A\subset Q\times \R$ is a connected component of periodic orbits for a generic family and $\sim$ is the equivalence relation defined above, it follows from the above discussion that $A/\sim$ has a fairly simple structure, except possibly for phenomena involving orbits with very large periods; compare with Figure \ref{fig:ex-orbit-diagram}.
    After stating the following definition, we record the properties of $A/\sim$ we need in Proposition \ref{prop:generic-background}.

	\begin{Def}[Consistently oriented curves in the M\"{o}bius band]\label{def:consist-orient}
		Let $X$ be the M\"{o}bius band (with boundary). 
		Let $\Gamma_1$ be the middle circle of the M\"{o}bius band, $\Gamma_2$ be the boundary circle, and let $\pi\colon X\to \Gamma_1$ be the straight-line retraction of $X$ onto the middle circle.
		Then (depending on orientations) the degree of $\pi|_{\Gamma_2}\colon \Gamma_2\to \Gamma_1$ is $\pm 2$.
		We say that $\Gamma_1$ and $\Gamma_2$ are \concept{consistently oriented} if the degree of $\pi|_{\Gamma_2}$ is $+2$.	    
	\end{Def}    
    
	\begin{Prop}\label{prop:generic-background}
		Let $Y\subset Q \times \R$ be an arbitrary subset of nonstationary periodic orbits for a generic family $f\in \K\subset C^5(Q\times \R,\T Q)$.
		Define an equivalence relation $\sim$ on $Y$ so that $(x,\mu)\sim (y,\nu)$ if and only if $\mu = \nu$ and $x,y$ lie on the same periodic orbit. 
		Let $\pi\colon Y\to Y/\sim$ be the quotient map, and let $[(x,\mu)]\coloneqq \pi(x,\mu)$ denote the equivalence class of $(x,\mu)\in Y$.
		If $\gamma$ is a periodic orbit for $f_\mu$ with image $\Gamma$ satisfying $\Gamma\times \{\mu\} \subset Y$, then by an abuse of notation we let $[\gamma]\coloneqq [(\gamma(0),\mu)]$.
		 
		We have the following.
		\begin{enumerate}
			\item\label{item:pi-open} The quotient map $\pi\colon Y\to Y/\sim$ is open.
			If the periods of orbits in $Y$ are uniformly bounded from above, then $\pi$ is also closed and $Y/\sim$ is Hausdorff.
			\item \label{item:type0-1-continuation}  Assume that $Y$ is an open subset of a connected component of nonstationary periodic orbits.
			If $[\gamma]\in Y/\sim$ is a type 0 or type 1 orbit, then there exists $\epsilon > 0$ and a $C^5$ homotopy $H\colon S^1 \times (-\epsilon,\epsilon)\to Q\times \R$ with the following properties.
			\begin{itemize}
				\item For each $s\in (-\epsilon,\epsilon)$, $H_s\coloneqq H(\slot,s)$ is a diffeomorphism onto the image of a periodic orbit in $Y$.
				\item For any $z\in S^1$, the map $(-\epsilon,\epsilon)\to Y/\sim$ given by $s\mapsto \pi \circ H_s(z)$ is a homeomorphism onto a subset $U\subset Y/\sim$ containing $[\gamma]$, and $\pi \circ H_0(z)=[\gamma]$.
				\item For every $N > 0$, there exists a neighborhood $V_N\subset Y/\sim$ of $[\gamma]$ such that $V_N\setminus U$ contains only orbits with periods greater than $N$.
			\end{itemize}
			
			\item\label{item:vertex-nbhd} Assume that $Y$ is an open subset of a connected component of nonstationary periodic orbits.
			If $[\gamma]\in Y/\sim$ is a type 2 orbit, then there are three disjoint arcs $S_1,S_2,S_3\subset Y/\sim$ homeomorphic to open intervals such that the following holds.
			\begin{itemize}
				\item There exists $\epsilon > 0$ and a $C^5$ homotopy $H\colon S^1\times (-\epsilon,\epsilon)\to Q\times \R$ satisfying the same properties as the homotopy in \ref{item:type0-1-continuation}, except that the map $s\mapsto \pi\circ H_s(z)$ is a homeomorphism onto $U\coloneqq S_1 \cup [\gamma]\cup S_2$. 
				
				\item 	If $[\alpha]\in S_3$, then there exists a $C^4$ embedded M\"{o}bius band $X\subset Q\times \R$ such that, when viewed as subsets of $Q\times \R$, the images of $\gamma$ and $\alpha$ are respectively the middle and boundary circles of $X$, and these images are consistently oriented when given the orientations induced by $\gamma$ and $\alpha$.
				\item For every $N > 0$, there exists a neighborhood $V_N\subset Y/\sim$ of $[\gamma]$ such that $V_N\setminus ([\gamma]\cup S_1\cup S_2\cup S_3)$ contains only orbits with periods greater than $N$.
			\end{itemize}

			\item\label{item:isolated-from-bounded-period-orbits} Assume that $Y$ is a connected component of periodic orbits. If $(x_n,\mu_n)$ is a sequence of points on the images of periodic orbits $\gamma_n$ with $(x_n,\mu_n)\not \in Y$ but $(x_n,\mu_n)\to (x,\mu)\in Y$ as $n\to \infty$, then the periods $\tau_n$ of the $\gamma_n$ satisfy $\tau_n\to \infty$.
		\end{enumerate}		
	\end{Prop}
	
	\begin{proof}
		We begin by proving \ref{item:pi-open}, which is true even without the hypothesis that $f\in \K$.
		Let $\Phi$ be the flow of the vector field $(f,0)$ on $Q\times \R$.
		First note that, for any subset $S\subset Q\times \R$, $\pi^{-1}(\pi(S \cap Y))$ is equal to the intersection $Y \cap \bigcup_{t\in \R}\Phi^t(S)$, and is additionally equal to the intersection $Y \cap \bigcup_{t\in [0,T]}\Phi^t(S)$ if the periods of orbits through $S$ are bounded above by $T$.
		Next, note that $\pi$ is an open (closed) map if and only if, for every open (closed) subset $S\subset Q\times \R$, $\pi^{-1}(\pi(S\cap Y))$ is open (closed) in $Y$. 
		It follows that $\pi$ is open since $$\pi^{-1}(\pi(U\cap Y)) = Y \cap \bigcup_{t\in \R}\Phi^t(U)$$
		is the intersection of an open subset with $Y$ if $U\subset Q\times \R$ is open.
		If the periods of $Y$ are bounded above by $T>0$, then it follows that $\pi$ is closed since $$\pi^{-1}(\pi(C\cap Y)) = Y \cap \bigcup_{t\in [0,T]}\Phi^t(C)$$ is the intersection of a closed subset with $Y$ if $C\subset Q\times \R$ is closed.
		To show that $Y/\sim$ is Hausdorff if the periods of $Y$ have an upper bound $T$, consider any $(x_1,\mu_1), (x_2,\mu_2)\in Y$ belonging to distinct orbits, and let $U_1,U_2\subset Q\times \R$ be disjoint neighborhoods of the images $\Gamma_1,\Gamma_2 \subset Q\times \R$ of the periodic orbits of $(f,0)$ through the $(x_i,\mu_i)$.
		For $i\in \{1,2\}$ we have that $\Phi^{-1}(U_i)$ is an open neighborhood of $\Gamma_i\times \R$ and therefore contains a subset of the form $V_i\times [0,T]$ with $V_i\subset U_i$ an open neighborhood of $\Gamma_i$.
		Hence $\pi^{-1}(\pi(V_1\cap Y)) = Y\cap \bigcup_{t\in [0,T]}\Phi^t(V_1)\subset U_1$ and $\pi^{-1}(\pi(V_2\cap Y)) = Y \cap \bigcup_{t\in [0,T]}\Phi^t(V_2)\subset U_2$ are disjoint open neighborhoods of $\Gamma_1$ and $\Gamma_2$ in $Y$, so it follows that $\pi(Y\cap V_1)$ and $\pi(Y\cap V_2)$ are disjoint neighborhoods of $\pi(\Gamma_1)$ and $\pi(\Gamma_2)$ as desired. 
		This completes the proof of \ref{item:pi-open}.
		
		The existence of homotopies $H$ satisfying the properties claimed in \ref{item:type0-1-continuation} and \ref{item:vertex-nbhd} follows from the discussion preceding Definition \ref{def:consist-orient} and standard techniques from the textbooks cited therein.
		To show that the neighborhoods $V_N$ of \ref{item:type0-1-continuation} and \ref{item:vertex-nbhd} exist, fix any $N > 0$ and let $\gamma$ be a type $0$, $1$, or $2$ orbit with image $\Gamma$ and period $T_0$.
		Assume, to obtain a contradiction, that every neighborhood $V\subset Q\times \R$ of $\Gamma$ contains a point $(x,\mu)$ on a periodic orbit having period less than $N$. 
		By the discussion preceding Definition \ref{def:consist-orient} and continuity of the flow, by taking $V$ to be a sufficiently small tubular neighborhood of $\Gamma$ we may assume that the period $T$ of the orbit through any such $(x,\mu)$ satisfies $2 T_0 - \epsilon \leq  T \leq N$ if $\gamma$ is a type $0$ or 1 orbit, and $3T_0-\epsilon \leq T \leq N$ for the case that $\gamma$ is a type 2 orbit.
		Here $\epsilon > 0$ is some small number which can be chosen to tend to $0$ as the size of the neighborhood $V$ tends to zero.
		It follows that $\gamma$ has a virtual period which is of order at least $2$ if $\gamma$ is type $0$ or $1$ and of order at least $3$ if $\gamma$ is type $2$ \cite[Prop.~3.2]{chow1983periodic}.
		But this contradicts the fact that, by definition, type $0$ and type $1$ orbits have no multipliers which are second or higher roots of unity, and type $2$ orbits have no multipliers which are third or higher roots of unity.
		This shows that $V_N\coloneqq \pi(V\cap Y)$ is a neighborhood satisfying the properties claimed in \ref{item:type0-1-continuation} and \ref{item:vertex-nbhd}; since $N$ was arbitrary, this argument also proves \ref{item:isolated-from-bounded-period-orbits}.
		
		It remains only to prove the M\"{o}bius band claim in \ref{item:vertex-nbhd}.
		Let $\gamma$ be a type 2 orbit with image $\Gamma$ and let $(x_0,\mu_0)\in \Gamma$.
		Without loss of generality, assume that the period-doubling bifurcation corresponding to $\gamma$ is supercritical.
		Let $S\subset Q\times \R$ be a Poincar\'{e} section centered at $(x_0,\mu_0)$ and let $P\colon U\subset S\to S$ be a return map for the flow $\Phi$ of $(f,0)$.
		After shrinking the arc $S_3$ if necessary, the period-doubling proof based on Lyapunov-Schmidt reduction yields $\epsilon > 0$ and a $C^4$ arc $Z\subset U \cap (Q\times [\mu_0, \epsilon))$ of fixed points of $P\circ P$ such that (i) $S_3 \cup [\gamma] = \pi(Z)$, (ii) $P|_Z\colon Z\to Z$ is an orientation-reversing diffeomorphism fixing $(x_0,\mu_0)$, and (iii) for all $\mu \in (\mu_0,\mu_0+\epsilon)$, $Z_\mu\coloneqq Z\cap (Q\times [\mu_0,\mu])$ is a connected $P$-invariant $C^4$ embedded submanifold with boundary $\partial Z_\mu = Z\cap (Q\times \{\mu\})$ consisting of two points.
		Hence clearly $\widetilde{X}_\mu\coloneqq \bigcup_{t\in \R}\Phi^t(Z_\mu)$ is a $C^4$ embedded M\"{o}bius band with middle circle given by $\Gamma$ and consistently oriented boundary circle given by the image of a periodic orbit $\alpha$ with $[\alpha]\in S_3$.
        Additionally, (i) implies that the image of the orbit $\alpha$ corresponding to any $[\alpha]\in S_3$ is the boundary of $\widetilde{X}_\mu$ for some $\mu \in (\mu_0,\mu_0+\epsilon)$.
        This completes the proof.
	\end{proof}

	\subsection{Global continuation for generic families}\label{sec:generic-family-continuation}
	
	In this section, we establish in Lemma \ref{lem:cf-ell-cont-generic-fam} (a slightly strengthened version of) Theorem \ref{th:main-thm} in the special case that $f\in \K$ is a generic family.
	This will enable us to prove Theorem \ref{th:main-thm} by approximating an arbitrary family $f$ by generic families $g\in \K$.

	For convenience, we record here the following standard result.
	\begin{Lem}[Homotopy invariance]\label{lem:link-homotopy-invariance}
		Let $M$ be a smooth manifold and $\cf$ be a $C^1$ closed 1-form on $M$.	
		If $\alpha,\beta\colon S^1\to M$ are $C^1$ maps which are homotopic, then $$\int_{\alpha}\cf= \int_{\beta}\cf.$$	
	\end{Lem}	
	The following preliminary result is also straightforward.
	We include a proof for convenience. 
		
	\begin{Lem}\label{lem:mobius-c1f-integrals}
		Let $X$ be a $C^1$ M\"{o}bius band (with boundary).
		Let $\Gamma_1$ be the middle circle of the M\"{o}bius band and $\Gamma_2$ the boundary, and assume $\Gamma_1,\Gamma_2$ are consistently oriented (Definition~\ref{def:consist-orient}).
		Then if $\cf$ is any $C^1$ closed $1$-form on $X$,
		\begin{equation}
		\int_{\Gamma_2}\cf = 2\int_{\Gamma_1} \cf.
		\end{equation}  
	\end{Lem}
	\begin{proof}
		For $i=1,2$, let $\iota_i\colon \Gamma_i\to X$ be the inclusion.
		Let $\pi\colon X\to \Gamma_1$ be the straight line retraction of $X$ onto $\Gamma_1$.
		Let $h \colon X\times [0,1]\to X$ be the straight-line deformation retraction of $X$ onto $\Gamma_1$, with $h_t\coloneqq h(\slot,t)$, $h_0 = \id_X$, and $h_1 = \iota_1 \circ \pi$.

		By Definition \ref{def:consist-orient}, the degree of $\pi|_{\Gamma_2}\colon \Gamma_2\to \Gamma_1$ is $+2$.
		Hence
		\begin{equation}\label{eq:mobius-c1f-integrals-1}
		\int_{\Gamma_2}(\pi|_{\Gamma_2})^* \cf \coloneqq  \int_{\Gamma_2}(\pi|_{\Gamma_2})^* \iota_1^* \cf = 2 \int_{\Gamma_1}\iota_1^* \cf \eqqcolon 2 \int_{\Gamma_1} \cf.
		\end{equation}
		Since $h_t|_{\Gamma_2}$ yields a homotopy
		\begin{equation*}
		\iota_2 = h_0|_{\Gamma_2} \simeq h_1|_{\Gamma_2} =  \iota_1 \circ \pi|_{\Gamma_2},
		\end{equation*}
		$\iota_2^* \colon \Hdr^1(X)\to \Hdr^1(\Gamma_2)$ and $(\iota_1\circ \pi|_{\Gamma_2})^*\colon \Hdr^1(X)\to \Hdr^1(\Gamma_2)$ are the same map on cohomology.
		It follows that $\iota_2^* \cf = (\iota_1\circ \pi|_{\Gamma_2})^*\cf + dV$ for some exact 1-form $dV$, so $\iota_2^* \cf$ and  $(\iota_1\circ \pi|_{\Gamma_2})^*\cf$ have the same integral over $\Gamma_2$.\footnote{Strictly speaking, we are also using the fact that every $C^1$ closed form is cohomologous to a $C^\infty$ closed form \cite[pp.~61--70]{deRham1984differentiable} since we only assume $\cf\in C^1$.} 
		Since $(\iota_1\circ \pi|_{\Gamma_2})^* = (\pi|_{\Gamma_2})^*\iota_1^*$ we obtain
		\begin{equation*}
		\int_{\Gamma_2}\cf \coloneqq \int_{\Gamma_2}\iota_2^*\cf = \int_{\Gamma_2}(\pi|_{\Gamma_2})^*\iota_1^*\cf  \eqqcolon \int_{\Gamma_2}(\pi|_{\Gamma_2})^*\cf = 2 \int_{\Gamma_1}\cf,
		\end{equation*}
		where the last equality follows from \eqref{eq:mobius-c1f-integrals-1}.
		This completes the proof.
	\end{proof}		
	
	One of the key ideas needed for Lemma \ref{lem:cf-ell-cont-generic-fam} is contained in the following Lemma \ref{lem:A-ell-manifold} which shows that, for a generic family, the periodic orbit components $\widetilde{A}^i_\ell$ of Definition \ref{def:cf-ell-global-cont} are topological 1-manifolds if the periods of $\widetilde{A}^i_\ell$ are uniformly bounded. 	
	
	Recall that $\iota_\mu\colon Q\hookrightarrow Q\times \R$ denotes the inclusion $\iota_\mu(x)=(x,\mu)$ for each $\mu \in \R$.
	
	\begin{Lem}\label{lem:A-ell-manifold}
		Let $f\in \K\subset C^5(Q\times \R,\T Q)$ be a generic family of vector fields and let $\cf$ be a $C^1$ closed 1-form on an open subset $\domcf\subset Q\times \R$.
		Let $A\subset \domcf$ be a connected component of nonstationary periodic orbits of $f|_{\domcf}$, and let the equivalence relation $\sim$ on $A$ be as in Proposition \ref{prop:generic-background}.
		For any $\ell > 0$, let $A_\ell\subset A$ be the subset of points on periodic orbits $\alpha_\mu$ with $\left|\int_{\alpha_\mu} \iota_\mu^*\cf\right| = \ell$ and $A_{\leq \ell}$ the subset with $\left|\int_{\alpha_\mu} \iota_\mu^*\cf\right| \leq \ell$.
		
		Fix an open subset $U\subset A/\sim$ and $\ell > 0$, and assume that $U\cap (A_\ell/\sim) \neq \varnothing$.
		Let $V_\ell$ be a connected component of $U\cap (A_\ell/\sim)$ and let $V_{\leq \ell}$ be the unique component of $U\cap (A_{\leq\ell}/\sim)$ containing $V_\ell$.
		Assume that the periods of orbits belonging to $V_{\leq \ell}$ have a uniform upper bound.
		Then
		\begin{enumerate}
			\item $V_{\ell}$ is a topological $1$-manifold (without boundary).
			\item If $V_\ell = V_{\leq \ell}$, then $V_\ell$ is also closed as a subset of $U$.
		\end{enumerate}
	\end{Lem}

	\begin{proof}
		Let $\widetilde{V}_{\leq \ell}$ and $\widetilde{V}_\ell$ be the union of orbits in $A$ with $(\widetilde{V}_{\leq \ell}/\sim) = V_{\leq \ell}$ and $(\widetilde{V}_\ell/\sim) = V_\ell$, respectively.  Define the quotient map $\pi\colon \widetilde{V}_{\leq \ell}\to V_{\leq \ell}$.
		Since the periods of orbits in $\widetilde{V}_{\leq \ell}$ are bounded above, part \ref{item:pi-open} of Proposition \ref{prop:generic-background} implies that $\pi$ is open, $\pi$ is closed, and $V_{\leq \ell}$ is Hausdorff; since $\pi$ is open and $\widetilde{V}_{\leq \ell}$ is second countable, so is $V_{\leq \ell}$.
		It follows that the subspace $V_{\ell}\subset V_{\leq \ell}$ is also Hausdorff and second countable.

		We now show that $V_\ell$ is a topological $1$-manifold.
		Since we have already shown that $V_\ell$ is Hausdorff and second countable, we need only establish that $V_\ell$ is locally Euclidean of dimension $1$.	
		If $\alpha$ is a type $0$ or type $1$ orbit with $[\alpha]\in V_\ell$, then Lemma \ref{lem:link-homotopy-invariance}, period-boundedness, and part \ref{item:type0-1-continuation} of Proposition \ref{prop:generic-background} imply that $[\alpha]$ has a neighborhood in $U$ homeomorphic to an open interval and contained in $V_\ell$.
		If instead $\alpha$ is a type $2$ orbit, then Lemmas \ref{lem:link-homotopy-invariance} and \ref{lem:mobius-c1f-integrals}, period-boundedness, and part \ref{item:vertex-nbhd} of Proposition \ref{prop:generic-background} imply (since we are assuming $\ell >0$) that $[\alpha]$ again has a neighborhood in $U$ homeomorphic to an open interval and contained in $V_\ell$. 
		This shows that $V_\ell$ is locally Euclidean and completes the proof that $V_\ell$ is a topological $1$-manifold.
		
		We next show that $V_{\leq \ell}$ is closed as a subset of $U$.
		Being a connected component of $U\cap(A_{\leq\ell}/\sim)$, $V_{\leq\ell}$ is automatically closed in $U\cap(A_{\leq\ell}/\sim)$, so any $[\beta] \in U\cap(A_{\leq\ell}/\sim)\setminus V_{\leq \ell}$ has a neighborhood disjoint from $V_{\leq \ell}$.
		It remains only to show that any point in $U$ not in $A_{\leq\ell}/\sim$ has a neighborhood disjoint from $V_{\leq \ell}$.
		Fix $[\beta]\in U\setminus(A_{\leq\ell}/\sim)$ with $\beta$ an orbit of $f_\mu$, so that $\left|\int_{\beta}\iota_{\mu}^* \cf\right|	= \ell' > \ell$.	
		If $\beta$ is a type 0 or type 1 orbit, then Lemma \ref{lem:link-homotopy-invariance}, period-boundedness of $V_{\leq \ell}$, and part \ref{item:type0-1-continuation} of Proposition \ref{prop:generic-background} imply that $[\beta]$ has a neighborhood in $U$ disjoint from $V_{\leq\ell}$.
		If instead $\beta$ is a type 2 orbit, then part \ref{item:vertex-nbhd} of Proposition \ref{prop:generic-background} implies that, for each $N > 0$, $[\beta]$ has a neighborhood $W\subset U$ such that every $[\gamma]\in W$ having period smaller than $N$ satisfies $\left|\int_{\gamma}\iota_{\mu'}^*\cf\right|\in \{\ell',2\ell'\}$, where $\mu'$ is such that $\gamma$ is an orbit of $f_{\mu'}$.
		Taking $N$ to be larger than an upper bound for the periods of $V_{\leq \ell}$ and using the fact that $\ell' > \ell$, it follows that $W \cap V_{\leq\ell} = \varnothing$.
		This completes the proof that $V_{\leq \ell}$ is closed in $U$.				
		
		Since $V_{\leq \ell}$ is closed in $U$, the additional assumption that $V_\ell = V_{\leq \ell}$ implies that $V_\ell$ is closed in $U$.
		This completes the proof.
	\end{proof}
	
    \begin{figure}
    	\centering
    	%\fbox{ 
    	\def\svgwidth{.5\columnwidth}
    	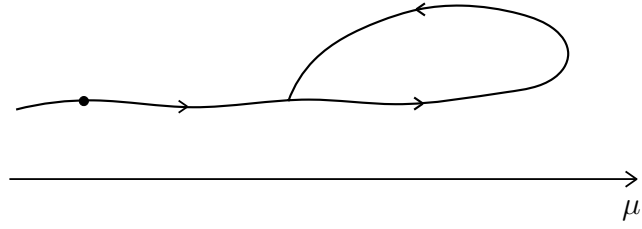
    	%\input{global-figs/foliation-vs-fiber-bundle.pdf_tex}
    	%}
    	\caption{The main idea behind Lemma \ref{lem:A-ell-manifold} is that its hypotheses imply that portions of orbit diagrams such as the one shown above cannot occur if the corresponding periodic orbits are contained in $\domcf$ and have uniformly bounded periods.
    	    In more detail: if the periodic orbit $\gamma$ represented by the above dot at $\mu_0$ satisfies $\ell\coloneqq \left|\int_{\gamma} \iota_{\mu_0}^*\cf\right| > 0 $, then the orbit diagram above cannot occur for a generic one-parameter family.
    		To see this, let $\beta$ be a periodic orbit represented by a point in the top of the loop at $\mu_1$.
    		There is a homotopy of periodic orbits corresponding to the path indicated by the arrows above, so homotopy invariance (Lemma \ref{lem:link-homotopy-invariance}) implies that $\left|\int_{\beta} \iota_{\mu_1}^*\cf\right|=\ell$.
    		On the other hand, applying Lemma \ref{lem:mobius-c1f-integrals} to the branch of bifurcating orbits near the type 2 orbit implies that $\left|\int_{\beta} \iota_{\mu_1}^* \cf\right|=2\ell\neq \ell$, a contradiction.
    	}\label{fig:lem-mfld-b}
    \end{figure}  	
		
	    \begin{figure}
	    	\centering
	    	%\fbox{ 
	    	\def\svgwidth{.5\columnwidth}
	    	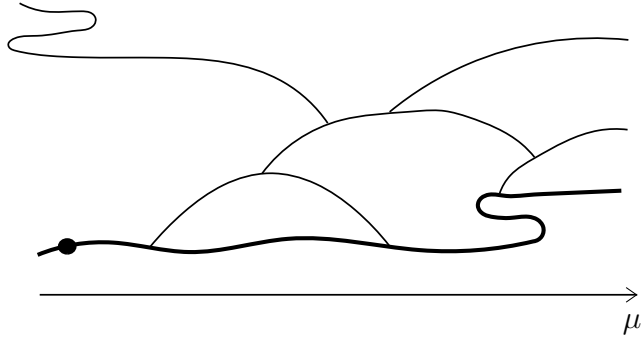
	    	%\input{global-figs/foliation-vs-fiber-bundle.pdf_tex}
	    	%}
	    	\caption{A portion of an orbit diagram for a component $A\subset \domcf$ of a generic family having uniformly bounded periods.
	    	In the notation of Lemma \ref{lem:A-ell-manifold}, the thick curve represents a portion of a set $V_\ell$ satisfying all hypotheses (including $V_\ell = V_{\leq \ell}$) of Lemma \ref{lem:A-ell-manifold} (here $U = A/\sim$, where $\sim$ is the equivalence relation defined in Proposition \ref{prop:generic-background}.). 
	    	Since $\ell > 0$, all periodic orbits $\alpha_\mu$ represented by points on the thick curve also satisfy $\left|\int_{\alpha_\mu} \iota_\mu^*\cf\right|=\ell$.
	    	$V_\ell$ is a topological $1$-manifold since the situation depicted in Figure \ref{fig:lem-mfld-b} cannot occur, and $V_\ell$ is a closed subset of $A/\sim$ since parts \ref{item:type0-1-continuation} and \ref{item:vertex-nbhd} of Proposition \ref{prop:generic-background} imply that the function $[\alpha_\mu]\mapsto\left|\int_{\alpha_\mu} \iota_\mu^*\cf\right|$ is lower semi-continuous	on $(A/\sim)$. 
	    	}\label{fig:lem-mfld-a}
	    \end{figure}  
	
	We now state the main result of this section. 
	Lemma \ref{lem:cf-ell-cont-generic-fam} yields a result for general families slightly stronger than Theorem \ref{th:main-thm}, because it does not require the hypothesis that $+1$ is not a Floquet multiplier of the periodic orbit $\gamma$.
	\begin{Lem}[$(\cf,\ell)$-global continuability for generic families]\label{lem:cf-ell-cont-generic-fam}
		Let $f\in \K\subset C^5(Q\times \R,\T Q)$ be a generic family of vector fields, and let $\cf$ be a $C^1$ closed 1-form on an open subset $\domcf\subset Q\times \R$.
		Let $A\subset \domcf$ be a connected component of nonstationary periodic orbits of $f|_{\domcf}$.
		Let $\gamma$ be a periodic orbit for some $f_{\mu_0}$ with image $\Gamma$ satisfying $\Gamma\times \{\mu_0\}\subset A$, define $\ell\coloneqq \left|\int_{\gamma}\iota_{\mu_0}^*\cf\right|$, and assume $\ell > 0$.		
		Then $\gamma$ is $(\cf,\ell)$-globally continuable. 
	\end{Lem}
	\begin{Rem}
		The following proof is similar in spirit to the proof of \cite[Thm~2.2]{alligood1984families} with  ``orbits $\alpha_\mu$ satisfying $\left|\int_{\alpha_\mu}\iota_{\mu}^*\cf\right|=\ell$'' playing the role of ``non-M\"{o}bius orbits.'' (Recall that a M\"{o}bius orbit is a periodic orbit which has an odd number of Floquet multipliers in $(-\infty,1)$ and which additionally has no multiplier equal to $-1$.) 
	\end{Rem}
        
	\begin{proof}	 
		We use the notation of Definition \ref{def:cf-ell-global-cont}, and identify $\Gamma$ with $\Gamma \times \{\mu_0\} \subset Q\times \R$ in the following.
		
		Assume that $\gamma$ is not $(\cf,\ell)$-globally continuable. 
		Then $\widetilde{A}_{\leq \ell}\setminus \Gamma$ is disconnected, one of the components $\widetilde{A}_{\leq \ell}^1$ of $\widetilde{A}_{\leq \ell}\setminus \Gamma$ is equal to a component $\widetilde{A}_\ell^1$ of $\widetilde{A}_\ell\setminus \Gamma$, the closure $\cl(\widetilde{A}_\ell^1)$ of $\widetilde{A}_\ell^1$ in $\domcf$ is compact and contains no generalized centers, and the periods of orbits in $\widetilde{A}_\ell^1$ have a uniform upper bound.

		Since $(A\setminus \Gamma)/\sim$ is an open subset of $(A/\sim)$, Lemma \ref{lem:A-ell-manifold} implies that $(\widetilde{A}^1_\ell/\sim)\subset (A\setminus \Gamma)/\sim$ is a topological 1-manifold which is closed as a subset of $(A\setminus \Gamma)/\sim$.
		$(\widetilde{A}^1_\ell/\sim)$ is not compact because it has the sole limit point $[\gamma]\in (A/\sim)\setminus (\widetilde{A}^1_\ell/\sim)$.
		Therefore, $(\widetilde{A}^1_\ell\cup \Gamma)/\sim$ is closed as a subset of $A/\sim$, and the classification theorem for topological $1$-manifolds implies that $(\widetilde{A}^1_\ell\cup \Gamma)/\sim$ is homeomorphic to $[0,1)$.
	
        To complete the proof it suffices to show that $\widetilde{A}^1_\ell\cup \Gamma$ is closed in $\domcf$ and therefore compact, because this would imply that $(\widetilde{A}^1_{\ell}\cup \Gamma)/\sim$ is compact, contradicting the fact that $(\widetilde{A}^1_\ell\cup \Gamma)/\sim$ is homeomorphic to $[0,1)$.
        So let $(x,\mu)\not \in (\widetilde{A}^1_\ell\cup \Gamma)$ be a limit point of $(\widetilde{A}^1_\ell \cup \Gamma)$ in $\domcf$.
        Then there is a sequence $(x_n,\mu_n)$ in $\widetilde{A}^1_\ell \cup \Gamma$ of points on periodic orbits $\gamma_n$ with $(x_n,\mu_n) \to (x,\mu)$.
        Let $\tau_n$ be the period of $\gamma_n$.
		Since we are assuming that the periods of $\widetilde{A}^1_\ell$ are bounded, we may pass to a subsequence and assume that $\tau_n\to \tau > 0$.
		Letting $\Phi_{\mu_n}$ be the flow of $f_{\mu_n}$, by continuity we have
		\begin{equation}
		\Phi^\tau_\mu(x) = \lim_{n\to \infty}\Phi^{\tau_n}_{\mu_n}(x_n) = \lim_{n\to\infty}x_n = x.
		\end{equation}
		Since we are assuming that the closure $\cl(\widetilde{A}^1_\ell)$ of $\widetilde{A}^1_\ell$ in $\domcf$ contains no generalized centers, it follows that $(x,\mu)\in \cl(\widetilde{A}^1_\ell \cup \Gamma)$ must be a nonstationary periodic orbit for $f_\mu$.\footnote{This is because, if $x$ were an equilibrium for $f_\mu$, then the boundedness of the $\tau_n$ would imply that $x$ is a generalized center for $f_\mu$. This is true even if $f\in C^1$, and follows from \cite[Prop.~3.2]{chow1983periodic}; see also \cite[Cor.~4.6]{fiedler1988global}. A proof for the case $f\in C^2$ is given in \cite[Prop.~3.1]{mallet1982snakes}.}
		It cannot be the case that $(x,\mu)$ belongs to a component $B\subset \domcf$ of periodic orbits of $f|_{\domcf}$ different from $A$, because this would contradict the fact that $A$ is closed in the space of periodic orbits of $f|_{\domcf}$ (being a connected component).
		Hence $(x,\mu)\in A$, and since $\widetilde{A}^1_\ell \cup \Gamma$ is closed in $A$ it follows that $(x,\mu)\in \widetilde{A}^1_\ell \cup \Gamma$.
		Hence $\widetilde{A}^1_\ell\cup \Gamma$ is closed in $\domcf$.
		As discussed above, this implies a contradiction and completes the proof.
	\end{proof}

\section{Non-generic families}\label{sec:non-generic-families}			
    In this section, we prove our main theorems on global continuation of periodic orbits for arbitrary $C^1$ families of vector fields.
    Before doing this, we require one additional lemma. 
    Lemma \ref{lem:linking-number-lower-semi-cont} enables us to prove Theorem \ref{th:main-thm} without the consideration of ``virtual periods'' as required in \cite[Thm~3.1, Lem.~3.2]{alligood1984families}. 
    \begin{Lem}\label{lem:linking-number-lower-semi-cont}
    	Let $f_n\in C^1(M,\T M)$ be a sequence of $C^1$ vector fields on a smooth manifold $M$ which converge in the weak $C^1$ topology to a $C^1$ vector field $f$ on $M$, and let $\cf$ be a $C^1$ closed 1-form on $M$.
    	For each $n$ let $\gamma_n$ be a periodic orbit of $f_n$ with image $\Gamma_n$ and (minimal) period $\tau_n$, and let $\gamma$ be a periodic orbit of $f$ with image $\Gamma$ and (minimal) period $\tau$.
    	Assume that the periods $\tau_n$ have a uniform upper bound, and assume that for each $n$ there exists $x_n\in \Gamma_n$ such that $x_n\to x_0\in \Gamma$.
    	Then
    	\begin{enumerate}
    		\item\label{item:lem:linking-number-lower-semi-cont1}    	$\liminf_{n\to\infty} \left|\int_{\gamma_n}\cf \right| \geq \left|\int_{\gamma}\cf \right|$ and $\liminf_{n\to\infty}\tau_n \geq \tau$;
         	\item\label{item:lem:linking-number-lower-semi-cont2} $\lim_{n\to\infty} \left|\int_{\gamma_n}\cf \right| = \left|\int_{\gamma}\cf \right|$ if and only if $\lim_{n\to\infty}\tau_n = \tau$.
    	\end{enumerate}
 
    \end{Lem}
    
    \begin{proof}
    	We begin with some preparations.
    	Let $\pi\colon U\to \Gamma$ be a $C^1$ tubular neighborhood of $\Gamma$ with $U$ precompact, so in particular $\pi$ is a submersion and retraction.
    	Since $\pi|_\Gamma=\id_\Gamma$,
    	by shrinking $U$ we may assume by continuity that $\D_y\pi f(y)\neq 0$ for all $y\in U$.
    	Since $f_n \to f$ uniformly on $U$, there exists $N_0>0$ such that the same is true of $\D_y\pi f_n(y)$ for all $n > N_0$.
    	Since the periods of the $\gamma_n$ have a uniform upper bound and since $f_n\to f$, continuous dependence of a flow on its vector field implies that there exists $N_1>N_0$ such that $\Gamma_n\subset U$ for all $n>N_1$.
    	Hence $\pi|_{\Gamma_n}\colon \Gamma_n \to \Gamma$ is well-defined and an orientation-preserving local diffeomorphism for $n> N_1$, where $\Gamma$ and $\Gamma_n$ are given the orientations induced by $\gamma$ and $\gamma_n$.
    	
        Next, since $f$ is transverse to the manifold $S_0\coloneqq \pi^{-1}(x_0)$, the implicit function theorem implies that there is a well-defined $C^1$ ``first impact time map'' $t_f\colon S_1\to S_0$ from a neighborhood $S_1\subset S_0$ of $x_0$ to $S_0$, with $t_f(y)$ defined to be the smallest positive real number such that $\Phi_{f}^{t_f(y)}(y)\in S_0$, where $y\in S_1$ and $\Phi_f$ is the local flow of $f$.
        By the implicit function theorem, $t_f(y)$ is a fortiori jointly continuous in $y$ and $f$ in the $C^1$ topology.  	
        Let $N_2 > N_1$ be such that $\Gamma_n \cap S_0 \subset S_1$ for all $n > N_2$.    	
    	In the remainder of the proof, assume $n > N_2$.
    	
    	We now proceed with the proof of \ref{item:lem:linking-number-lower-semi-cont1}.
    	First, note that for any $y_n\in \Gamma_n\cap S_1$, the definition of the first impact time map implies
    	\begin{equation}\label{eq:lem-cf-per-eqn}
    	t_{f_n}(y_n) \leq \tau_n. 
    	\end{equation}
    	Since the impact time map is continuous and since $y_n\to x_0$, the left hand side converges to $\tau$.
    	This proves the statement about the periods in \ref{item:lem:linking-number-lower-semi-cont1}. 
    	Next, since $\pi|_{\Gamma_n}$ is an orientation-preserving local diffeomorphism, it follows that the degree $d_n$ of $\pi|_{\Gamma_n}$ satisfies $d_n \geq 1$.
    	Since $U$ deformation retracts onto $\Gamma$, the inclusion $\Gamma_n\hookrightarrow U$ is homotopic to the composition of $\pi|_{\Gamma_n}$ with the inclusion $\Gamma\hookrightarrow U$.
    	Hence we have
    	\begin{equation}\label{eq:lem:cf-deg-eqn}
    	\left|\int_{\gamma_n}\cf\right| = \left|\int_{\Gamma_n}(\pi|_{\Gamma_n})^*\cf \right|= d_n \left|\int_{\gamma}\cf\right| \geq \left|\int_{\gamma}\cf\right|.    	
    	\end{equation}
    	This completes the proof of \ref{item:lem:linking-number-lower-semi-cont1}.

        Next, note that \eqref{eq:lem:cf-deg-eqn} implies that $\lim_{n\to\infty} \left|\int_{\gamma_n}\cf \right| = \left|\int_{\gamma}\cf \right|$ if and only if $\lim_{n\to\infty}d_n = 1$.
    	Since $\pi|_{\Gamma_n}\colon \Gamma_n \to \Gamma$ is an orientation-preserving local diffeomorphism, this in turn holds if and only if $\Gamma_n$ intersects $S_1$ in a single point for all sufficiently large $n$.
    	And by the definition of the first impact time map, this latter statement holds if and only if $t_{f_n}(y_n)=\tau_n$ for all sufficiently large $n$, where $y_n\in \Gamma_n\cap S_0$. 
    	So to prove \ref{item:lem:linking-number-lower-semi-cont2}, it suffices to prove that this final statement holds if and only if $\lim_{n\to\infty}\tau_n = \tau$.
    	
    	Assume that $t_{f_n}(y_n)=\tau_n$ for all large $n$.
    	Since $t_{f_n}(y_n)\to \tau$, it follows that $\tau_n\to \tau$.
    	Conversely, assume that there exists a subsequence $n_k\to\infty$ arbitrarily large with $t_{f_{n_k}}(y_{n_k})<\tau_{n_k}$. 
        Then
    	$$t_{f_{n_k}}\left(\Phi_{f_{n_k}}^{t_{f_{n_k}}(y_{n_k})}(y_{n_k})\right) + t_{f_{n_k}}(y_{n_k})\leq \tau_{n_k}.$$
    	By continuity of the impact time map and of $\Phi$ with respect to all arguments and the fact that $y_n\to x_0$, the left hand side converges to $2t_f(x_0) = 2\tau$.
    	Hence $\limsup_{n\to\infty}\tau_n>\tau$.
    	This proves \ref{item:lem:linking-number-lower-semi-cont2}.
    \end{proof}

  \ThmContinuability*
  
  \begin{Rem}
  	Our proof is inspired by the proof of \cite[Thm~3.1, Lem.~3.2]{alligood1984families}, and we have tried to keep our proof similar to theirs in an effort to make the similarities and differences readily discernible, although we have added some details.
  	One key difference is that there is no mention of ``virtual periods'' anywhere in our proof; using Lemma \ref{lem:linking-number-lower-semi-cont}, their role is instead filled by integrals of the form $\left|\int_\alpha \iota_\mu^*\cf\right|$. 
  	This difference also explains why \cite[Lem~3.2]{alligood1984families} requires the assumption that $\gamma$ has no Floquet multipliers which are roots of unity, whereas we need only assume that $+1$ is not a multiplier. 
  	Another key difference in our proof is that our definition of the function $F$ in \eqref{eq:F-def} below differs from the definition of $F$ in the proof of \cite[Lem~3.2]{alligood1984families} in that we have added a second term imposing a ``cost'' for $\left|\int\iota_\mu^*\cf\right|$ to deviate from $\ell$.
  \end{Rem}

\begin{proof}
	The weak $C^1$ (or $C^1$ compact-open) topology on $C^1(Q\times \R, \T Q)$ is (completely) metrizable, and this induces a metric on the closed subspace of one-parameter families of vector fields \cite[p.~62]{hirsch1976differential}.
	Throughout the remainder of this proof, we denote this metric by $\dco(\slot,\slot)$.
	In the following, we identify the images of periodic orbits such as $\Gamma$ with sets of the form $\Gamma \times \{\mu_0\} \subset Q\times \R$ when convenient, and similarly conflate periodic orbits of $(f,0)$ with those of the appropriate $f_\mu$.
	We also note that all topological closures are as subsets of $\domcf$ in the following; when we say that a subset is precompact, we mean that its closure in $\domcf$ is compact.
	
	Assume that $\gamma$ is not $(\cf,\ell)$-globally continuable.
    Using the notation of Definition \ref{def:cf-ell-global-cont} (with $B\subset \domcf$ replacing $A$), it follows that $\widetilde{B}_{\leq \ell}\setminus \Gamma$ is disconnected and has a component $\widetilde{B}_{\leq \ell}^1$ which contains a component $\widetilde{B}_{ \ell}^1$ of $\widetilde{B}_{\ell}\setminus \Gamma$ such that $\widetilde{B}_{\ell}^1$ and $\widetilde{B}_{\leq \ell}^1$ satisfy none of the conditions of  Definition \ref{def:cf-ell-global-cont}.
	In particular, $\widetilde{B}_\ell^1 = \widetilde{B}_{\leq \ell}^1$, and $f_\mu(x)\neq 0$ for all $(x,\mu)\in \cl (\widetilde{B}_{\ell}^1)$.
	Since we assume that $+1$ is not a Floquet multiplier of $\gamma$, it follows from the implicit function theorem applied to a Poincar\'{e} map and Lemma \ref{lem:linking-number-lower-semi-cont} that there is a relatively open neighborhood $W\subset Q\times \{\mu_0\}$ of $\Gamma$ with $W\subset \dom(\cf)$ and such that $\gamma$ is the only periodic orbit in $W$ on which $\left|\int\iota_{\mu_0}^*\cf\right| \leq \ell$---except, perhaps, for orbits of very long period.
	
	Let
	\begin{equation*}
	p_0 \coloneqq \inf_{(x,\mu)\in \widetilde{B}_{\leq\ell}^1}\left\{\tau\colon \text{$\tau$ is the period of the orbit through $(x,\mu)$}\right\}
	\end{equation*}
	and
	\begin{equation*}
	p_1 \coloneqq \sup_{(x,\mu)\in \widetilde{B}_{\leq \ell}^1}\left\{\tau\colon \text{$\tau$ is the period of the orbit through $(x,\mu)$}\right\}.
	\end{equation*}
	Note that $p_0 > 0$ \cite{yorke1969periods} and $p_1 < \infty$ by our assumptions.
	If $t\mapsto \Phi_\mu^t(x)$ is the trajectory of $f$ through $(x,\mu)$, we define the function
	\begin{equation}\label{eq:F-def}
	F(x,\mu)\coloneqq \min_{\frac{1}{2}p_0\leq t \leq 2p_1} d(\Phi_\mu^t(x),x) + \left|\ell -\left| \int_{\Phi_\mu^{[0,t]}(x)}\iota_\mu^*\cf\right|\right|
	\end{equation}
	on $\domcf$, where $d(\slot,\slot)$ is the distance associated to some Riemannian metric on $Q$. 
	The zeros of $F$ are points on the images of periodic orbits $\alpha$ of $f$ such that $\ell$ is an integer multiple of $\left|\int_\alpha \iota_\mu^*\cf\right|$.
	Loosely speaking, $F$ measures how close the trajectory through $(x,\mu)$ is to being periodic and satisfying $\ell / \left|\int \iota_\mu^*\cf\right| \in  \N$, for periods between $\frac{1}{2}p_0$ and $2p_1$.
	Since $\Phi$ is continuous in $(x,\mu,t)$, $F$ is continuous in $(x,\mu)$.
	
	For $\epsilon > 0$ let $N_\epsilon\coloneqq \{(x,\mu)\in \domcf\colon F(x,\mu)\leq \epsilon\}$, and let $N^0_\epsilon$ be the component of $N_\epsilon$ containing $\widetilde{B}^1_{\leq \ell}=\widetilde{B}^1_{ \ell}$.
	Choose $\epsilon$ small enough so that
	\begin{enumerate}
		\item\label{item:ng-1} the component $M_\epsilon$ of $N^0_\epsilon \cap (Q\times \{\mu_0\})$ containing $\Gamma$ is a subset of $W$;
		\item\label{item:ng-2} $N^0_\epsilon\setminus W$ is disconnected, and we denote by $N_\epsilon^1$ the component containing $\widetilde{B}_\ell^1$;
		\item\label{item:ng-3} there are no zeros of $f$ in the closure $\cl(N_\epsilon^1)$ of $N_\epsilon^1$ in $\domcf$;
		\item\label{item:ng-4} The closure $\cl(N_\epsilon^1)$ of $N_\epsilon^1$ in $\domcf$  is compact; 
		\item\label{item:ng-m-eps-uniq-orbit-ell} there exists $\rho_1 > 0$ such that, when $\dco(f,g) < \rho_1$, the system $\dot{x} = g(x,\mu_0)$ will have exactly one periodic orbit $\gamma_g$ in $M_\epsilon$ having period $\leq 2 p_1$ and satisfying $\left|\int_{\gamma_g}\iota_{\mu_0}^*\cf \right|\leq \ell$, and this orbit satisfies $\left|\int_{\gamma_g}\iota_{\mu_0}^* \cf\right| = \ell$ and does not have $+1$ as a Floquet multiplier;

		\item\label{item:ng-min-per} there exists $\rho_2 > 0$ such that, when $\dco(f,g)<\rho_2$, the system $\dot{x}=g(x,\mu)$ will have no periodic orbits $\alpha$ contained in $N_\epsilon^1$ satisfying either (i) the period of $\alpha$ is $\leq 2p_1$ and $\left|\int_\alpha \iota_{\mu}^*\cf\right| < \ell$, or (ii) the period of $\alpha$ belongs to $J=(-\infty,\frac{3}{4}p_0]\cup [\frac{4}{3}p_1,\frac{5}{3}p_1]$ and $\left|\int_\alpha \iota_{\mu}^*\cf\right| = \ell$ .		

	\end{enumerate}
	
    By the definition of $F$ and the sentence preceding the definition of $p_0$, it follows that $\Gamma = F^{-1}(0) \cap W$.
    It follows that $F$ attains a minimum $m > 0$ on the compact boundary of a tubular neighborhood of $\Gamma$ in $W$.
	Taking $\epsilon < m$ ensures that \ref{item:ng-1} is satisfied.
	
	We now argue that conditions \ref{item:ng-2}, \ref{item:ng-3}, and \ref{item:ng-4} can be satisfied by taking $\epsilon$ sufficiently small.
	Let $U\subset\domcf$ be an arbitrary precompact open neighborhood of $\widetilde{B}_\ell^1 = \widetilde{B}_{\leq \ell}^1$ such that (i)  $\Gamma \cup \widetilde{B}^1_\ell$ is a connected component of $\cl(U) \cap F^{-1}(0)$ and (ii) $\Gamma \cap \cl((\partial U)\setminus W) = \varnothing$.\footnote{Here is an explicit construction of such a $U$. 
	Let $d(\slot,\slot)$ be the distance induced by any complete Riemannian metric on $Q\times \R$.
    Since $+1$ is not a multiplier of $\Gamma$, we may choose $\delta > 0$ small enough so that the $\delta$-neighborhood $V_\delta\coloneqq \{(x,\mu)\in Q\times \R\colon d((x,\mu),\Gamma) < \delta\}$ of $\Gamma$ contains at most one periodic orbit of $F^{-1}(0) \cap (Q\times \{\mu\})$ for each $\mu$ and satisfies $V_\delta \cap (Q\times \{\mu_0\}) \subset W$.
    It follows that of $V_\delta \setminus W$ is disconnected, and the triangle inequality further implies that the $\frac{\delta}{2}$-neighborhood $\{(x,\mu)\in (Q\times \R)\setminus \cl(W)\colon d((x,\mu),\widetilde{B}_\ell^1) < \frac{\delta}{2}\}$ of $\widetilde{B}_\ell^1$ consists of precisely two connected components, one of which contains $\widetilde{B}^1_\ell$ and is disjoint from $\widetilde{B}^2_{\leq \ell} \cap V_\delta$; define $U$ to be this component. 
    Property (i) follows since, e.g., $\widetilde{B}^1_\ell$ is a connected component of $F^{-1}(0)\cap U$ and (ii) follows since $d((x,\mu),\Gamma) = \frac{\delta}{2}$ for all $(x,\mu) \in \cl((\partial U)\setminus W)$.
    Note that $U$ is precompact since it is bounded and the metric inducing $d(\slot,\slot)$ is complete.}  
	We claim that, for all $\epsilon > 0$ sufficiently small, the component $\widetilde{N}^0_\epsilon$ of $(N_\epsilon^0 \cap \cl(U)) \setminus W$ containing $\widetilde{B}_\ell^1$ is contained in $\interior(U) = U$.
	If not, there is a sequence $(\epsilon_i)_{i\in \N}$ decreasing to zero with $\widetilde{N}^0_{\epsilon_i} \cap (\partial U)\setminus W  \neq \varnothing$ for all $i\in \N$.
	Since $\cl(\widetilde{N}^0_{\epsilon_1}) \supset \cl(\widetilde{N}^0_{\epsilon_2}) \supset \cdots$ is a decreasing sequence of compact sets having nonempty intersection with $(\partial U)\setminus W$, it follows that $\cl((\partial U)\setminus W) \cap \bigcap_{i\in \N} \cl(\widetilde{N}^0_{\epsilon_i})\neq \varnothing$.
	Since the intersection of any decreasing sequence of compact connected subsets of a Hausdorff space is always connected, it follows that $\bigcap_{i\in \N} \cl(\widetilde{N}^0_{\epsilon_i})$ is a connected subset of $\cl(U)\cap F^{-1}(0)$ containing both $\Gamma \cup \widetilde{B}_\ell^1$ and some point in $\cl((\partial U) \setminus W)$.
	But $(\Gamma \cup \widetilde{B}^1_\ell) \cap \cl((\partial U) \setminus W) = \varnothing$ by (ii) and the fact that $U$ is a neighborhood of $\widetilde{B}^1_\ell$, so $\Gamma\cup \widetilde{B}^1_\ell$ is a proper subset of the connected set $\bigcap_{i\in \N} \cl(\widetilde{N}^0_{\epsilon_i})\subset (\cl(U)\cap F^{-1}(0))$; this contradicts (i), so the claim that $\widetilde{N}_\epsilon^0 \subset 
	\interior(U)$ for all sufficiently small $\epsilon$ is proved.
	From this claim it follows that $N^0_\epsilon\setminus W$ is disconnected for sufficiently small $\epsilon$ (with the component $N^1_\epsilon$ of $N^0_\epsilon
	\setminus W$ containing $\widetilde{B}^1_\epsilon$ being equal to $\widetilde{N}_\epsilon^0$), proving that we may choose $\epsilon$ small enough so that \ref{item:ng-2} is satisfied; \ref{item:ng-4} is automatically satisfied since $U$ is precompact and $\cl(N^1_\epsilon)\subset \cl(U)$.
	Since $f$ is bounded away from zero on the compact set $\cl(\widetilde{B}^1_\epsilon)\subset \cl(U)$, and since the precompact neighborhood $U$ satisfying (i-ii) was arbitrary, we may ensure that \ref{item:ng-3} is satisfied by shrinking $U$ so that $f$ is nonzero on $\cl(U)$ and choosing $\epsilon$ sufficiently small so that $\widetilde{N}^0_\epsilon\subset \interior(U)$ as above. 
	Since $U$ satisfying (i-ii) was arbitrary, the above discussion also implies the following fact which we will use: 
	\begin{equation}\label{eq:N1eps-intersect}
	\bigcap_{\epsilon\in(0,\epsilon_0)}N^1_\epsilon = \widetilde{B}^1_\ell,
	\end{equation}
	where $\epsilon_0$ is sufficiently small so that $N^0_\epsilon\setminus W$ is disconnected and $N^1_\epsilon$ is well-defined.
	
	Let $\tau_0$ be the period of $\beta$.
	To show that condition \ref{item:ng-m-eps-uniq-orbit-ell} can be satisfied by taking $\epsilon$ sufficiently small, we argue as follows.
	First, note that the implicit function theorem applied to a Poincar\'{e} map and Lemma \ref{lem:linking-number-lower-semi-cont} imply that there are $\epsilon_1,\rho_{1}$ such that $\dot{x}=g(x,\mu)$ will have only one orbit $\gamma_g$ in $M_{\epsilon_1}$ satisfying $\left|\int_{\gamma_g}\iota_{\mu_0}^*\cf\right| = \ell$ and having period $\leq 2p_1$ when $\dco(f,g)<\rho_{1}$, and by choosing $\epsilon_1$ small enough we can ensure that $\gamma_g$ has no Floquet multiplier equal to $+1$.
	Suppose that there exist sequences $(\epsilon_i)_{i\in\N}$ and $(\rho_{i})_{i\in \N}$ decreasing to zero and $(g_i)_{i\in\N}$ satisfying $\dco(f,g_i)<\rho_i$ and with each $g_i$ having a periodic orbit $\gamma_i$ in $M_{\epsilon_i}$ having period $\leq 2p_1$ and satisfying $\left|\int_{\gamma_i}\iota_{\mu_0}^*\cf\right| < \ell$.
	The images of $\gamma_i$ converge uniformly to the image of $\gamma$ since $\epsilon_i\to 0$ and since the periods of the $\gamma_i$ are bounded, so Lemma \ref{lem:linking-number-lower-semi-cont} implies that $\left|\int_{\gamma}\iota_{\mu_0}^*\cf\right| < \ell$, a contradiction.
	Hence we may ensure the satisfaction of condition \ref{item:ng-m-eps-uniq-orbit-ell} by choosing $\epsilon$ sufficiently small.
		
	To show that condition \ref{item:ng-min-per} may be satisfied, we argue as follows.
	Suppose that, for sequences $(\epsilon_i)_{i\in\N}$ and $(\rho_{i})_{i\in\N}$ decreasing to zero, there exist a sequence of functions $(g_i)_{i\in\N}$ with $\dco(f,g_i)<\rho_{i}$ and a corresponding sequence of orbits $(\alpha_i)_{i\in\N}$ such that $\alpha_i$ is a periodic orbit of $\dot{x} = g_i(x,\mu)$ satisfying at least one of (i, ii) of condition \ref{item:ng-min-per}.
	Choose a point $(x_i,\mu_i)\in N_{\epsilon_i}^1$ on the image of $\alpha_i$ for each $i$.
	Since $N_{\epsilon_i}^1$ is precompact for $i$ large, the $(x_i,\mu_i)$ will have an accumulation point, and \eqref{eq:N1eps-intersect} implies that this accumulation point belongs to the image of a periodic orbit $\alpha$ of $f$ contained in $\widetilde{B}_{\ell}^1 \cup \Gamma$.
    The property (ii) cannot be satisfied, since if the $\alpha_i$ satisfy $\left|\int_{\alpha_i}\iota_{\mu_i}^*\cf\right|=\ell$ for all $i$, then Lemma \ref{lem:linking-number-lower-semi-cont} implies that the periods $\tau_i$ of the $\alpha_i$ satisfy $\tau_i\to\tau \in [p_0,p_1]$, so $\tau_i\not\in J$ for all sufficiently large $i$. 
    And if the property (i) is satisfied, so that 
	$\tau_i\leq 2p_1$ and $\left|\int_{\alpha_i} \iota_{\mu_i}^*\cf \right| < \ell$ for all $i$, then Lemma \ref{lem:linking-number-lower-semi-cont} implies that $\alpha$ satisfies $\left|\int_{\alpha} \iota_{\mu}^*\cf\right| < \ell$, contradicting $\left|\int_\alpha \iota_{\mu}^*\cf\right| = \ell$.
	Hence we may ensure the satisfaction of condition \ref{item:ng-min-per} by choosing $\epsilon$ sufficiently small. 

    Following the choice of $\epsilon$, we let $g_\epsilon\in \K$ be a generic family sufficiently close to $f$ in the weak $C^1$ topology so that
    \begin{enumerate}
    	\item\label{item:ng-g-cond-f-g-eps-close} $\dco(f,g_\epsilon)<\min\{\rho_1,\rho_2\}$;
    	\item\label{item:min-g-not-too-small} $\min_{(x,\mu)\in\cl(N_\epsilon^1)}\norm{g(x,\mu)} > \frac{1}{2} \min_{(x,\mu)\in\cl(N_\epsilon^1)}\norm{f(x,\mu)}$; and
    	\item\label{item:ng-g-cond-2} $\max_{(x,\mu)\in \cl(N_\epsilon^1)}|F-G_\epsilon|<\frac{\epsilon}{2}$, where $G_\epsilon$ is defined analogously to $F$ for solutions of $\dot{x} = g_\epsilon(x,\mu)$ (again using $\ell$ and $\frac{1}{2}p_0\leq t \leq 2p_1$).
    \end{enumerate}
    Condition \ref{item:ng-g-cond-f-g-eps-close} can be satisfied since the set of $C^5$ vector field families is dense in the space of $C^1$ vector field families \cite[Ch.~2.2,~Ex.~3]{hirsch1976differential}, and $\K$ is dense in the space of $C^5$ families \cite[Thm~A]{sotomayor1973generic} as discussed in \S \ref{sec:background-generic-gamilies}.
    Similar reasoning implies that condition \ref{item:min-g-not-too-small} can be satisfied, using also the fact that $f$ has no zeros in the compact set $\cl(N_\epsilon^1)$.
    To show that $g_\epsilon$ can be chosen to satisfy condition \ref{item:ng-g-cond-2}, we argue as follows. 
    Let $\Psi^t_\mu(x)$ be the solution of 
    \begin{equation}\label{eq:ode-g}
    \dot{x}=g_\epsilon(x,\mu)
    \end{equation}
    through $(x,\mu)$.
    Since $\cl(N_\epsilon^1)$ is compact and since flows depend continuously on their vector fields, $g_\epsilon$ can be chosen so that
    $$\left|\left(d(\Phi_\mu^t(x),x) + \left|\ell -\left| \int_{\Phi_\mu^{[0,t]}(x)}\iota_{\mu}^*\cf\right|\right|\right) - \left(d(\Psi_\mu^t(x),x) + \left|\ell -\left| \int_{\Psi_\mu^{[0,t]}(x)}\iota_{\mu}^*\cf\right|\right|\right)\right| < \frac{\epsilon}{2}$$
    for all $(x,\mu)\in \cl (N_\epsilon^1)$ and $\frac{1}{2}p_0\leq t\leq 2p_1$.
    Since in general two real-valued functions $P, Q$ uniformly satisfying $|P-Q|<\frac{\epsilon}{2}$ must also satisfy $|\min P - \min Q| < \frac{\epsilon}{2}$, if follows that $|F-G_\epsilon|<\frac{\epsilon}{2}$ uniformly on $\cl(N_\epsilon^1)$.
    
    Let $\gamma_\epsilon$ be the unique solution of \eqref{eq:ode-g}  in $M_\epsilon$ having period $\leq 2p_1$ and satisfying $\left|\int_{\gamma_\epsilon}\cf\right| = \ell$, let $\Gamma_\epsilon$ be the image of $\gamma_\epsilon$, and define the sets $\widetilde{A}_{\epsilon,\ell}, \widetilde{A}_{\epsilon,\leq \ell}$ as in Definition \ref{def:cf-ell-global-cont} (with $A_\epsilon$ replacing $A$).
    Define $Y\coloneqq (\widetilde{A}_{\epsilon, \ell}\setminus \Gamma_\epsilon)\cap \cl(N_\epsilon^1)$ and $Z\coloneqq (\widetilde{A}_{\epsilon, \leq\ell}\setminus \Gamma_\epsilon)\cap \cl(N_\epsilon^1)$.
    Because $+1$ is not a Floquet multiplier of $\gamma_\epsilon$, $Y$ and $Z$ are not empty.
    We want to show that $Y$ is contained in $\interior(N_\epsilon^1)$ and that $Y = Z$ .
    We begin with the first statement.
    Now $Y$ can only escape from $\interior(N_\epsilon^1)$ through $\partial N_\epsilon^1\subset \domcf$, i.e., (i) through $X\coloneqq \partial N_\epsilon^0 \cap N_\epsilon^1$ or (ii) through $M_\epsilon$.
    We discuss each case separately.
    
    Suppose $(x,\mu)\in Y \cap X$.
    Then by condition \ref{item:ng-g-cond-2} on $g_\epsilon$, $(x,\mu)$ must be on an orbit with period $\tau$, where $\tau\in (-\infty,\frac{1}{2}p_0)\cup(2p_1,\infty)$.
    By taking $\epsilon$ smaller, we may assume that $\gamma_\epsilon$ is sufficiently near $\gamma$ that the period of $\gamma_\epsilon$ belongs to $(\frac{3}{4}p_0,\frac{4}{3}p_1)$.
    Suppose $\tau > 2 p_1$.
    Then $Y$ contains orbits with periods less than $\frac{4}{3}p_1$ and orbits with periods greater than $2p_1$.
    Since periods on a branch of $(\int\iota_\mu^*\cf)$-constant orbits of a generic family change continuously, there must be an orbit $\alpha$ with image in $Y$ and with period in $[\frac{4}{3}p_1,\frac{5}{3}p_1]$, and no orbit on the ``path'' from $\gamma_\epsilon$ to $\alpha$ with period greater than $\frac{5}{3}p_1$.
    But then it is easily seen that the image of $\alpha$ is contained in $N_\epsilon^1$, contradicting condition \ref{item:ng-min-per} on the choice of $\epsilon$.
    A similar argument shows that $\tau < \frac{1}{2}p_0$ would also contradict \ref{item:ng-min-per}.
    Thus $Y$ and $X$ are disjoint. 
    By condition \ref{item:ng-m-eps-uniq-orbit-ell} on the choice of $\epsilon$, $\gamma_\epsilon$ is the only periodic orbit of $g_\epsilon$ with image in $M_\epsilon$ having period less than or equal to $2 p_1$ and satisfying $\left|\int_{\gamma_\epsilon}\iota_{\mu_0}^*\cf \right|= \ell$.
    By condition \ref{item:ng-min-per} on the choice of $\epsilon$, $Y$ contains no orbits with periods greater than $\frac{4}{3}p_1$.
    Thus $Y$ and $M_\epsilon$ are also disjoint. 
    It follows that $\widetilde{A}_{\epsilon, \ell}\setminus \Gamma_\epsilon$ (and hence also $\widetilde{A}_{\epsilon, \leq\ell}\setminus \Gamma_\epsilon$) is disconnected, with two components $\widetilde{A}_{\epsilon, \ell}^2$ and $Y = \widetilde{A}_{\epsilon, \ell}^1$.
    
    The fact that $Y\subset \interior(N_\epsilon^1)$ contains no periods larger than $\frac{4}{3}p_1$ also implies that $Y = Z$, for the orbit $\alpha$ through any limit point of $Y$ in $Z\setminus Y$ would be contained in $\cl(N_\epsilon^1)=N^1_\epsilon\cup M_\epsilon$, have period less than or equal to $\frac{2}{3}p_1$, and satisfy $\left|\int_\alpha\iota_{\mu}^*\cf\right| < \ell$, contradicting either condition \ref{item:ng-m-eps-uniq-orbit-ell} or condition \ref{item:ng-min-per} on the choice of $\epsilon$.
    It follows that the connected component $\widetilde{A}_{\epsilon, \leq \ell}^1$ of $\widetilde{A}_{\epsilon, \leq\ell}\setminus \Gamma_\epsilon$ satisfies $ \widetilde{A}_{\epsilon, \leq \ell}^1 = \widetilde{A}_{\epsilon, \ell}^1 = Y = Z$.
    
	To summarize, we have shown that $\widetilde{A}_{\epsilon,\leq \ell} \setminus \Gamma_\epsilon = \widetilde{A}_{\epsilon,\leq\ell}^1 \cup \widetilde{A}_{\epsilon,\leq\ell}^2$ is disconnected, that $\widetilde{A}_{\epsilon,\leq\ell}^1 = \widetilde{A}_{\epsilon,\ell}^1$, and that  $\widetilde{A}_{\epsilon,\ell}^1$ is contained in the compact set $\cl(N_\epsilon^1)\subset \domcf$ which contains no generalized centers.
	Additionally, the periods $\widetilde{A}_{\epsilon,\ell}^1$ have the uniform upper bound $\frac{4}{3}p_1$.
	But this implies that $\gamma_\epsilon$ is not $(\cf,\ell)$-globally continuable, contradicting Lemma \ref{lem:cf-ell-cont-generic-fam}.
	This contradiction implies that $\gamma$ must in fact be $(\cf,\ell)$-globally continuable and completes the proof.
\end{proof}

    Armed with Theorem \ref{th:main-thm}, we now proceed to prove our main results on existence of periodic orbits.
    We will use the following lemma to convert data from a closed 1-form and a vector field into a priori bounds on the periods of periodic orbits.

	\begin{Lem}\label{lem:cf-period-bound}
		Let $M$ be a smooth manifold and let $\cf$ be a $C^1$ 1-form on $M$.
		Let $\gamma$ be a periodic orbit of a $C^1$ vector field $f \colon M\to \T M$.
		Assume that there exists $\epsilon > 0$ such that $\cf(\dot{\gamma}(t))\geq \epsilon$ for all $t$.
		Then $\ell\coloneqq \int_{\gamma}\cf > 0$, and the period $\tau$ of $\gamma$ satisfies $$\tau \leq \frac{\ell}{\epsilon}.$$
	\end{Lem}
	\begin{proof}
		We have
		\begin{equation}
		\ell = \int_\gamma \cf  =\int_{0}^\tau\cf(\dot{\gamma}(t))\, dt \geq \epsilon \tau > 0,
		\end{equation}
		with the first inequality following since $\cf(\dot{\gamma}) \geq \epsilon$.
		This completes the proof. 
	\end{proof}
	
	We now prove Theorem \ref{th:existence}, our first periodic orbit existence result.
	Theorem \ref{th:existence} is fairly general, and it follows straightforwardly from Theorem \ref{th:main-thm} and Lemma \ref{lem:cf-period-bound}.
	We continue to identify $\Gamma$ with $\Gamma\times \{\mu_0\}$ when there is no risk of confusion in the following.
	  
	Given a subset $X\subset Q\times \R$ and any interval $J\subset \R$, we use the notation $X_J\coloneqq X\cap (Q\times J)$  in Theorems \ref{th:existence} and \ref{th:hopf} below.   
	
	\ThmExistence*
	\begin{proof}
		We prove the result in the case that $\mu^* < \mu_0$, with the proof for the case $\mu^* > \mu_0$ being similar.
		
		Assume, to obtain a contradiction, that there exists $\mu_1 > \mu_0$ such that $\widetilde{A}_\ell^1\cap (Q\times \{\mu_1\})=\varnothing$.
		Then connectedness of $\widetilde{A}_\ell^1$ and hypotheses \ref{item:th-ex-1} and \ref{item:th-ex-2} imply that $\widetilde{A}_\ell^1$ is contained in the set $\cc_{[\mu^*,\mu_1]}$.
		Since $\cc_{[\mu^*,\mu_1]}$ is compact by hypothesis \ref{item:th-ex-4}, hypothesis \ref{item:th-ex-3} implies that there is $\epsilon > 0$ such that $\iota_\mu^*\cf(f_\mu(x)) \geq \epsilon$ for all $(x,\mu) \in \cc_{[\mu^*,\mu_1]}$.
		Hence Lemma \ref{lem:cf-period-bound} implies that the periods of orbits in $\widetilde{A}_{\ell}^1$ are uniformly bounded above by $\frac{\ell}{\epsilon}$. 
		By assumption we also have $\widetilde{A}_{\leq \ell}^1 = \widetilde{A}_{\ell}^1$, and $\cl(\widetilde{A}_{\ell})$ contains no equilibria and hence no generalized centers by hypotheses \ref{item:th-ex-2} and \ref{item:th-ex-3}.
		But Theorem \ref{th:main-thm} implies that $\gamma$ is $(\cf,\ell)$-globally continuable, so we have a contradiction.
		This completes the proof.   
	\end{proof}
	We now use Theorem \ref{th:existence} to formalize a rather specific argument involving Theorem \ref{th:existence} and a Hopf bifurcation, which we have used to prove the existence of periodic orbits in applications (see \S \ref{sec:examples}).
	While the statement appears rather complicated, the upshot is that we do not have to repeat this argument in each of our individual examples.
	
	Given a subset $X\subset Q\times \R$ and any interval $J\subset \R$, we again use the notation $X_J\coloneqq X\cap (Q\times J)$  in Theorem \ref{th:hopf} below.  
		
	\ThHopf*
	\begin{proof}
		We prove the theorem for the case that the Hopf bifurcation is supercritical and $\mu^*\leq \mu_c$, with the other three cases being similar.
		
		The proof of the Hopf bifurcation theorem \cite[Thm~3.4.2]{guckenheimer1983nonlinear} implies that there exists $\delta > 0$ and a $\mu$-dependent two-dimensional center manifold $W^c_\mu$ for $\mu\in (\mu_c-\delta,\mu_c+\delta)$ such that (i) $(x_c,\mu_c) \in \iota_{\mu_c}(W^c_{\mu_c})\cap M$, (ii) the orbit at each $\mu$ on the bifurcating branch of periodic orbits is contained in $W^c_\mu$, and (iii) the image of the periodic orbit at $\mu$ on the bifurcating branch tends to $x_c$ uniformly as $\mu\to \mu_c$. 
		Since $\T_{x_c}W^c_{\mu_c} = E^c$, after shrinking $W^c_{\mu_c}$ if necessary it follows from hypothesis \ref{item:th-hopf-7} that $\iota_{\mu_c}(W^c_{\mu_c})$ intersects $M$ transversely.
		Hence (by an implicit function theorem argument) there exist local coordinates $(y,z,\mu)$ on a neighborhood of $(x_c,\mu_c)\subset Q\times \R$ in which $W^c_\mu \times \{\mu\} = \{(y,0,\mu)\}$ and  $M= \{(0,z,\mu)\}$. 
		This fact, (ii--iii) above, and hypothesis \ref{item:th-hopf-6} imply that, for $\mu>\mu_c$ sufficiently close to $\mu_c$, the disk $D_\mu \subset W^c_\mu$ bounded by the image of the bifurcating periodic orbit $\gamma_\mu$ intersects $\iota_{\mu}^{-1}(M)$ exactly once (and this intersection is transverse by hypothesis \ref{item:th-hopf-7}).
		Fixing such a $\mu_0>\mu_c$ sufficiently close to $\mu_c$ and defining $\gamma\coloneqq \gamma_{\mu_0}$, it follows that the intersection number of $\iota_{\mu_0}(\interior (D_{\mu_0}))$ with $M$ is $\pm 1$.
		Since $\cf$ is Poincar\'{e} dual to $N\setminus M$ in $(Q\times \R)\setminus M$, it can be shown using the technique of \cite[pp.~231--234]{bott1982differential} that $\int_{\gamma}\iota_{\mu_0}^*\cf = \int_{\iota_{\mu_0}\circ \gamma} \cf$ is equal to this intersection number up to a sign:\footnote{In \cite{bott1982differential} the authors work with $C^\infty$ forms, whereas we assume $\cf \in C^1$, but there is no issue since every $C^1$ closed form is cohomologous to a $C^\infty$ closed form \cite[pp.~61--70]{deRham1984differentiable}.}
		\begin{equation}\label{eq:th-hopf-int-cf-1}
		\left|\int_{\gamma}\iota_{\mu_0}^*\cf\right| = 1.
		\end{equation}
		Note that, since $(x_c,\mu_c)\in \interior (\cc)$ by hypothesis \ref{item:th-hopf-5}, we may assume that $\mu_0$ is chosen sufficiently close to $\mu_c$ that $\Gamma\subset \interior(\cc)$.
		By the proof of the Hopf bifurcation theorem we may furthermore assume that $\mu_0$ is chosen sufficiently close to $\mu_c$ that $\gamma$ is hyperbolic, so in particular $+1$ is not a Floquet multiplier of $\gamma$.
		
		Let $\domcf\coloneqq (Q\times \R)\setminus M$, and let $A, \widetilde{A}_1, \widetilde{A}_{\leq 1}$ be the components containing $\Gamma$ as in Definition \ref{def:cf-ell-global-cont} with $\ell = 1$ (identifying $\Gamma$ with $\Gamma\times \{\mu_0\}$).
		Note that the periodic orbit component $A$ of $f|_{\domcf}$ is also a periodic orbit component of $f$ due to hypotheses \ref{item:th-hopf-1},  \ref{item:th-hopf-2}, and  \ref{item:th-hopf-6}.
		The proof of the Hopf bifurcation theorem implies the existence of a compact neighborhood $K_0\subset Q\times \R$ of $(x_c,\mu_c)$ containing a connected subset $B\subset \widetilde{A}_1\cap K_0$ of nonstationary periodic orbits of $f$ such that (i) $\Gamma\subset B$, (ii) all periods of $B$ are close to the period of $\Gamma$, (iii) any other periodic orbits in $K_0\setminus B$ have very large period, and (iv) $B\setminus \Gamma$ consists of two connected components $B^1, B^2$ with $B^2\subset \interior(K_0)$ and the closure of $B^2$ containing $(x_c,\mu_c)$.
		Due to (iii) above, hypothesis \ref{item:th-hopf-3}, and Lemma \ref{lem:cf-period-bound}, it follows that $B = \widetilde{A}_1\cap K_0 = \widetilde{A}_{\leq 1}\cap K_0$.
		Hence 
		both $\widetilde{A}_{ 1}\setminus \Gamma$ and $\widetilde{A}_{\leq 1}\setminus \Gamma$ are disconnected with two connected components.
		As in Definition \ref{def:cf-ell-global-cont}, let $\widetilde{A}_1^1, \widetilde{A}_1^2$ and $\widetilde{A}_{\leq 1}^1, \widetilde{A}_{\leq 1}^2$ denote the connected components of $\widetilde{A}_{1}\setminus \Gamma$ and $\widetilde{A}_{\leq 1}\setminus \Gamma$, labeled so that $B^i\subset \widetilde{A}_1^i \subset \widetilde{A}_{\leq 1}^i$.
		Since $\cf$ is the Poincar\'{e} dual of a submanifold, its integral around any periodic orbit is an integer (which is nonzero by Lemma \ref{lem:cf-period-bound}), so we automatically have $$1 = \min\left\{\left|\int_{\alpha_\mu} \iota_\mu^*\cf\right|\colon \text{$\alpha_\mu$ is a periodic orbit with image in $A$}\right\}.$$ 
		In particular, it follows that $\widetilde{A}_{1}^i = \widetilde{A}_{\leq 1}^i$ for both $i=1,2$.
		
        By the above paragraph, there exists a neighborhood $U_0\subset K_0$ of $(x_c,\mu_c)$ such that $\widetilde{A}_1^1 \cap U_0 = \varnothing$.
		Hypothesis \ref{item:th-hopf-2} implies that $\widetilde{A}_1^1\subset \interior(\cc)_{(\mu^*,\infty)}$, so we have $\widetilde{A}_1^1\subset (\interior(\cc)\setminus U_0)_{(\mu^*,\infty)}$.
        Since $\domcf = (Q\times \R)\setminus M$, we have $\widetilde{A}_1^1 \cap M = \varnothing$ by definition.
        We now show that there is furthermore a neighborhood $U_1\subset Q\times \R$ of $M$ such that $\widetilde{A}_1^1 \cap U_1 = \varnothing$.
        If this were not true, then there is a sequence $(x_i,\mu_i)_{i\in \N}$ in $\widetilde{A}_1^1$ with\footnote{If not, then (since $Q\times \R$ is first countable) each $(x,\mu)\in M$ has an open neighborhood $U_{x,\mu}$ satisfying  $U_{x,\mu}\cap \widetilde{A}_{1}^1 = \varnothing$. 
       	But then $\bigcup_{(x,\mu)\in M}U_{x,\mu}$ is an open neighborhood of $M$ having empty intersection with $\widetilde{A}_{1}^1$, a contradiction.} $(x_i,\mu_i)\to (x,\mu)\in M\setminus U_0$.
        This sequence must be contained in $(\cc\setminus M)_{[\mu^*,\mu_1]}$ for some $\mu_1$, so hypothesis \ref{item:th-hopf-3} and Lemma \ref{lem:cf-period-bound} imply that the periods of the orbits through $(x_i,\mu_i)$ are uniformly bounded above by $\frac{1}{\epsilon}$ for some $\epsilon > 0$.
        This implies that $(x,\mu)$ is either a generalized center or lies on a nonstationary periodic orbit; hypothesis \ref{item:th-hopf-6} rules out the latter option, so $(x,\mu)$ is a generalized center. 
        But hypothesis \ref{item:th-hopf-2} further implies that $(x,\mu)\in \cc_{[\mu^*,\infty)}$, and this contradicts hypothesis \ref{item:th-hopf-5}.
        Hence there exists a neighborhood $U_1\subset Q\times \R$ of $M$ with $\widetilde{A}_1^1 \cap U_1 = \varnothing$, as desired.
		
		Define the set $\widetilde{\cc}\coloneqq \cc \setminus U_1$.
		By the last paragraph, $\widetilde{A}_1^1\subset \widetilde{\cc}$.
		Additionally, $\widetilde{A}_1^1\cap (Q\times \{\mu^*\}) = \varnothing$ by hypothesis \ref{item:th-hopf-1}, $\iota_\mu^*\cf(f_\mu(x))>0$ for all $(x,\mu)\in\widetilde{\cc}$ by hypothesis \ref{item:th-hopf-3}, and $\widetilde{\cc}$ contains no generalized centers for $f$ by \ref{item:th-hopf-5}.
		Finally, for every $\mu \geq \mu^*$, $$\widetilde{\cc}_{ [\mu^*,\mu]} = \cc_{[\mu^*,\mu]} \setminus U_1$$ is compact by hypothesis \ref{item:th-hopf-4}.
		Since we have already shown that $\widetilde{A}_1^1 = \widetilde{A}_{\leq 1}^1$ and that $+1$ is not a Floquet multiplier of $\gamma$, it follows that the hypotheses of Theorem \ref{th:existence} are satisfied with $\ell = 1$ and $\widetilde{\cc}$ playing the role of $\cc$.
        Hence $\widetilde{A}_{1}^1 \cap (Q\times \{\mu\}) \neq \varnothing$ for all $\mu > \mu_0$.
        In particular, $f_\mu$ has a nonstationary periodic orbit contained in $(\cc\setminus M)_{\{\mu\}}$ for all $\mu \geq \mu_0$.   
        Since $\mu_0 > \mu_c$ was arbitrary, it follows that $f_\mu$ has a nonstationary periodic orbit contained in $(\cc\setminus M)_{\{\mu\}}$ for all $\mu > \mu_c$ as well.
        This completes the proof.		
	\end{proof}

   \section{Examples}\label{sec:examples}
   In this section, we illustrate our results by proving periodic orbit existence results for the repressilator \eqref{eq:repressilator-intro} in \S \ref{sec:repressilator} and for the Sprott system \eqref{eq:sprott-intro} in \S \ref{sec:sprott}. 
   We begin with some preliminary discussion relevant for both systems.
   Both systems admit the symmetry $(x,y,z)\mapsto (y,z,x)$, and we discuss some consequences of this fact in \S \ref{sec:symmetry}.
   In \S \ref{sec:dtheta} we define a 1-form $d\theta$---to be used in the proofs for both systems---and briefly discuss some of its properties.

   \subsection{Basic symmetry considerations}\label{sec:symmetry}
   	In the remainder of \S \ref{sec:examples}, we use the notation $\1\coloneqq (1,1,1)$ and $\textbf{x}\coloneqq (x,y,z)$.   
   
	Define the linear permutation symmetry $\sigma\colon \R^3 \to \R^3$ via $\sigma(x,y,z)\coloneqq (y,z,x)$. 
	Letting $f_\mu\colon \R^3 \to \R^3$ denote either the repressilator \eqref{eq:repressilator-intro} or Sprott \eqref{eq:sprott-intro} vector fields, we see that $\sigma_* f_\mu \coloneqq \D \sigma \circ f_\mu \circ \sigma^{-1} = f_\mu.$
	It follows that $\sigma$ commutes with the (local) flow $\Phi_\mu$ of $f_\mu$ and therefore maps solution curves to solution curves.
	Since the diagonal $\Delta \coloneqq \textnormal{span}\{\1\}$ is the fixed point set of $\sigma$, it follows that $\Delta$ is $\Phi_\mu$-invariant since, for all $p \in \Delta$, $\sigma \circ \Phi^t_\mu(p) = \Phi^t_\mu\circ \sigma (p) = \Phi^t_\mu(p)$. 	
	Since $\sigma^3 = \id_{\R^3}$, the dynamics have a $\Z_3$ symmetry group whose action on $\R^3$ is generated by $\sigma$.
	It follows that any invariant set
	is either fixed by $\sigma$ or is one member of a family of three distinct invariant sets permuted by $\sigma$.

	The linear map $\sigma \in \SO(3) \subset \GL(3,\R)$ is a rotation having the unique finest $\sigma$-invariant splitting $\Delta \oplus \Delta^\perp = \R^3$.
	Identifying $\sigma$ with $\D \sigma$, it follows that, for any $\x \in \Delta$, the matrix representing $\sigma$ commutes with the matrix $\D_{\x}f_\mu$; assume $\x\in \Delta$ in the following.
	$\sigma$-invariance of the splitting $\Delta\oplus \Delta^\perp$ implies that $\D_\x f_\mu \Delta$ and $\D_\x f_\mu \Delta^\perp$ are $\sigma$-invariant subspaces, since
	\begin{equation*}
	\begin{split}
	\sigma \circ \D_\x f_\mu(\Delta) &= \D_\x f_\mu\circ \sigma(\Delta) = \D_\x f_\mu \Delta\\	
	\sigma \circ \D_\x f_\mu(\Delta^\perp) &= \D_\x f_\mu\circ \sigma(\Delta^\perp) = \D_\x f_\mu \Delta^\perp.
	\end{split}
	\end{equation*} 
	In particular, if $\D_\x f_\mu$ is invertible then $\D_\x f_\mu \Delta\oplus \D_\x f_\mu \Delta^\perp=\R^3$ is a $\sigma$-invariant splitting of $\R^3$ into one and two-dimensional subspaces, so uniqueness of the finest $\sigma$-invariant splitting $\Delta \oplus \Delta^\perp=\R^3$ implies that 
	\begin{equation}\label{eq:Df-Delta-eq-Delta}
	\D_\x f_\mu \Delta = \Delta, \qquad \D_\x f_\mu \Delta^\perp = \Delta^\perp.
	\end{equation}
	If $(x,\mu)$ is a point of generic Hopf bifurcation for $f_\mu$, then $\D_\x f_\mu$ is invertible and there is a unique finest $\D_\x f_\mu$-invariant splitting $ E\oplus E^c=\R^3$ into one and two-dimensional subspaces, with $E^c$ the two-dimensional center subspace.
	Uniqueness of the finest $\D_\x f_\mu$-invariant splitting and \eqref{eq:Df-Delta-eq-Delta} therefore imply that 
	\begin{equation}\label{eq:Ec-eq-Delta-perp}
	E^c = \Delta^\perp.
	\end{equation}

    \subsection{A closed $1$-form}\label{sec:dtheta}
	With respect to the orthogonal splitting $$\R^3 = \Delta \oplus \Delta^\perp,$$
	we may write any $\x\in \R^3$ uniquely as $$\x = \x_{\parallel} + \x_{\perp}$$
	with $\x_{\parallel}\in \Delta$ and $\x_\perp \in \Delta^\perp$.
	A direct computation shows that $\norm{\x_{\perp}}^2= \frac{2}{3}(\norm{\x}^2-\langle \x, \sigma(\x)\rangle)$, where $\norm{\slot}$ and $\langle \slot, \slot \rangle$ are the Euclidean norm and inner product.
	
	We now define a 1-form $d\theta$ on $\R^3\setminus \Delta$:
	\begin{equation}\label{eq:dth-def}
	d\theta \coloneqq \frac{\sqrt{3}}{2}\frac{(z-y)dx + (x-z)dy + (y-x)dz}{\norm{\x}^2-\langle \x, \sigma(\x)\rangle} = \frac{1}{\sqrt{3}}\frac{(z-y)dx + (x-z)dy + (y-x)dz}{\norm{\x_{\perp}}^2}.
	\end{equation}
	It can be shown that $d\theta$ is closed. 
	In fact, choose orthogonal coordinates $(u,v,w)$ adapted to the splitting $\R^3 = \Delta^\perp \oplus \Delta$ so that $(u,v)$ are coordinates for $\Delta^\perp$ and $w$ is a coordinate for $\Delta$.
	Then it can be shown that $d\theta$ is equal to the standard ``angle 1-form'' about the $w$-axis in these coordinates:
	$$d\theta = \frac{udv-vdu}{u^2+v^2}.$$ 
	Note that $\frac{d\theta}{2\pi}$ is Poincar\'{e} dual to $\{\x\colon x = y \text{ and } z < x\}$ on $\R^3\setminus \Delta$ (see \cite[Ex.~5.16(a)]{bott1982differential}).

   \subsection{The repressilator: existence of periodic orbits}\label{sec:repressilator}
   In this section we apply our theory to the repressilator (see \cite{elowitz2000synthetic, buse2009existence}) which models a synthetic
   genetic regulatory network consisting of a ring oscillator. 
   We consider here
   the three-dimensional reduced-order model studied in \cite{buse2009existence,buse2010dynamical} and prove existence of nonstationary periodic orbits. 
   This result was already established in \cite{buse2009existence}, but our proof is new.
   We note that our proof does not use techniques specific to the class of monotone cyclic feedback systems \cite{mallet1990poincare} to which this repressilator model belongs. 
      
   Fix $s > 0$ and consider the one-parameter family of ODEs on $\R^3$ given by 
   	\begin{equation}\label{eq:repressilator}
   	\begin{split}
   	\dot{x} &= \frac{\mu}{1+y^s} - x\\
   	\dot{y} &= \frac{\mu}{1+z^s} - y\\
   	\dot{z} &= \frac{\mu}{1+x^s} - z,
   	\end{split}
   	\end{equation}
   	with parameter $\mu\in\R$.
   	Let $\R^3_+\coloneqq \{(x,y,z)\in \R^3\colon x,y,z \geq 0\}$ be the closed positive orthant.
   	Notice that, for any $s,\mu > 0$, $\R^3_+$ is positively invariant for the flow $f_{s,\mu}$ of \eqref{eq:repressilator}.
   	Furthermore, since $0< \frac{\mu}{1+r^s}\leq \mu$ whenever $s,\mu > 0$ and $r\geq 0$, it follows that the cube $\{(x,y,z)\colon 0 \leq x,y,z \leq \mu \}$ is positively invariant and attracts every initial condition $\x\in \R^3_+$.
   	It follows that the same is true of the interior of the smaller cube
   	 $$K_\mu\coloneqq \left\{(x,y,z)\colon \frac{\mu}{2+\mu^s}\leq x,y,z \leq \mu\right \}$$
   	 since $\frac{\mu}{1+r^s} \geq \frac{\mu}{1+(\mu+\epsilon)^s} >  \frac{\mu}{2+\mu^s}$ whenever $r \leq \mu+\epsilon$ and $\epsilon > 0$ is sufficiently small; in particular, $\partial K_\mu$ immediately flows into $\interior(K_\mu)$. 
   	
   	We now prove that \eqref{eq:repressilator} has a periodic orbit for all $s> 2$ and $\mu > \mu_c$, where $\mu_c$ is defined below.
   	To do this, we simply verify that \eqref{eq:repressilator} satisfies all of the hypotheses of Theorem \ref{th:hopf}.
   	We delay the (slightly lengthier) verification of hypothesis \ref{item:th-hopf-3} of Theorem \ref{th:hopf} to \S \ref{sec:rotation-repress} below.
   	
   	\begin{Th}\label{th:repress}
   		Let $f_{s,\mu} = f(\slot,s,\mu)$ be the repressilator vector field \eqref{eq:repressilator}.
   		Fix $s> 2$ and define $$\mu_c\coloneqq \left(\frac{2}{s-2}\right)^{\frac{s+1}{s}} + \left(\frac{2}{s-2}\right)^{\frac{1}{s}}.$$   		Then for all $\mu > \mu_c$, $f_{s,\mu}$ has a periodic orbit contained in the cube $K_\mu$.
   	\end{Th}
   \begin{Rem}\label{rem:classical-difficulty-repress}
For the reasons explained in Remark~\ref{rem:classical-result-difficulties}, it seems very difficult to prove Theorem~\ref{th:repress} directly using either of the classical continuation results Proposition~\ref{prop:p-global-cont} (\cite[Thm 4.2, Thm 2.2]{mallet1982snakes,alligood1983index}) or Proposition~\ref{prop:global-cont} (\cite[Thm~3.1]{alligood1984families}).
   \end{Rem}
   	
   	\begin{proof}
   		Define $\cc \coloneqq \{(\x,\mu)\in \R^3\times \R\colon \mu \geq 0 \text{ and } \x \in K_{\mu} \}$, $N\coloneqq \{\x\colon x = y \text{ and } z \leq x\}\times \R$, and $M\coloneqq \Delta \times \R = \partial N.$ 
   		
   		Since the origin is exponentially stable for $f_{s,0}$, there exists $\mu^* > 0$ such that, for all $0 <\mu \leq\mu^*$, $f_{s,\mu}$ has no periodic orbits whose images intersect\footnote{Proof: fix $s> 2$. It is shown in \cite{buse2009existence} that \eqref{eq:repressilator} has a unique equilibrium $\x_{s,\mu}\in K_\mu$ for all $\mu\geq 0$ which depends continuously on $\mu$; define $V_{s,\mu}(\x)\coloneqq \frac{1}{2}\norm{\x-\x_{s,\mu}}^2$. Applying Taylor's theorem to $f_{s,\mu}$ about the point $\x_{s,\mu}$ shows that the derivative of $V_{s,\mu}$ along the flow of $f_{s,\mu}$ is $\dot{V}_{s,\mu}(\x) = \langle \x-\x_{s,\mu},f_{s,\mu}(\x)\rangle= -\norm{\x-\x_{s,\mu}}^2 + R_s(\x,\mu)\norm{\x-\x_{s,\mu}}^2$, where $R_s$ is continuous and satisfies $R_s(\slot,0)\equiv 0$.
   		Hence $|R_s|<\frac{1}{2}$ on some neighborhood $U$ of $\R^3\times \{0\}$, so $\dot{V}_{s,\mu}(\x)\leq -\frac{1}{2}\norm{\x-\x_{s,\mu}}^2$ for all $(\x,\mu)\in U$.
   		Continuity of $\mu\mapsto \x_{s,\mu}$ therefore implies that there are $\mu_0,\epsilon_0 > 0$ such that, for all $0<\epsilon \leq \epsilon_0$ and $0< \mu \leq \mu_0$, $\dot{V}_{s,\mu} \leq -\frac{1}{2}\norm{\x-\x_{s,\mu}}^2$ on the closed ball $B_\epsilon(\x_{s,\mu})$ of radius $\epsilon$ centered at $\x_{s,\mu}$.
   	    For such values of $\epsilon,\mu$ it follows that $B_\epsilon(\x_{s,\mu})$ is contained in the stability basin of $\x_{s,\mu}$ and so does not meet the images of any nonstationary periodic orbits. 
   	    Finally, defining $\mu^*\coloneqq \min\{\frac{\epsilon_0}{2\sqrt{3}},\mu_0\}$ suffices to prove the claim since $K_{\mu}\subset \{0\leq x,y,z \leq \mu\}\subset B_{2\sqrt{3}\mu}(\x_{s,\mu}) \subset B_{\epsilon_0}(\x_{s,\mu})$ for $0<\mu< \mu^*$, with the second inclusion following since $\x_{s,\mu}\in K_\mu$. (Something stronger is actually true: $\x_{s,\mu}$ is globally asymptotically stable for $|\mu|$ sufficiently small, but we will not need this. This fact follows from \cite[Cor.~2.3]{smith1999perturbation}.)} $K_{\mu}$, so in particular hypothesis \ref{item:th-hopf-1} of Theorem \ref{th:hopf} is satisfied.
   		If $\mu > 0$ and $\Phi_{s,\mu}$ is the flow of $f_{s,\mu}$, then every initial condition $\x\in \partial K_\mu$ satisfies $\Phi_{s,\mu}^t(\x)\in \interior(K_\mu)$ for all $t > 0$, so no periodic orbits of $f_{s,\mu}$ intersect $\partial K_{\mu}$; hence hypothesis \ref{item:th-hopf-2} of Theorem \ref{th:hopf} is satisfied.
   		The compactness hypothesis \ref{item:th-hopf-4} is satisfied since any set of the form $\cc_{[\mu^*,\mu]}\coloneqq \cc\cap (\R^3\times [\mu^*,\mu])$ is a closed subset of the compact set $\{\x\colon 0\leq x,y,z\leq \mu \}\times [\mu^*,\mu]$.

   		In \cite[Sec.2, Appendix]{buse2009existence} it is shown that there is exactly one generalized center $(\x_c,\mu_c)$ for $f_s$, that $\x_c\in \Delta \cap \interior(K_{\mu_c})$, that $\mu_c > 0$ (hence we may assume $\mu^* < \mu_c$), and that $f_s$ undergoes a supercritical generic Hopf bifurcation at $(\x_c,\mu_c)$.
   		Hence hypothesis \ref{item:th-hopf-5} of Theorem \ref{th:hopf} is satisfied.
   		Hypothesis \ref{item:th-hopf-6} is satisfied because $\Delta$ is an invariant manifold for $f_{s,\mu}$ by symmetry (see \S \ref{sec:symmetry}) and $\Delta$ is diffeomorphic to $\R$, so no nonstationary periodic orbits can intersect $\Delta$.   		
   		Finally, the center subspace $E^c$ of $\D_{\x_c}f_{s,\mu_c}$ is orthogonal to $\Delta$ by Equation \eqref{eq:Ec-eq-Delta-perp}, so hypothesis \ref{item:th-hopf-7} is satisfied.
   		
   		In \S \ref{sec:dtheta} we defined a closed 1-form $\frac{d\theta}{2\pi}$ on $\R^3\setminus \Delta$ such that $\frac{d\theta}{2\pi}$ is Poincar\'{e} dual to $\{x = y \text{ and } z< x\}$ on $\R^3\setminus \Delta$.
   		In \S \ref{sec:rotation-repress} below, in Proposition \ref{prop:dth} we prove that, for every $s>2$, $\mu_1 > \mu^*$,
   		there exists $\epsilon > 0$ such that
   		$\frac{d\theta}{2\pi}(f_{s,\mu}) \geq \epsilon$ on $ K_\mu\setminus \Delta$ for all $\mu \in [\mu^*,\mu_1]$.
   		Let $\pi_1\colon \R^3\times \R\to \R^3$ denote the projection onto the first factor, and for any $\mu\in \R$ let $\iota_\mu\colon \R^3 \hookrightarrow \R^3\times \R$ be the inclusion $\iota_\mu(\x)=(\x,\mu)$.
   		Defining $\cf\coloneqq \pi_1^* (\frac{d\theta}{2\pi})$, noting that $\pi_1^* (\frac{d\theta}{2\pi})$ is Poincar\'{e} dual to $N\setminus M$ in $(\R^3\times \R)\setminus M$ \cite[p.~69]{bott1982differential}, and noting that $\iota_\mu^*\cf = \frac{d\theta}{2\pi}$ for any $\mu \in \R$, it follows that the lone remaining hypothesis \ref{item:th-hopf-3} of Theorem \ref{th:hopf} is also satisfied.
   		This completes the proof.
   	\end{proof}
   	
	\subsubsection{Rotational rate of the flow}\label{sec:rotation-repress}	
    In this section, we complete the proof of Theorem \ref{th:repress} by showing that $\frac{d\theta}{2\pi}$ satisfies the remaining hypothesis \ref{item:th-hopf-3} of Theorem \ref{th:hopf}.	
    \begin{Lem}\label{lem:dth-lem1}
    Fix $s,\mu > 0$ and let $f_{s,\mu}$ be the repressilator vector field \eqref{eq:repressilator} on $\R^3$.
    Let $\R^3_+$ be the closed positive orthant and $\Delta\subset \R^3$ be the diagonal.
    Then $d\theta(f_{s,\mu})> 0$ on $\R^3_+\setminus \Delta$.
    \end{Lem}
    
    \begin{proof}
	Define the 1-form $\omega\coloneqq (z-y)dx + (x-z)dy + (y-x)dz$ to be the ``numerator'' of $d\theta$.
	It suffices to show that $\omega(f_{s,\mu})>0$ on $\R^3_+\setminus \Delta$.
	We compute    
	\begin{equation}
	\omega(f_\mu) = \frac{z-y}{p(y)}+\frac{x-z}{p(z)}+\frac{y-x}{p(x)},	
	\end{equation}
	where the positive function $p$ is defined as $p(r)\coloneqq \mu^{-1}\cdot(1+r^s)$.
	Define the function
	\begin{equation}
	N(x,y,z)\coloneqq (z-y)p(z)p(x) + (x-z)p(x)p(y) + (y-x)p(y)p(z).
	\end{equation}
	Writing $\omega \cdot f_{s,\mu}\coloneqq \omega(f_{s,\mu})$, note that $(\omega\cdot f_{s,\mu})(x,y,z) = \frac{N(x,y,z)}{p(x)p(y)p(z)}$, so that $\omega(f_{s,\mu}) > 0$ if and only if $N > 0$.
	
	Let $\x = (x,y,z)\in \R^3_+\setminus \Delta$ and consider the terms $(z-y)$, $(x-z)$, $(y-x)$. Since these terms sum to zero and since $\x\not \in \Delta$, it must be the case that there is one nonzero term which has a different sign than both of the other two terms.\footnote{Note that this term need not be unique, since one of the terms may be zero.}
	Divide this term which has sign different from the other two by the pair of functions that 
	multiply it. 
	Without loss of generality, assume that $(z-y)$ is nonzero and has sign different from $(x-z),(y-x)$.
	We obtain
	\begin{equation}
	\frac{N(x,y,z)}{p(z)p(x)}=(z-y) + (x-z)\frac{p(y)}{p(z)} + (y-x)\frac{p(y)}{p(x)}.
	\end{equation}  
	Since $r\mapsto p(r)$ is strictly increasing, in the case that $(z-y)>0$ we obtain $\frac{p(y)}{p(z)} < 1$ and $\frac{p(y)}{p(x)}\leq 1$, with $\frac{p(y)}{p(x)} = 1$ if and only if $(y - x)=0$.
	It is clear that $\frac{N(x,y,z)}{p(z)p(x)}$---and hence $N(x,y,z)$---is positive in this case.
	Similarly, in the case that $(z-y)<0$ we obtain $\frac{p(y)}{p(z)}>1$ and $\frac{p(y)}{p(x)}\geq 1$, with $\frac{p(y)}{p(x)} = 1$ if and only if $(y-x)=0$, so again $N(x,y,z)$ is positive.
	As discussed it follows that, in both cases, we have $(\omega \cdot f_{s_,\mu})(\x)>0$ and hence also $(d\theta \cdot f_{s,\mu})(\x) > 0$, completing the proof. 
	  	
    \end{proof}
   	For use in Lemma \ref{lem:dth-lem2} and Proposition \ref{prop:dth} below, we recall the definition of the set $$\cc \coloneqq \{(\x,\mu)\in \R^3\times \R\colon \mu \geq 0 \text{ and } \x \in K_{\mu} \},$$ where again $K_\mu$ is defined for $\mu \geq 0$ as $$K_\mu\coloneqq \left\{(x,y,z)\colon \frac{\mu}{2+\mu^s}\leq x,y,z \leq \mu\right \}.$$
   	For any interval $J\subset \R$, we also define $\cc_{J}\coloneqq \cc \cap (\R^3\times J)$.
   	
   	\begin{Lem}\label{lem:dth-lem2}
   		Fix $s>2$ and $\mu^* > 0$.
   		Then for every $\mu_1 > \mu^*$, there exists $\delta > 0$ and a relatively open neighborhood $U\subset \cc_{[\mu^*,\mu_1]}$ of $(\Delta\times [\mu^*,\mu_1]) \cap \cc_{[\mu^*,\mu_1]}$ such that, for all $(\x,\mu)\in U\setminus (\Delta\times [\mu^*,\mu_1])$,  $$d\theta( f_{s,\mu}(\x)) \geq \delta.$$
   	\end{Lem}
   	
   	\begin{proof}
   		Define the 1-form $\omega\coloneqq (z-y)dx + (x-z)dy + (y-x)dz$ to be the ``numerator'' of $d\theta$.
   		Writing $\omega \cdot f_{s,\mu}\coloneqq \omega(f_{s,\mu})$ and defining $q_\mu(r)\coloneqq \frac{\mu}{1+r^s}$ for $\mu > 0$, we have
   		\begin{equation}\label{eq:om-f}
   		\omega \cdot f_{s,\mu} = (z-y)q_\mu(y) + (x-z)q_\mu(z) + (y-x)q_\mu(x),
   		\end{equation}
   		so
   		\begin{equation}\label{eq:om-f-0}
   		(\omega \cdot f_{s,\mu})|_\Delta \equiv 0.
   		\end{equation}
   		Note that $q_\mu$ is $C^\infty$ on $\R\setminus \{0\}$.
   		From \eqref{eq:om-f} we compute the first derivative $\D_\x(\omega\cdot f_{s,\mu})$ at $\x=(x,y,z)\neq 0$ to be
   		\begin{equation}\label{eq:D-om-f}
   		\begin{split}
   		\D_\x(\omega\cdot f_{s,\mu}) = \big[
   		q_\mu(z) - q_\mu(x) + (y-x)q_\mu'(x),\quad q_\mu(x)-q_\mu(y) + (z-y)q_\mu'(y),\\  q_\mu(y) - q_\mu(z) + (x-z)q_\mu'(z)
   		\big],
   		\end{split}
   		\end{equation} 
   		from which it follows that
   		\begin{equation}\label{eq:D-om-f-0}
   		\D(\omega\cdot f_{s,\mu})|_{\Delta\setminus \{0\}} \equiv 0. 
   		\end{equation}
   		\iffalse
   		From \eqref{eq:D-om-f-0} we compute the second derivative at 
   		
   		$\x\neq 0$ to be
   		\begin{equation}
   		\begin{split}
   		\D^2_\x(\omega\cdot f_{s,\mu}) = \begin{bmatrix}
   		 -2q'(x) + (y-x)q''(x) & q'(x) & q'(z)\\
   		 q'(x) & -2q'(y) + (z-y)q''(y) & q'(y) \\
   		 q'(z) & q'(y) & -2q'(z) + (x-z)q''(z)
   		\end{bmatrix}.
   		\end{split}
   		\end{equation}
   		\fi
   		From \eqref{eq:D-om-f} we compute the second derivative at $(r,r,r)\in \Delta\setminus \{0\}$ to be
   		\begin{equation}\label{eq:D2-om-f-diag}
   		\begin{split}
   		\D^2_{(r,r,r)}(\omega\cdot f_{s,\mu}) = q_\mu'(r)\begin{bmatrix}
   		-2 & 1 & 1\\
   		1 & -2 & 1 \\
   		1 & 1 & -2
   		\end{bmatrix},
   		\end{split}
   		\end{equation}
   		so for any $\vv\in \R^3$ we have
   		\begin{equation}\label{eq:D2-expr-1}
   		\D^2_{(r,r,r)}(\omega\cdot f_{s,\mu}) \cdot (
   		\vv,\vv) = -2 q_\mu'(r)\left(\norm{\vv}^2 - \langle \vv, \sigma(\vv)\rangle \right) = -3q_\mu'(r)\norm{\vv_\perp}^2,
   		\end{equation}  	
   		 where $\sigma$ is the cyclic permutation $\sigma(x,y,z) = (y,z,x)$ and the notation $\vv = \vv_\parallel +  \vv_{\perp} \in \Delta \oplus \Delta^\perp = \R^3$ is defined preceding \eqref{eq:dth-def}.
   		 Writing $\x = \x_\parallel + \x_\perp$ and using $\x_\parallel = \frac{x+y+z}{3}\1$, equations \eqref{eq:om-f-0}, \eqref{eq:D-om-f-0}, and \eqref{eq:D2-expr-1} together with Taylor's theorem imply that, for all $\x\in \R^3$,
   		\begin{equation}\label{eq:om-f-taylor}
   		\begin{split}
   		(\omega\cdot f_{s,\mu})(\x)&= \frac{1}{2} \D^2_{\x_\parallel}(\omega\cdot f_{s,\mu})\cdot(\x_\perp,\x_\perp)  + R_\mu(\x_{\parallel})\x_{\perp}^{\otimes 3} \\
   		&= - \frac{3}{2}q_\mu'\left(\frac{x+y+z}{3}\right)\norm{\x_{\perp}}^2 + R_\mu(\x_{\parallel})\x_{\perp}^{\otimes 3},
   		\end{split}
   		\end{equation}
   		where $(\x_\parallel,\mu)\mapsto R_\mu(\x_\parallel)$ is smooth on\footnote{We restrict attention to $\x_\parallel \in \Delta \setminus \{0\}$ since $q_\mu''$ is not differentiable at zero if $2<s<3$. This poses no problem for us since $K_\mu\cap \Delta^\perp = \varnothing$ for $\mu > 0$, i.e., $\x\in K_\mu$ implies $\x_\parallel \neq 0$.} $(\Delta \setminus \{0\})\times \R$.
   		Since the function $(\R^3 \setminus \{0\})\times \R \to \R$ given by $(\x,\mu)\mapsto \norm{R_\mu(\x_{\parallel})}\norm{\x_{\perp}}$ is continuous and since $\cc_{[\mu^*,\mu_1]}$ is disjoint from $(\Delta^\perp \times \R)$, for each $0 <\epsilon < 1$ the set $$U_{\epsilon}\coloneqq \{(\x,\mu)\in \cc_{[\mu^*,\mu_1]}  \colon \norm{R_\mu(\x_{\parallel})}\norm{\x_{\perp}} < \epsilon \}$$
   		is a relatively open neighborhood of $(\Delta\times [\mu^*,\mu_1]) \cap \cc_{[\mu^*,\mu_1]}$ in $\cc_{[\mu^*,\mu_1]}$.
   		Since $s> 2$ and $\mu > 0$, $-\frac{\sqrt{3}}{2}q_\mu'(r) = \frac{3}{2}\frac{s\mu r^{s-1}}{1+r^s} > 0$ is jointly continuous in $(r,\mu)$ and hence attains a minimum $m > 0$ on the compact set $\{(r,\mu)\colon (r,r,r,\mu)\in \cc_{ [\mu^*,\mu_1]}\}$.
   		Choose $\epsilon \leq \frac{m}{2}$.        
        Using \eqref{eq:om-f-taylor}, the fact that $U_\epsilon \subset \cc_{ [\mu^*,\mu_1]}$, and the fact that $\x\in K_\mu$ implies $\x_{\parallel}= \frac{x+y+z}{3}\1\in K_\mu$, it follows that for all $(\x,\mu)\in U_\epsilon \setminus (\Delta\times [\mu^*,\mu_1])$:
   		\begin{equation}
   		\begin{split}
   		(d\theta \cdot f_{s,\mu})(\x) &= \frac{1}{\sqrt{3}}\frac{(\omega\cdot f_{s,\mu})(\x)}{\norm{\x_{\perp}^2}}\\
   		&= -\frac{\sqrt{3}}{2}q'\left(\frac{x+y+z}{3}\right) + R(\x_{\parallel})\frac{\x_{\perp}^{\otimes 3}}{\norm{\x_\perp}^2}\\
   		& \geq -\frac{\sqrt{3}}{2}q'\left(\frac{x+y+z}{3}\right) - \norm{R(\x_{\parallel})}\norm{\x_\perp}\\
   		& \geq -\frac{\sqrt{3}}{2}q'\left(\frac{x+y+z}{3}\right) - \epsilon\\
   		&\geq \frac{m}{2}.
   		\end{split}
   		\end{equation}
   		Taking $U\coloneqq U_\epsilon$ and $\delta \coloneqq \frac{m}{2}$ completes the proof.
   	\end{proof}
   	
   	\begin{Prop}\label{prop:dth}
   		Fix $s> 2$ and $\mu^* > 0$.
   		Then for every $\mu_1 > \mu^*$,
   		there exists $\epsilon > 0$ such that, for all $\mu \in [\mu^*,\mu_1],$
   		$d\theta(f_{s,\mu}) \geq \epsilon$ on $ K_\mu\setminus \Delta$.
    \end{Prop}

    \begin{proof}
    	By Lemma \ref{lem:dth-lem2}, there exists $\delta_1 > 0$ and a relatively open neighborhood $U\subset \cc_{[\mu^*,\mu_1]}$ of $(\Delta\times [\mu^*,\mu_1]) \cap \cc_{[\mu^*,\mu_1]}$ in $\cc_{[\mu^*,\mu_1]}$ such that, for all $(\x,\mu)\in U\setminus (\Delta\times [\mu^*,\mu_1])$,  $$d\theta( f_{s,\mu}(\x)) \geq \delta_1.$$
    	Here $\cc$ and $\cc_{[\mu^*,\mu_1]}$ are as defined preceding Lemma \ref{lem:dth-lem2}.
    	By Lemma \ref{lem:dth-lem1}, $d\theta(f_{s,\mu}(\x))> 0$ for all $(\x,\mu)$ in the compact set  $\cc_{[\mu^*,\mu_1]} \setminus U$ and therefore attains a minimum $\delta_2$ on this set.
    	Defining $\epsilon \coloneqq \min\{\delta_1,\delta_2\}$, it follows that $$\forall (\x,\mu)\in \cc_{[\mu^*,\mu_1]}\setminus(\Delta \times [\mu^*,\mu_1])\colon d\theta(f_{s,\mu}(\x)) \geq \epsilon.$$ 
    	From the definition of $\cc_{[\mu^*,\mu_1]}$ we can write $\cc_{[\mu^*,\mu_1]} = \{(\x,\mu)\in \R^3\times \R\colon \mu^*\leq \mu \leq \mu_1 \text{ and } \x \in K_{\mu} \}$, so it follows that $d\theta(f_{s,\mu}(\x))\geq \epsilon$ whenever $\mu\in [\mu^*,\mu_1]$ and $\x\in K_\mu \setminus \Delta$.
    	This completes the proof.
    \end{proof}
    
    \subsection{The Sprott system: existence of periodic orbits}\label{sec:sprott}
    In this section we apply our theory to prove existence of periodic orbits for the Sprott system discussed in \S \ref{sec:intro}.
    As far as we know, this is the first time that the existence of nonstationary periodic orbits has been proven rigorously for this system.
    The equations are given on $\R^3$ by 
	\begin{equation}\label{eq:sprott}
	\begin{split}
	\dot{x} &= y^2 - z - \mu x\\
	\dot{y} &= z^2 - x - \mu y\\
	\dot{z} &= x^2 - y - \mu z,
	\end{split}
	\end{equation}
	and depend on the parameter $\mu\in \R$.
    We note that, unlike the repressilator \eqref{eq:repressilator}, the Sprott system is not a monotone cyclic feedback system \cite{mallet1990poincare}. 	
	Some trajectory segments of the dynamics for $\mu = 0$ are shown in Figure \ref{fig:sprott-traj} and and for other values of $\mu$ in Figure \ref{fig:ball-traj-mu-various}.
	The sphere shown is defined in \S \ref{sec:compact-container}.
	In the sequel, we let $f_\mu$ denote the vector field defined by \eqref{eq:sprott}.
	
	At the end of \S\ref{sec:sprott-existence} we will prove that \eqref{eq:sprott} has a periodic orbit for all $\mu \in (-0.25,0.5)$.
	Just like for the repressilator, the proof will amount to showing that \eqref{eq:sprott} satisfies the hypotheses of Theorem \ref{th:hopf}.
	In the intervening sections we will construct the ingredients required to do this. 
	First, in \S \ref{sec:compact-container} we find a certain compact set $K_\mu$ which contains all bounded trajectories of \eqref{eq:sprott}; we will define the set $\cc$ of Theorem \ref{th:hopf} in terms of $K_\mu$.
	Unlike the sets $K_\mu$ defined for the repressilator, in this section $K_\mu$ is not a trapping region and is not even invariant; this illustrates the flexibility allowed by the hypotheses of Theorem \ref{th:hopf}.
	 \S \ref{sec:rotation} consists of deriving estimates involving $d\theta(f_\mu)$ (where $d\theta$ is defined in \S \ref{sec:dtheta}) used to establish hypothesis \ref{item:th-hopf-3} of Theorem \ref{th:hopf}.
	 In \S \ref{sec:equilibria} we determine the equilibria and associated eigenvalues of $\D f_\mu$. 
	 In \S  \ref{sec:sprott-hopf} we show that \eqref{eq:sprott} exhibits Hopf bifurcations, needed in particular to verify hypothesis \ref{item:th-hopf-5} of Theorem \ref{th:hopf}.
	 Finally, \S \ref{sec:sprott-existence} combines these ingredients to prove the periodic orbit existence theorem.

	\begin{figure}
		\centering
		\includegraphics[width=0.49\linewidth]{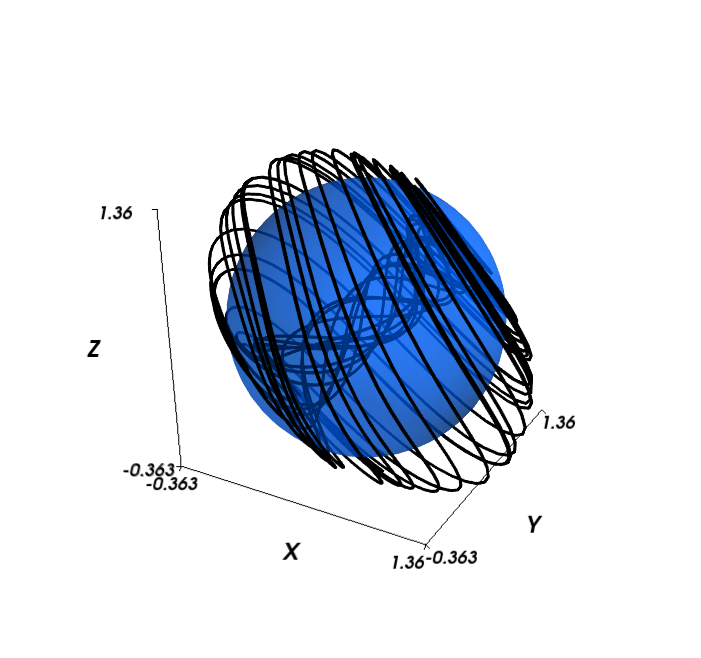}
		
		\includegraphics[width=0.49\linewidth]{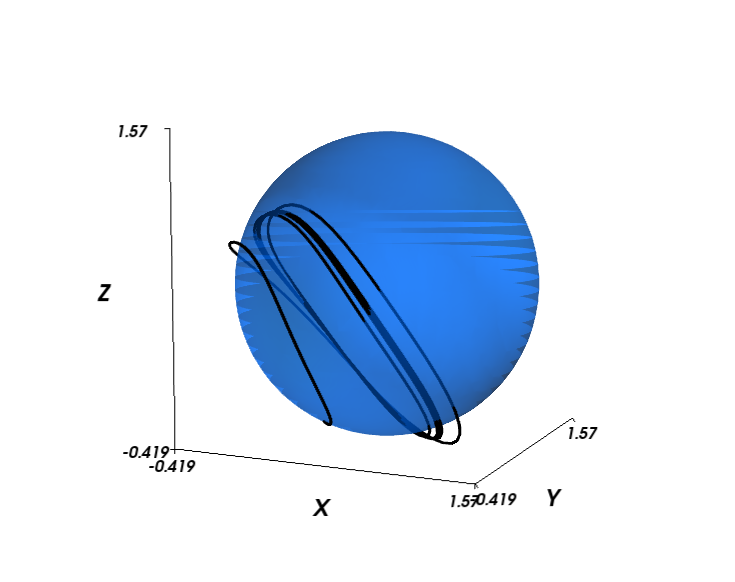} \includegraphics[width=0.49\linewidth]{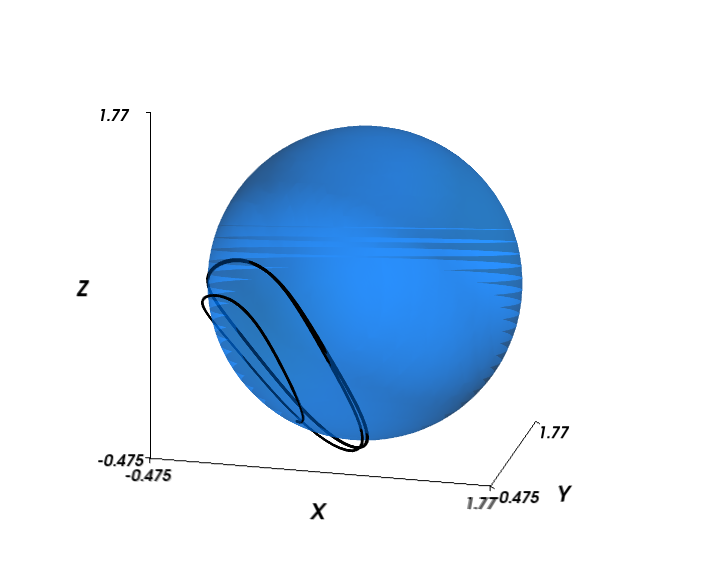}
		\caption{Shown here are trajectory segments of \eqref{eq:sprott} with $\mu = 0$ (top), $\mu = 0.15$ (bottom left), and $\mu = 0.3$ (bottom right).  
			All initial conditions are $(x_0,y_0,z_0) = (0.3, 0.2, -0.3)$. Also shown is the sphere $\dot{V}^{-1}(0)$ (c.f. Equation \eqref{eq:iso-func-def} and Figure \ref{fig:ball-planes}).}\label{fig:ball-traj-mu-various}	
	\end{figure}

	\subsubsection{A compact set containing all bounded trajectories}\label{sec:compact-container}
	Define the function $V\colon \R^3 \to \R$ via $V(\x) \coloneqq x + y + z$.
	A computation shows that the Lie derivative $\dot{V}$ of $V$ is
	\begin{equation}\label{eq:iso-func-def}
	\dot{V}(x,y,z) =  \norm{\x}^2-(\mu+1)(x+y+z) = \langle \textbf{1}, f_\mu(\x) \rangle.
	\end{equation}
	For any $c \geq -\frac{3}{4}(\mu + 1)^2$, the sublevel set $B_{\mu,c}\coloneqq \dot{V}^{-1}(-\infty,c]$ is the closed ball of radius $\frac{\sqrt{3(\mu+1)^2+4c}}{2}$ centered at $(\frac{1+\mu}{2})\textbf{1}$.
	In particular, the zero sublevel set of $\dot{V}$ is centered at the midpoint of two equilibria on the diagonal (the origin and $(1+\mu) \1$), with the two equilibria being antipodal points on the bounding sphere.
	Furthermore, the planes $V^{-1}(0)$ and $V^{-1}(3(1+\mu))$ are tangent to the sphere at these antipodal points.
	See Figure \ref{fig:ball-planes}.
	
	\begin{figure}
		\centering
		\includegraphics[width=0.5\linewidth]{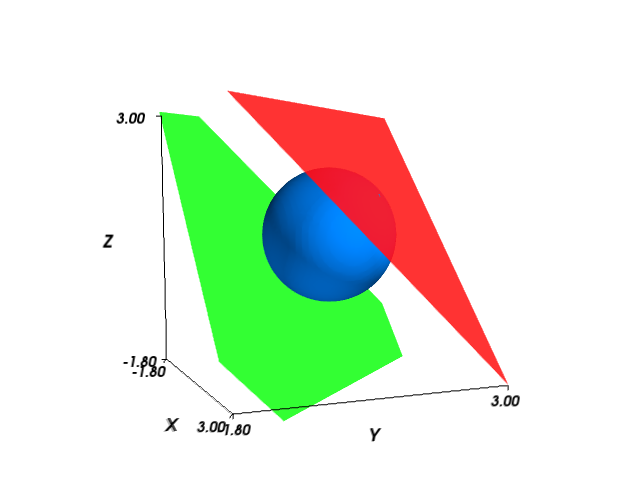}
		\caption{Shown here is the spherical level set $\dot{V}^{-1}(0)$ (teal), and the two planes $V^{-1}(0)$ (green) and $V^{-1}(3(1+\mu))$ (red) for $\mu = 0.4$.}\label{fig:ball-planes}	
	\end{figure}

	This geometry implies that the subsets $V^{-1}(-\infty,0)$ and $V^{-1}(3(1+\mu),\infty)$ are respectively negatively and positively invariant for $\mu \geq -1$.
	Furthermore, trajectories in these regions tend to $\infty$ in negative and positive time, respectively.
	It follows that any bounded trajectory must be contained in $V^{-1}[0,3(1+\mu)]$ when $\mu \geq -1$.
	We will further refine these considerations to produce a certain \emph{compact} set containing all bounded trajectories.
	
	Define translated coordinates $\textbf{x}_\mu\coloneqq (x_\mu,y_\mu,z_\mu)\coloneqq \x - \frac{(1+\mu)}{2}\textbf{1}$ and define $r_\mu\coloneqq \norm{\x_\mu}$.
	
	\begin{Th}\label{th:compact-container}
		For $\mu > -1$, every bounded trajectory is contained in the compact set $K_\mu$ defined by
		\begin{equation}\label{eq:Kmu-def}
		\begin{split}
		K_\mu \coloneqq& \bigg\{\x\in \R^3\colon V(\x)\geq r_\mu - \frac{3^{\frac{3}{4}}+3^{\frac{1}{4}}}{2}(1+\mu)\arctan\left(\frac{2 r_\mu}{(3^{1/4})(1+\mu)}\right)- \frac{\sqrt{3}}{2}(1+\mu)\\
		& + \frac{3^{\frac{3}{4}}+3^{\frac{1}{4}}}{2}(1+\mu)\arctan\left(3^{1/4}\right)\bigg\} \cap V^{-1}[0,3(1+\mu)],
		\end{split}	
		\end{equation}
		and $\dot{V}^{-1}(-\infty, 0]\subset K_\mu$.
		For $\mu = -1$, the only bounded trajectory is the equilibrium at the origin; we define $K_{-1}\coloneqq \{\mathbf{0}\}$.		
	\end{Th}
	For a visual depiction of $K_\mu$, see Figure \ref{fig:Kmu}.
	Note that the ball $\dot{V}^{-1}(-\infty,0]$ bounded by the sphere $\dot{V}^{-1}(0)$ shown in Figure \ref{fig:ball-planes} is contained in $K_\mu$. 

	\begin{figure}
		\centering
		\includegraphics[width=0.5\linewidth]{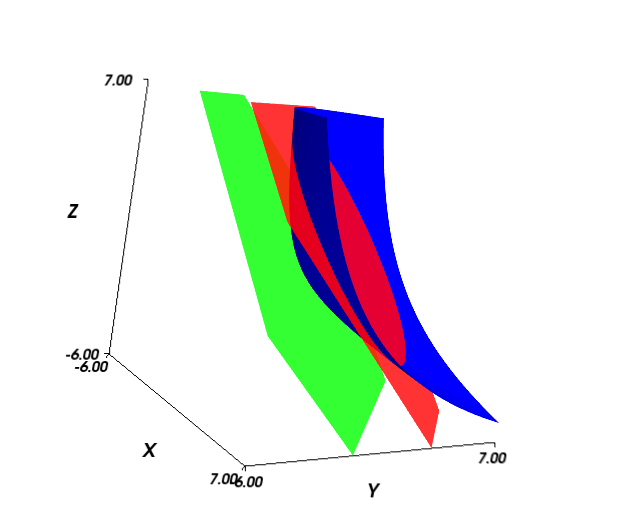}
		\caption{The compact set $K_\mu$ of Theorem \ref{th:compact-container} is the region bounded by the blue surface, red plane $V^{-1}(3(1+\mu))$, and green plane $V^{-1}(0)$.
			Note that the ball $\dot{V}^{-1}(-\infty,0]$ with boundary $\dot{V}^{-1}(0)$ shown in Figure \ref{fig:ball-planes} is contained in $K_\mu$. 
			This figure was generated using $\mu = 0.4$.}\label{fig:Kmu}	
	\end{figure}

	\begin{proof}
		For the case that $\mu = -1$, positive invariance of $V^{-1}(0,\infty)$, negative invariance of $V^{-1}(-\infty,0)$, and the fact that $\dot{V}^{-1}(0) = \{\mathbf{0}\}$ implies that the equilibrium at the origin is the only bounded trajectory of $f_{-1}$.
		For the remainder of the proof, we consider the case $\mu > -1$.
		
		$K_{\mu}$ is compact since it is clearly closed and bounded.
		We now show that 
		\begin{equation}\label{eq:ball-contained-Kmu}
		\dot{V}^{-1}(-\infty,0] = \left\{r_\mu \leq \frac{\sqrt{3}}{2}(1+\mu)\right\}\subset K_\mu.
		\end{equation}
		Note that (i) the midpoint $\frac{(1+\mu)}{2}\1$ of the ball $\dot{V}^{-1}(-\infty,0] = \{r_\mu \leq \frac{\sqrt{3}}{2}(1+\mu)\}$ belongs to $K_\mu$ since when $r_\mu = 0$ 
		the right side of the inequality in \eqref{eq:Kmu-def} is equal to $$ \frac{(1+\mu)}{2}(3^{\frac{3}{4}}+3^{\frac{1}{4}}-\sqrt{3})\leq \frac{3(1+\mu)}{2} = V\left(\frac{(1+\mu)}{2}\1 \right),$$ and (ii) $\partial( \dot{V}^{-1}(-\infty,0]) = \dot{V}^{-1}(0) \subset K_\mu$ since both $K_\mu, \dot{V}^{-1}(0)\subset V^{-1}[0,3(1+\mu)]$ and the right side of the inequality in \eqref{eq:Kmu-def} vanishes when $r_\mu = \frac{\sqrt{3}}{2}(1+\mu)$. 
		Since (a) $\dot{V}^{-1}(-\infty,0]$ is convex, (b) $V$ is linear, and (c) the right side of the inequality in \eqref{eq:Kmu-def} is a convex function with respect to $r_{\mu}$ (its second derivative with respect to $r_{\mu} > 0$ is positive everywhere), by considering the inequality in \eqref{eq:Kmu-def} on rays emanating from $\frac{(1+\mu)}{2}\1$ and using (i, ii) it follows that $\dot{V}^{-1}(-\infty,0] \subset K_\mu$ as desired.

		Next, let $t\mapsto \x(t)$ be a trajectory of \eqref{eq:sprott}.
		If $V(\textbf{x}(0)) \not \in [0,3(1+\mu)]$, then $\norm{\textbf{x}(t)} \to \infty$ in either positive or negative time, so every bounded trajectory is contained in $V^{-1}[0,3(1+\mu)]$.
		Hence it suffices to restrict our attention to trajectories satisfying $\textbf{x}(0)\in V^{-1}[0,3(1+\mu)]$.
		Since any trajectory in $V^{-1}(-\infty,0)$ tends to $\infty$ in negative time, to prove the theorem it suffices to show that, for all $\textbf{x}(0)\in V^{-1}[0,3(1+\mu)]$, if $\textbf{x}(0) \not \in K_\mu$ then there exists a time $t_f < 0$ such that $V(\textbf{x}(t_f)) < 0$.
		
		Define the shifted function $V_\mu\coloneqq V - \frac{3(1+\mu)}{2} = x_\mu + y_\mu + z_\mu$.
		We compute 
		\begin{align*}
		\dot{r}_\mu &= \frac{\textbf{x}_\mu\cdot f_\mu(\textbf{x}_\mu + \frac{(1+\mu)}{2}\textbf{1})}{r_\mu}\\
		&= \frac{(x_\mu y_\mu^2 + y_\mu z_\mu^2 + z_\mu x_\mu^2) + \mu(x_\mu y_\mu + y_\mu z_\mu+ z_\mu x_\mu) - \mu r_\mu^2 - \frac{1}{4}(1+\mu)^2V_\mu}{r_\mu}.
		\end{align*}
		The Cauchy-Schwarz inequality and subadditivity of $\sqrt{\cdot}$ applied to the first and second numerator terms yields
		\begin{equation}
		\dot{r}_\mu \leq \frac{r_\mu^3 - \frac{1}{4}(1+\mu)^2V_\mu}{r_\mu}.
		\end{equation}
		Additionally, we have
		\begin{equation}\label{eq:lem-dot-V-eqn}
		\dot{V}_\mu = \dot{V} = r_\mu^2 - \frac{3}{4}(1+\mu)^2. 
		\end{equation}
		
		Consider now a trajectory $\textbf{x}_\mu(t) = \x(t) - \frac{(1+\mu)}{2}\textbf{1}$ with initial condition $\textbf{x}_\mu(0) \in \{r_\mu^2 > \frac{3}{4}(1+\mu)^2\}$. 
		$V_\mu$ increases monotonically along $\textbf{x}_\mu(t)$ as long as $\textbf{x}_\mu(t) \in \{r_\mu > \frac{\sqrt{3}}{2}(1+\mu)\}$, so time can be written as a function $t(V_\mu)$ of $V_\mu$, and we may therefore parametrize $r_\mu$ as a function of $V_\mu$.
		Using the chain rule, we compute
		\begin{align*}
		\frac{dr_\mu}{dV_\mu} &= \frac{\dot{r}_\mu}{\dot{V}_\mu} \\&\leq \frac{r_\mu^2 - \frac{1}{4}(1+\mu)^2\frac{V_\mu}{r_\mu}}{r_\mu^2 - \frac{3}{4}(1+\mu)^2}.
		\end{align*}
		
		We now further restrict our attention to a trajectory segment satisfying $0 \leq V \leq 3(1+\mu)$, or $-\frac{3}{2}(1+\mu) \leq V_\mu \leq \frac{3}{2}(1+\mu)$.
		We continue to assume that $r_\mu > \frac{\sqrt{3}}{2}(1+\mu)$ along this trajectory segment.
		It follows that 
		\begin{equation}\label{eq:drdv-ineq}
		\frac{dr_\mu}{dV_\mu} \leq \frac{r_\mu^2 + \frac{\sqrt{3}}{4}(1+\mu)^2}{r_\mu^2 - \frac{3}{4}(1+\mu)^2}.
		\end{equation}
		Let $\tilde{r}_\mu(V_\mu)$ denote a solution to the ODE defined by replacing the inequality in \eqref{eq:drdv-ineq} with equality.
		This ODE is separable and admits the implicit solution family
		\begin{equation}\label{eq:Vmu-tilde-rmu-soln}
		c + V_\mu = \tilde{r}_\mu - \frac{3^{\frac{3}{4}}+3^{\frac{1}{4}}}{2}(1+\mu)\arctan\left(\frac{2 \tilde{r}_\mu}{(3^{1/4})(1+\mu)}\right),
		\end{equation}
		where $c$ is an arbitrary constant of integration.
		Considering \eqref{eq:Vmu-tilde-rmu-soln} for different values $V_{\mu,0}\coloneqq V_\mu(\x(0))$ and $V_{\mu,t_f}\coloneqq V_\mu(\x(t_f))$ and subtracting the resulting two equations, we obtain
		\begin{equation}\label{eq:V-difference-lemma}
		\begin{split}
		V_\mu(\textbf{x}(0))-V_\mu(\textbf{x}(t_f)) = \tilde{r}_\mu(V_{\mu,0}) - \frac{3^{\frac{3}{4}}+3^{\frac{1}{4}}}{2}(1+\mu)\arctan\left(\frac{2 \tilde{r}_\mu(V_{\mu,0})}{(3^{1/4})(1+\mu)}\right)- \tilde{r}_\mu(V_{\mu,t_f})\\
		+ \frac{3^{\frac{3}{4}}+3^{\frac{1}{4}}}{2}(1+\mu)\arctan\left(\frac{2 \tilde{r}_\mu(V_{\mu,t_f})}{(3^{1/4})(1+\mu)}\right).
		\end{split}
		\end{equation}	
		Positivity of the right-hand side of \eqref{eq:drdv-ineq} implies that $V_\mu \mapsto \tilde{r}_\mu(V_\mu)$ is strictly increasing,
		which in turn implies that the right-hand side of \eqref{eq:Vmu-tilde-rmu-soln} is a strictly increasing function of $\tilde{r}_\mu$ for $\tilde{r}_\mu \geq \frac{\sqrt{3}}{2}(1+\mu)$.
		If we assume that $\tilde{r}_\mu(V_{\mu,0}) = r_\mu(0)$ (viewing $r_\mu$ as a function of $t$) and stipulate that $t_f \leq 0$, then the comparison lemma \cite[Sec.~2.7]{arnold1973ordinary} and \eqref{eq:drdv-ineq} imply that $r(t_f) \geq \tilde{r}_\mu(V_{\mu,t_f})$,
		so it follows from \eqref{eq:V-difference-lemma} and the preceding sentence that
		\begin{equation}\label{eq:V0-2nd-last-big-ineq}
		\begin{split}
		V(\textbf{x}(0))-V(\textbf{x}(t_f)) &\geq r_\mu(0) - \frac{3^{\frac{3}{4}}+3^{\frac{1}{4}}}{2}(1+\mu)\arctan\left(\frac{2 r_\mu(0)}{(3^{1/4})(1+\mu)}\right)- r_\mu(t_f)\\
		& + \frac{3^{\frac{3}{4}}+3^{\frac{1}{4}}}{2}(1+\mu)\arctan\left(\frac{2 r_\mu(t_f)}{(3^{1/4})(1+\mu)}\right),
		\end{split}
		\end{equation}
		where we have used the fact that $V(\textbf{x}(t_f))-V(\textbf{x}(0)) = V_\mu(\textbf{x}(t_f))-V_\mu(\textbf{x}(0))$.
		
		Assume that $\inf_{t_f\leq 0} r_\mu(t_f) \leq \frac{\sqrt{3}}{2}(1+\mu)$.
		Then there exists a (possibly bounded) decreasing subsequence $(t_n)_{n\in \N}\subset (-\infty,0)$ of negative values of $t_f$ with $r_\mu(t_n) > \frac{\sqrt{3}}{2}(1+\mu)$ and   $\lim_{n\to\infty} r_\mu(t_n) = \frac{\sqrt{3}}{2}(1+\mu)$.
        Since $\{r_\mu = \frac{\sqrt{3}}{2}(1+\mu)\}\subset \{V\geq 0\}$ it follows that $\lim_{n\to\infty} V(\x(t_n)) \geq 0$, so substituting $t_f = t_n$ in both sides of \eqref{eq:V0-2nd-last-big-ineq} and taking the limit $n\to \infty$ yields
		\begin{equation}\label{eq:V0-final-big-inequality}
		\begin{split}
		V(\textbf{x}(0)) &\geq r_\mu(0) - \frac{3^{\frac{3}{4}}+3^{\frac{1}{4}}}{2}(1+\mu)\arctan\left(\frac{2 r_\mu(0)}{(3^{1/4})(1+\mu)}\right)- \frac{\sqrt{3}}{2}(1+\mu)\ldots\\
		&\ldots + \frac{3^{\frac{3}{4}}+3^{\frac{1}{4}}}{2}(1+\mu)\arctan\left(3^{1/4}\right).
		\end{split}
		\end{equation}
		By \eqref{eq:Kmu-def}, $K_{\mu}$ is precisely the set of points $\x(0)\in V^{-1}[0,3(1+\mu)]$ which satisfy \eqref{eq:V0-final-big-inequality}.	
		In summary, we have shown that a necessary condition for the closure $\cl(\x(-\infty,0])$ of the negative-time trajectory through $\x(0)\in V^{-1}[0,3(1+\mu)] \setminus \dot{V}^{-1}(-\infty,0]$ to intersect $\dot{V}^{-1}(-\infty,0]$ is that $\x(0)\in K_\mu$.
		This and \eqref{eq:ball-contained-Kmu} imply that the trajectory through any $\x(0)\in V^{-1}[0,3(1+\mu)]\setminus K_\mu$ is bounded away from $\dot{V}^{-1}(-\infty,0]$ uniformly for all negative time.
        For such an $\x(0)$, it follows that there exists $\epsilon > 0$ such that $\dot{V} < -\epsilon$ uniformly along the corresponding negative-time trajectory, and therefore $\x(t)$ must enter $V^{-1}(-\infty,0)$ in finite negative time.
		By the final sentence of the third paragraph of this proof, this completes the proof.		
	\end{proof}

	\subsubsection{Cylindrical coordinates and rotation of the flow}\label{sec:rotation}
	Define an orthogonal matrix $M$ via
	\begin{equation}
	M = \begin{bmatrix}\frac{\sqrt{6}}{6} & - \frac{\sqrt{2}}{2} & \frac{\sqrt{3}}{3}\\\frac{\sqrt{6}}{6} & \frac{\sqrt{2}}{2} & \frac{\sqrt{3}}{3}\\- \frac{\sqrt{6}}{3} & 0 & \frac{\sqrt{3}}{3}\end{bmatrix}
	\end{equation}
	and define coordinates $[u,v,w]^T\coloneqq M^{-1} [x,y,z]^T = M^T [x,y,z]^T  $.
	The $w$-axis corresponds to $\Delta$ in the original coordinates, and the $u$ and $v$ axes determine an orthonormal coordinate system for $\Delta^\perp$.

	We further define cylindrical coordinates $(\rho,\theta,w)$ via
	\begin{equation}
	\begin{split}
	u &= \rho \cos\theta\\
	v &= \rho \sin \theta.
	\end{split}
	\end{equation}
	Using the symbolic package SymPy, we obtain the equations of motion in these new coordinates in closed form:
	\begin{equation}\label{eq:cylind-eqns}
	\begin{split}
	\dot{\rho} &= \rho \left(-\sqrt{2} \rho \sin^{3}{\left (\theta \right )} +  \frac{\sqrt{6}}{3} \rho \cos^{3}{\left (\theta \right )} - \frac{\sqrt{6}}{2} \rho \cos{\left (\theta + \frac{\pi}{3} \right )} - \frac{1}{\sqrt{3}} w - \mu + \frac{1}{2}\right)\\
	\dot{\theta} &= \rho\left(\frac{\sqrt{6}}{3} \sin^{3}{\left (\theta \right )} - \frac{\sqrt{6}}{2} \sin{\left (\theta + \frac{\pi}{3} \right )} + \sqrt{2} \cos^{3}{\left (\theta \right )}\right) - w - \frac{\sqrt{3}}{2} \\
	\dot{w} &= \frac{1}{\sqrt{3}}(\rho^2 + w^2) - (\mu+1) w.
	\end{split}
	\end{equation}
	
	Now $$\dot{\theta} = \left(\frac{u dv - vdu}{u^2 + v^2}\right)(f_\mu),$$ and $(u,v,w)$ are orthogonal coordinates adapted to the splitting $\R^3 = \Delta^\perp \oplus \Delta$ with $(u,v)$ coordinates for $\Delta^\perp$ and with $w$ a coordinate for $\Delta$.
	It follows from the discussion in \S \ref{sec:dtheta} that $\dot{\theta} = d\theta(f_\mu)$, where $d\theta$ is defined in \S \ref{sec:dtheta}.
	Because the hypotheses of Theorem \ref{th:hopf} are stated in terms of a closed 1-form, we will write $d\theta(f_\mu)$ instead of $\dot{\theta}$ in the following results.	
	
	\begin{Lem}\label{lem:rot-rate}
		The following estimate holds on $\R^3\setminus \Delta$ or, equivalently, whenever $\rho > 0$:
		\begin{equation}
		-0.41\rho  - w - \frac{\sqrt{3}}{2} \leq d\theta(f_\mu) \leq  0.41 \rho - w - \frac{\sqrt{3}}{2}. 
		\end{equation}
	\end{Lem}
	\begin{proof}
		The sinusoidal function
		$\theta \mapsto \left(\frac{\sqrt{6}}{3} \sin^{3}{\left (\theta \right )} - \frac{\sqrt{6}}{2} \sin{\left (\theta + \frac{\pi}{3} \right )} + \sqrt{2} \cos^{3}{\left (\theta \right )}\right)$ has zero mean and amplitude smaller than $0.41$.
		The result now follows from \eqref{eq:cylind-eqns}.
	\end{proof}
	
	The following result concerns the rotation rate $\dot{\theta}=d\theta(f_\mu)$ on the compact set $K_\mu$ (defined in Theorem \ref{th:compact-container}) which contains all bounded trajectories.

	\begin{figure}
		\centering
		\includegraphics[width=0.5\linewidth]{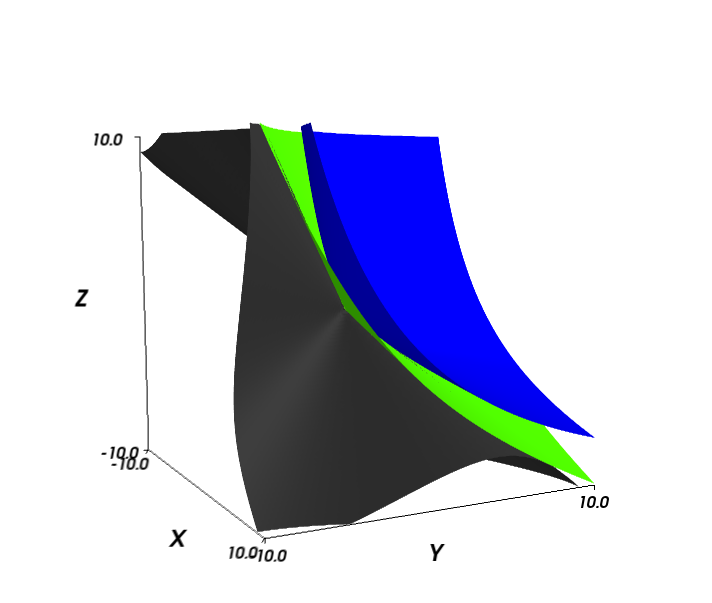}
		\caption{The blue surface is the same portion of the boundary of $K_\mu$ depicted in Figure \ref{fig:Kmu}.
			The dark surface is the region of space where $d\theta(f_\mu) = 0$ (not including $\Delta$), with $d\theta(f_\mu) < 0$ in the region of space containing the blue surface. The green surface is the boundary of the conservative inner approximation of the region of space where $d\theta(f_\mu) < 0$ obtained in Lemma \ref{lem:rot-rate}. This figure was generated using $\mu = 0.4$.}\label{fig:bad-rot-set-missed}	
	\end{figure}

	\begin{Th}\label{th:rot-rate-Kmu}
		There exists $\epsilon > 0$ such that, for all $-1\leq \mu \leq 0.55$, $d\theta(f_\mu) < -\epsilon$ on $K_\mu \setminus \Delta$.
	\end{Th}
	\begin{proof}
		We have $K_{-1}\setminus \Delta = \varnothing$ since $K_{-1} = \{\mathbf{0}\}$, so the statement holds vacuously for $\mu = -1$.
		For the remainder of the proof we assume that $\mu >-1$.
		
		It follows from Lemma \ref{lem:rot-rate} that $d\theta(f_\mu) < -\epsilon$ whenever $\rho < \frac{1}{0.41}w +  \frac{\sqrt{3}}{0.82} - \frac{\epsilon}{0.41}\approx 2.439 w + 2.112 - 2.438 \epsilon$.
		Since $w\geq 0$ on $K_\mu$ and $\frac{\sqrt{3}}{0.82}> 2.1$ it follows that, for any sufficiently small $\epsilon > 0$, the $\rho, w$ coordinates of every point in $K_\mu$ satisfy either the preceding inequality or the inequality $\rho > 2.1$.
		Hence it suffices to find an $\epsilon > 0$ such that $d\theta(f_\mu) < -\epsilon$ whenever $\x \in K_\mu$, $\rho > 2.1$ and $\mu\in[-1,0.55]$.	
		
		We use the notation from the statement and proof of Theorem \ref{th:compact-container}.
		From the definition of $K_\mu$, we have
		\begin{equation}\label{eq:V-ineq-rot-rate}
		\begin{split}
		V&\geq r_\mu - \frac{3^{\frac{3}{4}}+3^{\frac{1}{4}}}{2}(1+\mu)\arctan\left(\frac{2 r_\mu}{(3^{1/4})(1+\mu)}\right)- \frac{\sqrt{3}}{2}(1+\mu) + \frac{3^{\frac{3}{4}}+3^{\frac{1}{4}}}{2}(1+\mu)\arctan\left(3^{1/4}\right)\\
		&\geq \rho - \frac{3^{\frac{3}{4}}+3^{\frac{1}{4}}}{2}(1+\mu)\arctan\left(\frac{2 \rho}{(3^{1/4})(1+\mu)}\right)- \frac{\sqrt{3}}{2}(1+\mu) + \frac{3^{\frac{3}{4}}+3^{\frac{1}{4}}}{2}(1+\mu)\arctan\left(3^{1/4}\right)\\
		&\geq \rho -1.8(1+\mu)\arctan\left(\frac{1.52 \rho}{1+\mu}\right) + 0.78(1+\mu)	
		\end{split}
		\end{equation}
		on $K_\mu$.
		The second inequality follows from the following three observations:  (i) we showed in the proof of Theorem \ref{th:compact-container} that the first two terms on the right side of \eqref{eq:V-ineq-rot-rate}  constitute an increasing function of $\rho$ if $\rho \geq \frac{\sqrt{3}}{2}(1+\mu)$, (ii) $2.1 > \frac{\sqrt{3}}{2}(1+\mu)$ if $\mu \leq 1.41$, and (iii) $\rho \leq r_\mu$ (the distance to the diagonal $\Delta$ is at most the distance to any individual point on $\Delta$).
		Since $V = \sqrt{3}w$, 
		it now follows that
		\begin{equation}
		\begin{split}
		-w &\leq -\frac{1}{\sqrt{3}}\left(\rho -1.8(1+\mu)\arctan\left(\frac{1.52 \rho}{1+\mu}\right) + 0.78(1+\mu)\right)\\
		&\leq -0.57 \rho + 1.04(1+\mu)\arctan\left(\frac{1.52 \rho}{1+\mu}\right) - 0.45(1+\mu).
		\end{split} 
		\end{equation}
		Substituting this into Lemma \ref{lem:rot-rate}, we find that, when $\x\in K_\mu$ and $\rho > 2.1$,
		\begin{equation}\label{eq:dth-fmu-Fmu}
		d\theta(f_\mu) \leq -0.15 \rho + 1.04(1+\mu)\arctan\left(\frac{1.52 \rho}{1+\mu}\right) - 0.45(1+\mu) - 0.86\eqqcolon F_{\mu}(\rho).
		\end{equation}
		We compute the derivative
		\begin{equation}
		F_\mu'(\rho) = -0.15 + \frac{(1.04)(1.52)}{1+\frac{(1.52)^2}{(1+\mu)^2}\rho^2}
		\end{equation}
		and see that $F_\mu'$ is positive for small $\rho$ and is strictly decreasing.
		Therefore, $F_\mu$ has a unique critical point $\rho^*(\mu)\in (0,\infty)$, $\rho^*(\mu)$ is a local maximum, and it satisfies
		\begin{equation*}
		2.0318(1+\mu)\leq \rho^*(\mu) = \frac{(1+\mu)}{1.52} \sqrt{-1 + \frac{(1.04)(1.52)}{0.15}} \leq 2.0319(1+\mu).
		\end{equation*}
		Substituting this expression into \eqref{eq:dth-fmu-Fmu} yields
		\begin{align*}
		F_\mu(\rho^*(\mu))&\leq  (1+\mu)[(-0.15)(2.0318)+1.04\arctan((1.52)(2.0319))-0.45] -0.86\\
		&\leq 0.5532(1+\mu) - 0.86,
		\end{align*}
		which is strictly less than $-\epsilon\coloneqq -0.0001$ whenever 
		\begin{equation*}
		\mu < \frac{0.8599}{0.5532}-1.
		\end{equation*}
		The quantity on the right is strictly larger than $0.55$, so 
		$$d\theta(f_\mu)\leq F_\mu(\rho^*(\mu)) <  - \epsilon$$ whenever $\x\in K_\mu$, $\rho > 2.1$, and $\mu \leq 0.55$.
		By the discussion in the second paragraph of the proof, this completes the proof.
	\end{proof}
	
	\begin{Co}\label{co:equilib-on-diag}
		For $\mu \in [-1,0.55]$, all equilibria of $f_\mu$ belong to the diagonal $\Delta$.
	\end{Co}
	\begin{proof}
		For $\mu \in [-1,0.55]$, Theorem \ref{th:compact-container} implies that all equilibria lie in $K_\mu$, and Theorem \ref{th:rot-rate-Kmu} implies that $d\theta(f_\mu)< 0$ on $K_\mu \setminus \Delta$, so in particular $f_\mu \neq 0$ on $K_\mu \setminus \Delta$.
	\end{proof}
	
	\begin{Co}\label{co:sprott-per-orbits-contained-Kmu}
		For $\mu \in (-1,0.55]$, all periodic orbits of $f_\mu$ are contained in $K_\mu\setminus \Delta$, and the winding number $\frac{1}{2\pi}\int_\gamma d\theta$ of any nonstationary periodic orbit $\gamma$ around $\Delta$ satisfies $\frac{1}{2\pi}\int_\gamma d\theta \leq -1.$
		For the case $\mu = -1$, $f_{-1}$ has no nonstationary periodic orbits.
	\end{Co}
\begin{proof}
	For $\mu \in [-1,0.55]$, Theorem \ref{th:compact-container} implies that all periodic orbits lie in $K_\mu$.
	Furthermore, nonstationary periodic orbits must lie in $K_\mu \setminus \Delta$ since $\Delta$ is a $1$-dimensional invariant manifold (\S \ref{sec:symmetry}) and thus cannot intersect nonstationary periodic orbits. 
	The condition $\frac{1}{2\pi}\int_\gamma d\theta \leq -1$ follows since $d\theta(f_\mu)< 0$ on $K_\mu \setminus \Delta$ by Theorem \ref{th:rot-rate-Kmu}.
	$f_{-1}$ has no nonstationary periodic orbits since $K_{-1}=\mathbf{0}$.
\end{proof}
	
	\subsubsection{Equilibria}\label{sec:equilibria}
	By \S \ref{sec:symmetry}, $\Delta$ is invariant and the dynamics restricted to $\Delta$ are given by
	\begin{equation}\label{eq:diag-dynamics}
	\dot{x} = x^2 - x - \mu x = x(x- 1 - \mu).
	\end{equation}
	
	\begin{Th}
		For all $\mu \in \R$, the vector field $f_\mu$ has the equilibria $\mathbf{0}$ and $(1+\mu)\mathbf{1}$.
		For $-1 \leq \mu \leq 0.55$, these are the only equilibria.
	\end{Th} 
	\begin{proof}
		The first statement follows directly from \eqref{eq:diag-dynamics}. The second statement follows from Corollary \ref{co:equilib-on-diag}.
	\end{proof}
	
	We compute 
	\begin{equation}
	\D_{\mathbf{0}} f_\mu = \begin{bmatrix}
	-\mu & 0 & -1 \\
	-1 & -\mu & 0\\
	0 & -1 & -\mu
	\end{bmatrix}
	\end{equation}
	and
	\begin{equation}
	\D_{(1+\mu)\1} f_\mu = \begin{bmatrix}
	-\mu & 2(1+\mu) & -1 \\
	-1 & -\mu & 2(1+\mu)\\
	2(1+\mu) & -1 & -\mu
	\end{bmatrix}.
	\end{equation}
	A symbolic eigenvalue computation using SymPy shows that
	\begin{equation}\label{eq:spec-0}
	\textnormal{spec}(\D_{\mathbf{0}} f_\mu) = \left\{-\mu + \frac{1}{2} \pm i \frac{\sqrt{3}}{2}, -1 - \mu \right\}
	\end{equation}
	and 
	\begin{equation}\label{eq:spec-1-plus-mu}
	\textnormal{spec}(\D_{(1+\mu)\1} f_\mu) = \left\{-2 \mu - \frac{1}{2} \pm i \frac{\sqrt{3}}{2}(2\mu + 3), 1 + \mu \right\}.
	\end{equation}
	The quantity $(2\mu + 3)$ is nonzero except when $\mu = -\frac{3}{2}$.
	It follows in particular that $\D f_\mu$ evaluated at both of these equilibria is always invertible \textit{except} when $\mu = -1$, which is the value of $\mu$ at which these equilibria coalesce.
	Additionally, the eigenvalues $\pm(1+\mu)$ for the two zeros both correspond to the eigenvector $\1$.
	
	\subsubsection{Two Hopf bifurcations}\label{sec:sprott-hopf}
	Given an equilibrium $\textbf{x}$ for $f_\mu$ at a given value of $\mu$, define the matrix $A\coloneqq \D_{\x} f_\mu$ and the $(1,2)$ tensor $B\coloneqq \D^2_{\x}f_\mu$.
	Since $f_\mu$ is a quadratic vector field, all of its third partial derivatives vanish, and therefore the first Lyapunov coefficient $\ell_1(0)$ at an equilibrium $(\x,\mu)$ having a single pair of purely imaginary eigenvalues is given by \cite[Eq.~5.39]{kuznetsov2013elements}:
	\begin{equation}
	\ell_1(0) = \frac{1}{2 \omega_0} \textnormal{Re} \left[\left \langle p, B(\bar{q}, (2i\omega_0I_n - A)^{-1}B(q,q)) -2B(q,A^{-1}B(q,\bar{q}))\right \rangle \right],
	\end{equation}
	where $\pm i\omega_0$ with $\omega_0 > 0$ are the imaginary eigenvalues of $A$ and $p,q\in \C^n$ satisfy $A q = i\omega_0 q$, $A^T p = -i\omega_0p$, and $\langle p, q\rangle \coloneqq \bar{p} \cdot q = 1$.
	We numerically compute $\ell_1(0) \approx -0.808$ for the equilibrium $\textbf{0}$ at $\mu = 0.5$, and $\ell_1(0) \approx 0.514$ for the equilibrium $(1+\mu)\1$ at $\mu = -0.25$.\footnote{Note that the value of $\ell_1(0)$ depends on the normalization of the eigenvectors $q$ and $p$, but $\textnormal{sign}(\ell_1(0))$ (which is the only thing that matters for the Hopf bifurcation theorem \cite[pp.~97--98]{kuznetsov2013elements}) is invariant under scaling  of $q$, $p$ obeying the condition $\langle p, q \rangle = 1$ \cite[p.~98]{kuznetsov2013elements}.}
	Additionally, we see from \eqref{eq:spec-0} and \eqref{eq:spec-1-plus-mu} that the derivatives with respect to $\mu$ of the real part of the complex eigenvalues is negative for the origin at $\mu = 0.5$ and also negative for $(1+\mu)\1$ at $\mu = -0.25$. 
	From \cite[Thm~3.3]{kuznetsov2013elements} and the final displayed equation of \cite[p.~97]{kuznetsov2013elements}, a subcritical (resp. supercritical) Hopf bifurcation occurs when $\textnormal{sign}(\ell_1(0))$ is the same as (resp. different from) that of the derivative with respect to $\mu$ of the real part of the complex eigenvalues at the critical value of $\mu$.
	Therefore:	
	\begin{Th}\label{th:hopf-sprott}
		The equilibrium $\mathbf{0}$ undergoes a subcritical generic Hopf bifurcation at $\mu = 0.5$, and the equilibrium $(1+\mu)\mathbf{1}$ undergoes a supercritical generic Hopf bifurcation at $\mu = -0.25$.
		The first bifurcation produces an exponentially stable limit cycle near $\mathbf{0}$ for $0 < 0.5-\mu \ll 1$, and the second bifurcation produces an exponentially unstable limit cycle near $(1+\mu)\mathbf{1}$ for $0 < \mu - (-0.25) \ll 1$.
	\end{Th}
	
	\subsubsection{Existence of periodic orbits}\label{sec:sprott-existence}
	
    We now put together the preceding results to obtain a periodic orbit existence result for the Sprott vector field \eqref{eq:sprott}.
    To do this, we show that the restriction $f|_{(-\infty,0.5)}$ satisfies the hypotheses of Theorem \ref{th:hopf} after a (nonlinear) parameter rescaling.
	
   	\begin{Th}\label{th:sprott}
   		Let $f_{\mu}$ be the Sprott vector field \eqref{eq:sprott} and let $K_\mu$ be defined as in Theorem \ref{th:compact-container}.
   		For all $\mu \in (-0.25,0.5)$, $f_{\mu}$ has a periodic orbit contained in $K_\mu$.
   	\end{Th}
  \begin{Rem}\label{rem:classical-difficulty-sprott}
   	For the reasons explained in Remark~\ref{rem:classical-result-difficulties}, it seems very difficult to prove Theorem~\ref{th:sprott} directly using either of the classical continuation results Proposition~\ref{prop:p-global-cont} (\cite[Thm 4.2, Thm 2.2]{mallet1982snakes,alligood1983index}) or Proposition~\ref{prop:global-cont} (\cite[Thm~3.1]{alligood1984families}).
   \end{Rem}

   	\begin{proof}
   		Let $\varphi\colon \R\to (-\infty,0.5)$ be an increasing diffeomorphism satisfying $\varphi|_{(-\infty,0)} = \id_{(-\infty,0)}$ and $\lim_{s\to\infty}\varphi(s)  = 0.5$.
   		Letting $K_\mu$ be as in Theorem \ref{th:compact-container} and defining the family $g\coloneqq \R^3\times \R \to \R^3$ via $g_\mu \coloneqq f_{\varphi(\mu)}$, we will apply Theorem \ref{th:hopf} to show that $g_\mu$ has a periodic orbit contained in $K_{\varphi(\mu)}$ for all $\mu \in (-0.25,\infty)$.

   		Define $\tilde{\cc} \coloneqq \{(\x,\mu)\in \R^3\times \R\colon \mu \geq -1 \text{ and } \x \in K_{\varphi(\mu)} \}$.
   		Using the definition of $K_\mu$, it is easily seen that any set of the form $\tilde{\cc}_{[a,b]}\coloneqq \tilde{\cc}\cap (\R^3\times [a,b])\subset \R^4$ is closed and bounded, hence compact.
   		Letting the closed 1-form $d\theta$ be as defined in \S \ref{sec:dtheta}, Theorem \ref{th:rot-rate-Kmu} implies that there exists $\tilde{\epsilon} > 0$ such that $-d\theta(g_\mu) > \tilde{\epsilon}$ on $\tilde{\cc}\setminus (\Delta\times \R)$.
   		By continuity, there exists $\epsilon > 0$ and a set $\cc\subset \R^3\times \R$ slightly larger than $\tilde{\cc}$ satisfying $\tilde{\cc}\subset \interior(\cc)$, $-d\theta(g_\mu) > \epsilon$ on $\cc\setminus (\Delta\times \R)$, and with each set of the form $\cc_{[a,b]}\cap(\R^3\times [a,b])$ compact.
   		In particular, hypotheses \ref{item:th-hopf-3} and \ref{item:th-hopf-4} of Theorem \ref{th:hopf} are satisfied.\footnote{That $-\frac{d\theta}{2\pi}$ satisfies the relevant Poincar\'{e} duality hypotheses of Theorem \ref{th:hopf} follows exactly as in the proof of Theorem \ref{th:repress}, using the discussion in \S \ref{sec:dtheta} (after flipping the orientations of the submanifolds $M,N$ defined in the proof of Theorem \ref{th:repress} due to the minus sign).}
   		
		Corollary \ref{co:sprott-per-orbits-contained-Kmu} implies that $g_{-1}= f_{-1}$ has no nonstationary periodic orbits, so  hypothesis \ref{item:th-hopf-1} of Theorem \ref{th:hopf} is satisfied with $\mu^*\coloneqq -1$.
		It follows from Theorem \ref{th:compact-container} that every periodic orbit of $g|_{[-1,\infty)}$ is contained in $\tilde{\cc}\subset \interior(\cc)$, so in particular no periodic orbits of $g|_{[-1,\infty)}$ intersect $\partial \cc$; hence  hypothesis \ref{item:th-hopf-2} of Theorem \ref{th:hopf} is satisfied.
		We showed in \S \ref{sec:equilibria} and Theorem \ref{th:hopf-sprott} that $g$ has exactly one generalized center $(\x_c,\mu_c)\coloneqq (-0.75 \cdot \1,-0.25)\in \interior(\cc)$ at which $g$ undergoes a supercritical generic Hopf bifurcation.
   		Hence hypothesis \ref{item:th-hopf-5} of Theorem \ref{th:hopf} is satisfied.
   		Hypothesis \ref{item:th-hopf-6} is satisfied because $\Delta$ is an invariant manifold for each $g_{\mu}$ by symmetry (\S \ref{sec:symmetry}), and $\Delta$ is diffeomorphic to $\R$, so no periodic orbits can intersect $\Delta$.   		
   		Finally, the center subspace $E^c$ of $\D_{\x_c}g_{\mu_c}= \D_{\x_c}f_{\mu_c}$ is orthogonal to $\Delta$ by Equation \eqref{eq:Ec-eq-Delta-perp}, so hypothesis \ref{item:th-hopf-7} of Theorem \ref{th:hopf} is satisfied.
   		
   		Theorem \ref{th:hopf} now implies that $g_\mu$ has a periodic orbit contained in $K_{\varphi(\mu)}$ for all $\mu \in (-0.25,\infty)$.
   		Since $g_\mu = f_{\varphi(\mu)}$ by definition, it follows that $f_\mu$ has a periodic orbit contained in $K_\mu$ for all $\mu \in (-0.25,0.5)$.
   		This completes the proof.	
   	\end{proof}

    \subsection*{Acknowledgements}
    Kvalheim was supported by ARO award W911NF-14-1-0573 and by the ARO under the Multidisciplinary University Research Initiatives (MURI) Program, awards W911NF-17-1-0306 and W911NF-18-1-0327.
    Bloch was supported by NSF grant DMS-1613819 and AFOSR grant FA 0550-18-0028.  
    We would like to thank R. W. Brockett and H. L. Smith for valuable comments during the course of this work and J. Guckenheimer, E. Sander, and J. A. Yorke for useful discussions related to large-period phenomena. We would also like to thank S. Revzen for a suggestion regarding a calculation related to the repressilator and J. C. Sprott for information regarding the undamped version of his eponymous system.
    We thank the anonymous referee for useful suggestions about our exposition.

	\bibliographystyle{amsalpha}
	\bibliography{ref}
	
	\appendix
	\section{Closed $1$-forms}\label{app:closed-1-forms}
	For completeness, in this appendix we recall some standard results concerning closed $1$-forms.
	We follow portions of \cite[p.~35--37]{farber2004topology} nearly verbatim and refer the reader to \cite{deRham1984differentiable, bott1982differential, guillemin1974differential, lee2013smooth} in places for other details.
	The reader completely unfamiliar with differential forms may wish to consult \cite[Sec.~2.5,~2.7]{bloch2015nonholonomic} for a quick introduction to the basic definitions (of, e.g., $\wedge$, $d$, $\int$) with more details than we provide. 
	
	Let $M$ be a smooth manifold.
	A $C^{k\geq 0}$ \concept{differential $1$-form} (or simply \concept{$1$-form}) $\cf$ on $M$ is a $C^k$ section $M\to \T^* M$ of the cotangent bundle $\T^* M \to M$ \cite[Ch.~11]{lee2013smooth}; we will write $\cf_x$ instead of $\cf(x)$.
	In particular, given $x\in M$, the real-valued map $v\in \T_x M\mapsto \cf_x(v)\in \R$ is linear. 
	(Here $\T M\to M$ is the tangent bundle and $\T_x M$ is the tangent space to $M$ at $x$.)
	We will often simply write $\cf(v)$ instead of $\cf_x(v)$ for $v\in \T_x M$.
	Given a vector field $f$ on $M$, there is a map $x\in M \mapsto \cf(f(x))\in \R$ which we simply denote by $\cf(f)$.  
	In general, the context should make clear the precise meaning of any instance of $\cf(\slot)$.
	 
	In local coordinates $x_1,x_2,\ldots, x_n$ defined in an open subset $U\subset M$, any $C^k$ $1$-form $\cf$ is given by an expression of the form $\cf_x = a_1(x)dx_1 + a_2(x) dx_2 + \cdots + a_n(x) dx_n$, where $a_1(x), \ldots, a_n(x)$ are $C^k$ real-valued functions defined in $U$ and $dx_i(v) = v_i$ for a tangent vector $v = \sum_j v_j \frac{\partial}{\partial x_j}$.
	A $C^1$ $1$-form $\cf$ is called \concept{closed} if $d\cf = 0$, where $d$ is the \concept{exterior derivative} \cite[p.~365]{lee2013smooth}.
	In local coordinates, if $\cf  = \sum_{i=1}^n a_i dx_i$, then 
	$$d\cf = \sum_{i=1}^n da_i\wedge dx_i = \sum_{i,j=1}^n \frac{\partial a_i}{\partial x_j}dx_j \wedge dx_i = \sum_{i < j} \left(\frac{\partial a_j}{\partial x_i} - \frac{\partial a_i}{\partial x_j} \right) dx_i \wedge dx_j $$
	where $\wedge$ is the \concept{wedge product} \cite[p.~360]{lee2013smooth}.
	Hence the condition that $\cf$ is closed is equivalent to the equations
	$$\frac{\partial a_j}{\partial x_i} = \frac{\partial a_i}{\partial x_j}, \quad \textnormal{for all } i,j.$$
	A $C^k$ \concept{exact} $1$-form $\cf$ is one which can be represented globally as the differential $df$ of a $C^{k+1}$ function $f\colon M\to \R$. 
	If $\cf$ is a $C^1$ closed $1$-form, then for any simply connected open set $U\subset M$ the restriction $\cf|_U$ is an exact $1$-form.\footnote{This follows from (i) the Poincar\'{e} Lemma \cite[Thm~17.14]{lee2013smooth}, (ii) the De Rham Theorem \cite[Thm~18.14]{lee2013smooth}, and (iii) the fact that every $C^k$ closed $1$-form can be represented as the sum of a $C^\infty$ closed $1$-form with a $C^k$ exact $1$-form  \cite[pp.~61--70]{deRham1984differentiable}.} 
	If $U$ is connected, then the function $f_U$ is determined by $\cf|_U$ uniquely up to the addition of a constant.
	Thus, viewed locally, a closed $1$-form is the same thing as a real-valued function determined up to the addition of a constant.
	
	Given smooth manifolds $M_1, M_2$, a $C^1$ map $F\colon M_1\to M_2$, and a $1$-form $\cf$ on $M_2$, the \concept{pullback} $F^*\cf$ is the $1$-form on $M_1$ defined by the rule $F^*\cf(v) = \cf(\D_x F v)$ for $x\in M_1$ and $v\in \T_x M_1$, where $\D_x F\colon \T_x M_1 \to \T_{F(x)}M_2$ is the \concept{derivative} or \concept{tangent map} of $F$ at $x$ \cite[p.~360]{lee2013smooth}.  
	
	Given a $C^1$ $1$-form $\cf$ and piecewise-$C^1$ path $\gamma\colon [a,b]\to M$, the line integral $\int_{\gamma}\cf$ is well-defined.
	By Stokes's Theorem \cite[Thm~16.25]{lee2013smooth}, the condition that $\cf$ is closed is equivalent to the property that the integral $\int_\gamma \cf$ remains unchanged under any continuous homotopy of the path $\gamma$ with fixed end points.  
	
	The statement of Theorem~\ref{th:hopf} involves special cases of \concept{de Rham cohomology} and \concept{Poincar\'{e} duality}. 
	The \concept{first de Rham cohomology} of $M$ \cite[Ch.~17]{lee2013smooth} is the real quotient vector space $$\Hdr^1(M)\coloneqq \{\textnormal{$C^\infty$ closed $1$-forms}\}/\{\textnormal{$C^\infty$ exact  $1$-forms}\}.$$  
	The representative of a $C^\infty$ closed $1$-form $\cf$ in $\Hdr^1(M)$ is written $[\cf]\in \Hdr^1(M)$ and  is called the \concept{cohomology class} of $\cf$.
	Two $C^1$ closed $1$-forms $\cf_1$, $\cf_2$ are \concept{cohomologous} if $\cf_1 - \cf_2 = df$ for some $C^2$ function $f\colon M\to \R$. 
	Given any $C^{k\geq 1}$ closed $1$-form $\cf_1$, there exists a $C^\infty$ closed $1$-form $\cf_2$ which is cohomologous to $\cf_1$ \cite[pp.~61--70]{deRham1984differentiable}, so we may also define the cohomology class $[\cf_1]\in \Hdr^1(M)$ of $\cf_1$ to be the cohomology class $[\cf_2]\in \Hdr^1(M)$.
	
	Assume now that $M$ is oriented.
	Associated to any properly embedded, smooth, oriented, codimension-1 submanifold $N\subset M$, there is a cohomology class $[\cf]\in \Hdr^1(M)$ called the \concept{(closed) Poincar\'{e} dual} of $N$ \cite[pp.~50--53]{bott1982differential} satisfying the following property \cite[p.~69]{bott1982differential}: for any $\cf\in [\cf]$ and any $C^1$ embedding $\gamma\colon S^1\to M$ from the circle $S^1$ into $M$ which is transverse to $N$ (i.e., $\dot{\gamma}(t)\not \in \T_{\gamma(t)}N$ for all $t\in S^1$), $$\int_{\gamma}\cf = I(N,\gamma)\coloneqq \sum_{t\in \gamma^{-1}(N)} \epsilon_t. $$
	Here $\epsilon_t = +1$ (resp. $\epsilon_t = -1$) if a positively oriented basis of $\T_{\gamma(t)}N$ followed by the tangent vector $\dot{\gamma}(t)$ yields a positively (resp. negatively) oriented basis of $\T_{\gamma(t)}M$.
	$I(N,\gamma)$ is called the \concept{oriented intersection number} of $N$ with $\gamma$ \cite[p.~107]{guillemin1974differential}.
		
\end{document}